\documentclass[a4paper,10pt,reqno]{amsart}
\usepackage[latin1]{inputenc}
\usepackage{dsfont}
\usepackage{amsxtra}
\usepackage{amsmath}
\usepackage{amssymb}
\usepackage{amsfonts}
\usepackage[all]{xy}
\usepackage{mathrsfs}
\usepackage{amsthm}

\theoremstyle{plain}
\newtheorem{thm}{Theorem}[section]
\newtheorem{lemma}[thm]{Lemma}
\newtheorem{prop}[thm]{Proposition}
\newtheorem{cor}[thm]{Corollary}

\theoremstyle{definition}
\newtheorem{df}[thm]{Definition}

\theoremstyle{remark}
\newtheorem{rk}[thm]{Remark}

\numberwithin{equation}{section}

\def\C{\mathbb{C}}
\def\Q{\mathbb{Q}}
\def\Z{\mathbb{Z}}
\def\N{\mathbb{N}}

\def\bbD{\mathbb{D}}

\def\scrA{\mathscr{A}}
\def\scrB{\mathscr{B}}

\def\scrD{\mathscr{D}}

\def\scrF{\mathscr{F}}

\def\scrH{\mathscr{H}}
\def\scrI{\mathscr{I}}

\def\scrL{\mathscr{L}}
\def\scrM{\mathscr{M}}
\def\scrN{\mathscr{N}}
\def\scrO{\mathscr{O}}

\def\scrV{\mathscr{V}}

\def\frakC{\mathfrak{C}}

\def\frakE{\mathfrak{E}}

\def\frakS{\mathfrak{S}}

\def\frakX{\mathfrak{X}}

\def\fraka{\mathfrak{a}}
\def\frakb{\mathfrak{b}}

\def\g{\mathfrak{g}}
\def\h{\mathfrak{h}}

\def\frakl{\mathfrak{l}}
\def\frakm{\mathfrak{m}}
\def\frakn{\mathfrak{n}}

\def\frakp{\mathfrak{p}}
\def\frakq{\mathfrak{q}}

\def\frakt{\mathfrak{t}}

\def\fraksl{\mathfrak{sl}}

\def\calA{\mathcal{A}}
\def\calB{\mathcal{B}}
\def\calC{\mathcal{C}}
\def\calD{\mathcal{D}}
\def\calE{\mathcal{E}}
\def\calF{\mathcal{F}}

\def\calO{\mathcal{O}}
\def\calP{\mathcal{P}}

\def\calU{\mathcal{U}}

\def\calX{\mathcal{X}}

\def\bfone{\mathbf{1}}

\def\bfD{\mathbf{D}}

\def\bfH{\mathbf{H}}

\def\bfM{\mathbf{M}}

\def\bfO{\mathbf{O}}

\def\bfS{\mathbf{S}}

\def\bfU{\mathbf{U}}

\def\bfc{\mathbf{c}}

\def\bfk{\mathbf{k}}

\def\bfv{\mathbf{v}}

\newcommand{\al}{\alpha}
\newcommand{\ep}{\epsilon}
\newcommand{\vep}{\varepsilon}

\def\lam{\lambda}
\def\Lam{\Lambda}

\def\geqs{\geqslant}
\def\leqs{\leqslant}

\def\ra{\rightarrow}
\def\lra{\longrightarrow}

\def\simra{\overset{\sim}\ra}

\newcommand{\pair}[1]{\langle{#1}\rangle}
\newcommand{\wdg}[1]{|{#1}\rangle}
\def\ilim{\varinjlim}
\def\plim{\varprojlim}
\def\sss{\scriptscriptstyle}
\def\iff{\Longleftrightarrow}

\def\dim{{\mathrm{dim}}}

\def\dch{\mathbf{ch}}

\newcommand{\QS}{{}^{\scriptstyle{Q}}\negmedspace\frakS}
\newcommand{\PS}{{}^{\scriptstyle{P}}\negmedspace\frakS}
\newcommand{\tilim}{2\kern-.15em\ilim}
\newcommand{\tplim}{2\kern-.15em\plim}
\newcommand{\olX}{\overline{X}}
\newcommand{\olY}{\overline{Y}}

\def\homo{\operatorname{\it \mathscr{H}\kern-.25em om}}
\newcommand{\diff}{\operatorname{\it \mathscr{D}\kern-.15em iff}}
\def\ext{\operatorname{\it \mathscr{E}\kern-.25em xt}}
\def\edo{\operatorname{\it \mathscr{E}\kern-.25em nd}}
\def\der{\operatorname{\it \mathscr{D}\kern-.25em er}}
\def\Perv{\mathbf{Perv}}
\def\MHM{\mathbf{M}\mathbf{H}\mathbf{M}}
\def\Hom{\operatorname{Hom}\nolimits}
\def\End{\operatorname{End}\nolimits}

\def\Ext{\operatorname{Ext}\nolimits}

\def\Lie{\operatorname{Lie}\nolimits}
\def\Coker{\operatorname{Coker}\nolimits}
\def\Ker{\operatorname{Ker}\nolimits}
\def\modu{\operatorname{-mod}\nolimits}
\def\proj{\operatorname{-proj}\nolimits}

\def\ch{\operatorname{ch}\nolimits}

\def\DR{\operatorname{DR}\nolimits}
\def\Im{\operatorname{Im}\nolimits}
\def\hight{\operatorname{ht}\nolimits}
\def\pt{\operatorname{pt}\nolimits}
\def\pr{\operatorname{pr}\nolimits}
\def\Id{\operatorname{Id}\nolimits}
\def\Ad{\operatorname{Ad}\nolimits}
\def\tr{\operatorname{tr}\nolimits}

\def\op{\operatorname{op}\nolimits}

\def\Ext{\operatorname{Ext}\nolimits}

\def\gr{\operatorname{gr}\nolimits}

\def\Sh{\mathbf{Sh}}

\author{Peng Shan}
\address{D\'epartement de Math\'ematiques,
Universit\'e Paris 7, 175 rue du Chevaleret, F-75013 Paris.} \email{shan@math.jussieu.fr}
\title[Jantzen filtration]{Graded decomposition matrices of $v$-Schur algebras via Jantzen filtration}

\begin{document}
\begin{abstract}
We prove that certain parabolic Kazhdan-Lusztig polynomials
calculate the graded decomposition matrices of $v$-Schur algebras
given by the Jantzen filtration of Weyl modules, confirming a
conjecture of Leclerc and Thibon.
\end{abstract}

\maketitle
\tableofcontents

\section*{Introduction}

Let $v$ be a $r$-th root of unity in $\C$. The $v$-Schur algebra
$\bfS_v(n)$ over $\C$ is a finite dimensional quasi-hereditary
algebra. Its standard modules are the Weyl modules $W_v(\lam)$
indexed by partitions $\lam$ of $n$. The module $W_v(\lam)$ has a
simple quotient $L_v(\lam)$. See Section \ref{ss:schur} for more details.

The decomposition matrix of $\bfS_v(n)$ is given by the following
algorithm. Let $\calF_q$ be the Fock space of level one. It is a
$\Q(q)$-vector space with a basis $\{\wdg{\lam}\}$ indexed by the
set of partitions. Moreover it carries an action of the quantum enveloping
algebra $U_q(\widehat{\fraksl}_r)$. Let $L^+$ (resp. $L^-$) be the
$\Z[q]$-submodule (resp. $\Z[q^{-1}]$-submodule) in $\calF_q$
spanned by $\{\wdg{\lam}\}.$ Following Leclerc and Thibon \cite[Theorem 4.1]{LT1}, the Fock space $\calF_q$ admits two
particular bases $\{G^+_\lam\}$, $\{G^-_\lam\}$ with the properties
that
\begin{equation*}
G^+_\lam\equiv\wdg{\lam}\text{ mod }qL^+,\qquad
G^-_\lam\equiv\wdg{\lam}\text{ mod }q^{-1}L^-.
\end{equation*}
Let $d_{\lam\mu}(q), e_{\lam\mu}(q)$ be
elements in $\Z[q]$ such that
\begin{equation*}
G^+_\mu=\sum_{\lam}d_{\lam\mu}(q)\wdg{\lam},\qquad
G^-_{\lam}=\sum_{\mu}e_{\lam\mu}(-q^{-1})\wdg{\mu}.
\end{equation*}
For any partition $\lam$ write $\lam'$ for the transposed partition.
Then the Jordan-H\"{o}lder multiplicity of $L_v(\mu)$ in $W_v(\lam)$
is equal to the value of $d_{\lam'\mu'}(q)$ at $q=1$. This result
was conjectured by Leclerc and Thibon \cite[Conjecture 5.2]{LT1} and
has been proved by Varagnolo and Vasserot \cite{VV1}.

We are interested in the Jantzen filtration of $W_v(\lam)$ \cite{JM}
$$W_v(\lam)=J^0W_v(\lam)\supset J^1W_v(\lam)\supset\ldots.$$
It is a filtration by $\bfS_v(n)$-submodules. The graded decomposition matrix of $\bfS_v(n)$ counts the multiplicities of $L_v(\mu)$ in the associated graded module of $W_v(\lam)$. The graded version of the above algorithm was also conjectured by Leclerc and Thibon \cite[Conjecture 5.3]{LT1}. Our main result is a proof of this conjecture under a mild restriction on $v$.

\begin{thm}\label{mthm}
Suppose that $v=\exp(2\pi i/\kappa)$ with $\kappa\in\Z$ and $\kappa\leqs -3$. Let
$\lam$, $\mu$ be partitions of $n$. Then
\begin{equation}\label{eq:main}
d_{\lam'\mu'}(q)=\sum_{i\geqs 0}[J^iW_v(\lam)/J^{i+1}W_v(\lam):
L_v(\mu)]q^i.
\end{equation}
\end{thm}

Let us outline the idea of the proof. We first show that certain equivalence of highest weight categories preserves the Jantzen filtrations of standard modules (Proposition \ref{prop:jantiso}). By constructing such an equivalence between the module category of the $v$-Schur algebra and a subcategory of the affine parabolic category $\calO$ of negative level, we then transfer the problem of computing the Jantzen filtration of Weyl modules into the same problem for
parabolic Verma modules (Corollary \ref{cor:jantid}). The latter is
solved using Beilinson-Bernstein's technics (Sections
\ref{s:localization}, \ref{s:geojant}, \ref{s:proof}).

\subsection*{Acknowledgement} I deeply thank my advisor Eric Vasserot for many helpful discussions, lots of useful remarks and for his substantial support.

\vskip1cm

\section{Jantzen filtration of standard
modules}\label{s:highestweight}

\subsection{Notation}\label{ss:notation1}

We will denote by $A\modu$ the category of finitely generated modules over an algebra $A$, and by $A\proj$ its subcategory consisting of projective
objects. Let $R$ be a commutative noetherian
$\C$-algebra. By a \emph{finite projective $R$-algebra} we mean a
$R$-algebra that belongs to $R\proj$.

A $R$-category $\calC$ is a category whose $\Hom$ sets are
$R$-modules. All the functors between $R$-categories will be assumed
to be $R$-linear, i.e., they induce morphisms of $R$-modules on the
$\Hom$ sets. Unless otherwise specified, all the functors will be assumed to be covariant. If $\calC$ is abelian, we will write $\calC\proj$ for
the full subcategory consisting of projective objects. If there
exists a finite projective algebra $A$ together with an equivalence
of $R$-categories $F:\calC\cong A\modu$, then we define $\calC\cap
R\proj$ to be the full subcategory of $\calC$ consisting of objects
$M$ such that $F(M)$ belongs to $R\proj$. By Morita theory, the
definition of $\calC\cap R\proj$ is independent of $A$ or $F$.
Further, for any $\C$-algebra homomorphism $R\ra R'$ we will abbreviate $R'\calC$ for the category $(R'\otimes_R A)\modu$. The
definition of $R'\calC$ is independent of the choice of $A$ up to
equivalence of categories.

For any abelian category $\calC$ we will write $[\calC]$ for the Grothendieck group of $\calC$. Any exact functor $F$ from $\calC$ to
another abelian category $\calC'$ yields a group homomorphism
$[\calC]\ra[\calC']$, which we will again denote by $F$.

A $\C$-category $\calC$ is called \emph{artinian} if the $\Hom$ sets
are finite dimensional $\C$-vector spaces and every object has a
finite length. The Jordan-H\"{o}lder multiplicity of a simple object
$L$ in an object $M$ of $\calC$ will be denoted by $[M:L]$.

We abbreviate $\otimes=\otimes_\C$ and $\Hom=\Hom_\C$.

\subsection{Highest weight categories}\label{ss:highestweightcategory}

Let $\calC$ be a $R$-category that is equivalent to the category
$A\modu$ for some finite projective $R$-algebra $A$. Let $\Delta$ be
a finite set of objects of $\calC$ together with a partial order
$<$. Let $\calC^\Delta$ be the full subcategory of $\calC$
consisting of objects which admit a finite filtration such that the
successive quotients are isomorphic to objects in
$$\{D\otimes U\,|\,D\in\Delta,\ U\in R\proj\}.$$
We have the following definition, see \cite[Definition 4.11]{R}.
\begin{df}
The pair $(\calC,\Delta)$ is called a \emph{highest weight $R$-category} if the following conditions hold:
\begin{itemize}
\item[$\bullet$] the objects of $\Delta$ are projective over $R$,

\item[$\bullet$] we have $\End_{\calC}(D)=R$ for all $D\in\Delta$,

\item[$\bullet$] given $D_1$, $D_2\in\Delta$, if $\Hom_\calC(D_1, D_2)\neq 0$,
then $D_1\leqs D_2$,

\item[$\bullet$] if $M\in\calC$ satisfies $\Hom_\calC(D,M)=0$ for all $D\in\Delta$,
then $M=0$,

\item[$\bullet$] given $D\in\Delta$, there exists $P\in\calC\proj$ and a surjective morphism $f: P\ra D$
 such that $\ker f$ belongs to $\calC^\Delta$. Moreover, in the
 filtration of $\ker f$ only $D'\otimes U$ with $D'>D$ appears.
 \end{itemize}
\end{df}

The objects in $\Delta$ are called \emph{standard}. We say that an
object has \emph{a standard filtration} if it belongs to $\calC^\Delta$.
There is another set $\nabla$ of objets in $\calC$, called
\emph{costandard} objects, given by the following proposition.

\begin{prop}\label{prop:costandard}
Let $(\calC,\Delta)$ be a highest weight $R$-category. Then there is
a set $\nabla=\{D\spcheck\,|\, D\in\Delta\}$ of objects of $\calC$,
unique up to isomorphism, with the following properties:
\begin{itemize}
\item[(a)] the pair $(\calC^{\op},\nabla)$ is a highest
weight $R$-category, where $\nabla$ is equipped with the same
partial order as $\Delta$,
\item[(b)] for $D_1$, $D_2\in\Delta$ we have
$\Ext^i_\calC(D_1,D_2\spcheck)\cong
\begin{cases} R &
\text{ if }i=0\text{ and } D_1=D_2\\ 0 &\text{ else. }
\end{cases}$
\end{itemize}
\end{prop}
See \cite[Proposition 4.19]{R}.

\subsection{Base change for highest weight categories.}\label{ss:s}

From now on, unless otherwise specified we will fix $R=\C[[s]]$, the
ring of formal power series in the variable $s$. Let $\wp$ be its
maximal ideal and let $K$ be its fraction field. For any $R$-module
$M$, any morphism $f$ of $R$-modules and any $i\in\N$ we will write
\begin{eqnarray*}
  &M(\wp^i)=M\otimes_R(R/\wp^i R),\quad &M_K=M\otimes_RK,\\
  &f(\wp^i)=f\otimes_R(R/\wp^i R),\quad &f_K=f\otimes_RK.
\end{eqnarray*}
We will abbreviate
$$\calC(\wp)=R(\wp)\calC,\quad\calC_K=K\calC.$$
Let us first recall the following basic facts.
\begin{lemma}\label{lem:useful} Let $A$ be a finite projective
$R$-algebra. Let $P\in A\modu$.

(a) The $A$-module $P$ is projective if and only if $P$ is a
projective $R$-module and $P(\wp)$ belongs to $A(\wp)\proj$.

(b) If $P$ belongs to $A\proj$, then we have a canonical
isomorphism
$$\Hom_{A}(P,M)(\wp)\simra\Hom_{A(\wp)}(P(\wp),M(\wp)),\quad\forall\ M\in A\modu.$$
Further, if $M$ belongs to $R\proj$ then $\Hom_A(P,M)$
also belongs to $R\proj$.
\end{lemma}

We will also need the following theorem of Rouquier \cite[Theorem 4.15]{R}.

\begin{prop}\label{prop:specialization}
Let $\calC$ be a $R$-category that is equivalent to $A\modu$ for
some finite projective $R$-algebra $A$. Let $\Delta$ be a finite
poset of objects of $\calC\cap R\proj$. Then the category
$(\calC,\Delta)$ is a highest weight $R$-category if and only if
$(\calC(\wp),\Delta(\wp))$ is a highest weight $\C$-category.
\end{prop}

Finally, the costandard objects can also be characterized in the following way.

\begin{lemma}\label{lem:costandsp}
Let $(\calC,\Delta)$ be a highest weight $R$-category. Assume that
$\nabla'=\{{}\spcheck D\,|\,D\in\Delta\}$ is a set of objects of
$\calC\cap R\proj$ such that for any $D\in\Delta$ we have
$$({}\spcheck D) (\wp)\cong
D(\wp)\spcheck,\quad({}\spcheck D)_K\cong D_K.$$ Then we have
${}\spcheck D\cong D\spcheck\in\nabla.$
\end{lemma}
\begin{proof}
We prove the lemma by showing that $\nabla'$ has the properties
(a), (b) in Proposition \ref{prop:costandard} with ${}\spcheck D$
playing the role of $D\spcheck$. This will imply that ${}\spcheck
D\cong D\spcheck\in\nabla.$ To check (a) note that $\nabla'(\wp)$
is the set of costandard modules of $\calC(\wp)$ by assumption. So $(\calC(\wp)^{\op},\nabla'(\wp))$ is a highest weight $\C$-category.
Therefore $(\calC^{\op},\nabla')$ is a highest weight $R$-category by Proposition \ref{prop:specialization}. Now, let us concentrate on
(b). Given $D_1$, $D_2\in\Delta$, let $P_\bullet=0\ra
P_n\ra\cdots\ra P_0$ be a projective resolution of $D_1$ in $\calC$.
Then $\Ext^i_\calC(D_1,{}\spcheck D_2)$ is the cohomology of the
complex
$$C_\bullet=\Hom_{\calC}(P_\bullet,{}\spcheck D_2).$$
Since $D_1$ and all the $P_i$ belong to $R\proj$ and $R$ is a discrete valuation ring, by the Universal Coefficient Theorem the complex
$$P_\bullet(\wp)=0\ra P_n(\wp)\ra\cdots\ra P_0(\wp)$$ is a
resolution of $D_1(\wp)$ in $\calC(\wp)$. Further, each $P_i(\wp)$ is a
projective object in $\calC(\wp)$ by Lemma
\ref{lem:useful}(a). So $\Ext^i_{\calC(\wp)}(D_1(\wp),{}\spcheck
D_2(\wp))$ is given by the cohomology of the complex
$$C_\bullet(\wp)=\Hom_{\calC(\wp)}(P_\bullet(\wp),{}\spcheck D_2(\wp)).$$
Again, by the Universal Coefficient
Theorem, the canonical map
$$H_i(C_\bullet)(\wp)\lra
H_i(C(\wp)_\bullet)$$ is injective. In other words we have a
canonical injective map
\begin{equation}\label{eq:uct}
\Ext^i_\calC(D_1,{}\spcheck
D_2)(\wp)\lra\Ext^i_{\calC(\wp)}(D_1(\wp),{}\spcheck D_2(\wp)).
\end{equation}
Note that each $R$-module $C_i$ is finitely generated. Therefore
$\Ext^i_\calC(D_1,{}\spcheck D_2)$ is also finitely generated over
$R$. Note that if $i>0$, or $i=0$ and $D_1\neq D_2$, then the right hand
side of (\ref{eq:uct}) is zero by assumption. So
$\Ext^i_\calC(D_1,{}\spcheck D_2)(\wp)=0$, and hence
$\Ext^i_\calC(D_1,{}\spcheck D_2)=0$ by Nakayama's lemma. Now, let
us concentrate on the $R$-module $\Hom_\calC(D,{}\spcheck D)$ for
$D\in\Delta$. First, we have
\begin{eqnarray}
  \Hom_\calC(D,{}\spcheck D)\otimes_RK&=&
  \Hom_{\calC_K}(D_K,({}\spcheck D)_K)\nonumber\\
  &=&\End_{\calC_K}(D_K)\nonumber\\
  &=&\End_{\calC}(D)\otimes_RK\nonumber\\
  &=&K.\label{eq:overK}
\end{eqnarray}
Here the second equality is given by the isomorphism
$D_K\cong({}\spcheck D)_K$ and the last equality follows from
$\End_{\calC}(D)=R$. Next, note that $\Hom_\calC(D,{}\spcheck
D)(\wp)$ is included into the vector space
$\Hom_{\calC(\wp)}(D(\wp),{}\spcheck D(\wp))=\C$ by (\ref{eq:uct}).
So its dimension over $\C$ is less than one. Together with
(\ref{eq:overK}) this yields an isomorphism of $R$-modules
$\Hom_\calC(D,{}\spcheck D)\cong R$, because $R$ is a discrete
valuation ring. So we have verified that $\nabla'$ satisfies both property (a) and (b) in Proposition \ref{prop:costandard}. Therefore it coincides with $\nabla$ and ${}\spcheck D$ is isomorphic to $D\spcheck$.
\end{proof}

\subsection{The Jantzen filtration of standard modules}\label{ss:jantzen}

Let $(\calC_\C,\Delta_\C)$ be a highest weight $\C$-category and
$(\calC,\Delta)$ be a highest weight $R$-category such that
$(\calC_\C,\Delta_\C)\cong(\calC(\wp),\Delta(\wp))$. Then any
standard module in $\Delta_\C$ admits a Jantzen type filtration
associated with $(\calC,\Delta)$. It is given as follows.

\begin{df}\label{df:jantzen}
For any $D\in\Delta$ let $\phi:D\ra D\spcheck$ be a morphism in
$\calC$ such that $\phi(\wp)\neq 0$. For any positive integer $i$
let
\begin{equation}\label{eq:pii}
\pi_i:D\spcheck\lra D\spcheck/\wp^iD\spcheck
\end{equation}
be the canonical quotient map. Set
\begin{equation*}
D^i=\ker(\pi_i\circ \phi)\subset D,\quad J^iD(\wp)=(D^i+\wp D)/\wp D.
\end{equation*}
The \emph{Jantzen filtration} of $D(\wp)$ is the filtration
\begin{equation*}
D(\wp)=J^0D(\wp)\supset J^1D(\wp)\supset\cdots.
\end{equation*}
\end{df}

To see that the Jantzen filtration is well defined, one notices
first that the morphism $\phi$ always exists because
$\Hom_{\calC}(D,D\spcheck)(\wp)\cong R(\wp)$. Further, the
filtration is independent of the choice of $\phi$. Because if
$\phi':D\ra D\spcheck$ is another morphism such that $\phi'(\wp)\neq
0$, the fact that $\Hom_{\calC}(D,D\spcheck)\cong R$ and
$\phi(\wp)\neq 0$ implies that there exists an element $a$ in $R$
such that $\phi'=a\phi$. Moreover $\phi'(\wp)\neq 0$ implies that
$a$ is invertible in $R$. So $\phi$ and $\phi'$ define the same
filtration.

\begin{rk}\label{rk:jantzen}
If the category $\calC_K$ is semi-simple, then the Jantzen
filtration of any standard module $D(\wp)$ is finite. In fact, since
$\End_{\calC}(D)=R$ we have $\End_{\calC_K}(D_K)=K$. Therefore $D_K$
is an indecomposable object in $\calC_K$. So the semi-simplicity of $\calC_K$ implies that $D_K$ is simple. Similarly
$D\spcheck_K$ is also simple. So the morphism $\phi_K:D_K\ra
D\spcheck_K$ is an isomorphism. In particular $\phi$ is injective.
Now, consider the intersection
$$\bigcap_iJ^iD(\wp)=\bigcap_i(D^i+\wp D)/\wp D.$$
Since we have $D^i\supset D^{i+1}$, the intersection on the right
hand side is equal to $((\bigcap_iD^i)+\wp D)/\wp D$. The
injectivity of $\phi$ implies that $\bigcap_iD^i=\ker\phi$ is zero.
Hence $\bigcap_iJ^iD(\wp)=0$. Since $D(\wp)\in\calC(\wp)$ has a
finite length, we have $J^iD(\wp)=0$ for $i$ large enough.

\end{rk}

\subsection{Equivalences of highest weight categories and Jantzen filtrations.}\label{ss:equivjantzen}

Let $(\calC_1,\Delta_1)$, $(\calC_2,\Delta_2)$ be highest weight
$R$-categories (resp. $\C$-categories or $K$-categories). A functor
$F:\calC_1\ra\calC_2$ is an \emph{equivalence of highest weight
categories} if it is an equivalence of categories and if for any
$D_1\in\Delta_1$ there exists $D_2\in\Delta_2$ such that
$F(D_1)\cong D_2$. Note that for such an equivalence $F$ we also
have
\begin{equation}\label{eq:Fdual}
F(D_1\spcheck)\cong D_2\spcheck,
\end{equation}
because the two properties in Proposition \ref{prop:costandard}
which characterize the costandard objects are preserved by $F$.

Let $F:\calC_1\ra\calC_2$ be an exact functor. Since $\calC_1$ is
equivalent to $A\modu$ for some finite projective $R$-algebra $A$,
the functor $F$ is represented by a projective object $P$ in
$\calC_1$, i.e., we have $F\cong\Hom_{\calC_1}(P,-)$. Set
$$F(\wp)=\Hom_{\calC_1(\wp)}(P(\wp),-):
\calC_1(\wp)\ra\calC_2(\wp).$$ Note that the functor $F(\wp)$ is unique up to equivalence of categories. It is an exact functor, and it is isomorphic to the functor
$\Hom_{\calC_1}(P,-)(\wp)$, see Lemma \ref{lem:useful}. In particular, for $D\in\Delta_1$ there
are canonical isomorphisms
\begin{equation}\label{eq:Fstand}
F(D)(\wp)\cong F(\wp)(D(\wp)),\quad F(D\spcheck)(\wp)\cong
F(\wp)(D\spcheck(\wp)).
\end{equation}

\begin{prop}\label{prop:jantiso}
Let $(\calC_1,\Delta_1)$, $(\calC_2,\Delta_2)$ be two equivalent
highest weight $R$-categories. Fix an equivalence $F: \calC_1\ra
\calC_2$. Then the following holds.

(a) The functor $F(\wp)$ is an equivalence of highest weight
categories.

(b) The functor $F(\wp)$ preserves the Jantzen filtration of
standard modules, i.e., for any $D_1\in\Delta_1$ let
$D_2=F(D_1)\in\Delta_2$, then
\begin{equation*}
F(\wp)(J^iD_1(\wp))=J^iD_2(\wp),\quad \forall\ i\in\N.
\end{equation*}
\end{prop}
\begin{proof}
(a) If $G: \calC_2\ra \calC_1$ is a quasi-inverse of $F$ then
$G(\wp)$ is a quasi-inverse of $F(\wp)$. So $F(\wp)$ is an
equivalence of categories. It maps a standard object to a standard
one because of the first isomorphism in (\ref{eq:Fstand}).

(b) The functor $F$ yields an isomorphism of $R$-modules
$$\Hom_{\calC_1}(D_1,D_1\spcheck)\simra\Hom_{\calC_2}(F(D_1),F(D_1\spcheck)),$$
where the right hand side identifies with
$\Hom_{\calC_2}(D_2,D_2\spcheck)$ via the isomorphism
(\ref{eq:Fdual}). Let $\phi_1$ be an element in
$\Hom_{\calC_1}(D_1,D_1\spcheck)$ such that $\phi_1(\wp)\neq 0$. Let $$\phi_2=F(\phi_1):D_2\ra D_2\spcheck.$$ Then we also
have $\phi_2(\wp)\neq 0$.

For $a=1,2$ and $i\in\N$ let $\pi_{a,i}:D_a\spcheck\ra
D_a\spcheck(\wp^i)$ be the canonical quotient map. Since $F$ is
$R$-linear and exact, the isomorphism $F(D_1\spcheck)\cong
D_2\spcheck$ maps $F(\wp^iD_1\spcheck)$ to $\wp^iD_2\spcheck$ and
induces an isomorphism
$$F(D_1\spcheck(\wp^i))\cong D_2\spcheck(\wp^i).$$
Under these isomorphisms the morphism $F(\pi_{1,i})$ is identified
with $\pi_{2,i}$. So we have
\begin{eqnarray*}
F(D_1^i)&=& F(\ker(\pi_{1,i}\circ \phi_1))\\
&=&\ker(F(\pi_{1,i})\circ F(\phi_1))\\
&\cong&\ker(\pi_{2,i}\circ \phi_2)\\
&=&D_2^i.
\end{eqnarray*}
Now, apply $F$ to the short exact sequence
\begin{equation}
0\ra \wp D_1\ra D_1^i+\wp D_1\ra J^iD_1(\wp)\ra 0,
\end{equation}
we get
\begin{eqnarray*}
F(J^iD_1(\wp))&\cong& (F(D_1^i)+\wp F(D_1))/\wp F(D_1)\\
&\cong& J^iD_2(\wp).
\end{eqnarray*}
Since $F(J^iD_1(\wp))=F(\wp)(J^iD_1(\wp))$, the proposition is proved.
\end{proof}
\iffalse\begin{rk}
  Recall that the definition of $F(\wp)$ requires
  a choice of a projective generator $P\in\calC_1$. The validity of the proposition does not depend on this
  choice.
\end{rk}\fi

\vskip1cm

\section{Affine parabolic category $\calO$ and $v$-Schur algebras}\label{s:affineO}

\subsection{The affine Lie algebra}\label{ss:notation2}

Fix an integer $m>1$. Let $G_0\supset B_0\supset T_0$ be respectively the linear algebraic group $GL_m(\C)$, the Borel subgroup
of upper triangular matrices and the maximal torus of
diagonal matrices. Let $\g_0\supset\frakb_0\supset\frakt_0$ be their Lie
algebras. Let
$$\g=\g_0\otimes\C[t,t^{-1}]\oplus\C\bfone\oplus\C\partial$$
be the affine Lie algebra of $\g_0$. Its Lie bracket is given by
$$[\xi\otimes t^a+x \bfone+y\partial,\xi'\otimes
t^b+x'\bfone+y'\partial]=[\xi,\xi']\otimes
t^{a+b}+a\delta_{a,-b}\tr(\xi\xi')\bfone+by\xi'\otimes t^b-ay'\xi\otimes t^a,$$ where $\tr:\g_0\ra\C$ is the trace
map. Set $\frakt=\frakt_0\oplus\C\bfone\oplus\C\partial.$

For any Lie algebra $\fraka$ over $\C$, let $\calU(\fraka)$ be its
enveloping algebra. For any $\C$-algebra $R$, we will abbreviate
$\fraka_R=\fraka\otimes R$ and $\calU(\fraka_R)=\calU(\fraka)\otimes
R$.

\emph{In the rest of the paper, we will
fix once for all an integer $c$ such that}
\begin{equation}\label{eq:kappa}
  \kappa=c+m\in\Z_{<0}.
\end{equation}
Let $\calU_\kappa$ be the quotient of $\calU(\g)$
by the two-sided ideal generated by $\bfone-c$. The
$\calU_\kappa$-modules are precisely the $\g$-modules of level $c$.

Given a $\C$-linear map $\lam: \frakt\ra R$ and a $\g_R$-module
$M$ we set
\begin{equation}\label{eq:weight}
M_\lam=\{v\in M\,|\, hv=\lam(h)v,\ \forall\ h\in\frakt\}.
\end{equation}
Whenever $M_\lam$ is non zero, we call $\lam$ a \emph{weight} of $M$.

We equip $\frakt^\ast=\Hom_{\C}(\frakt,\C)$ with the basis
$\ep_1,\ldots,\ep_m$, $\omega_0$, $\delta$ such that
$\ep_1,\ldots,\ep_m\in\frakt_0^\ast$ is dual to the canonical basis
of $\frakt_0$,
$$\delta(\partial)=\omega_0(\bfone)=1,\quad
\omega_0(\frakt_0\oplus\C\partial)=\delta(\frakt_0\oplus\C\bfone)=0.$$
Let $\pair{\bullet:\bullet}$ be the symmetric bilinear form on $\frakt^\ast$ such that
$$\pair{\ep_i:\ep_j}=\delta_{ij},\quad \pair{\omega_0:\delta}=1,\quad
\pair{\frakt_0^\ast\oplus\C\delta:\delta}=\pair{\frakt_0^\ast\oplus\C\omega_0:\omega_0}=0.$$
For $h\in\frakt^\ast$ we will write $||h||^2=\pair{h:h}$. The weights of a $\calU_\kappa$-module belong to
\begin{equation*}
  {}_\kappa\!\frakt^\ast=\{\lam\in\frakt^\ast\,|\,\pair{\lam:\delta}=c\}.
\end{equation*}
Let $a$ denote the projection from $\frakt^\ast$ to $\frakt_0^\ast$.
Consider the map
\begin{equation}\label{eq:z}
z:\frakt^\ast\ra\C
\end{equation}
such that $\lam\mapsto z(\lam)\delta$ is the projection
$\frakt^\ast\ra\C\delta$. \iffalse In the following, we will use the
notation that for any subset $\Phi$ of $\frakt_0^\ast$ we write
\begin{equation*}
  \Phi_\kappa=\{\lam\in{}_\kappa\!\frakt^\ast: a(\lam)\in\Phi\}.
\end{equation*}\fi

Let $\Pi$ be the root system of $\g$ with simple roots
$\al_i=\ep_i-\ep_{i+1}$ for $1\leqs i\leqs m-1$ and
$\al_0=\delta-\sum_{i=1}^{m-1}\al_i$. The root system $\Pi_0$ of
$\g_0$ is the root subsystem of $\Pi$ generated by
$\al_1,\ldots,\al_{m-1}$. We will write $\Pi^+$, $\Pi_0^+$ for the
sets of positive roots in $\Pi$, $\Pi_0$ respectively. %Let $\leqs$ be the partial order on
%$\frakt^\ast$ such that $\lam\leqs\mu$\iff $\mu-\lam\in\N\Pi^+$.

The affine Weyl group $\frakS$ is a Coxeter group with simple
reflections $s_i$ for $0\leqs i\leqs m-1$. It is isomorphic to the
semi-direct product of the symmetric group $\frakS_0$ with the
lattice $\Z\Pi_0$. There is a linear action of $\frakS$ on
$\frakt^\ast$ such that $\frakS_0$ fixes $\omega_0$, $\delta$, and
acts on $\frakt_0^\ast$ by permuting $\ep_i$'s, and an element
$\tau\in\Z\Pi_0$ acts by
\begin{equation}\label{eq:weylaction}
\tau(\delta)=\delta,
\quad\tau(\omega_0)=\tau+\omega_0-\pair{\tau:\tau}\delta/2,
\quad
\tau(\lam)=\lam-\pair{\tau:\lam}\delta,\quad \forall\ \lam\in\frakt_0^\ast.
\end{equation}
Let $\rho_0$ be the half sum of
positive roots in $\Pi_0$ and $\rho=\rho_0+m\omega_0$. The dot action
of $\frakS$ on $\frakt^\ast$ is given by
$w\cdot\lam=w(\lam+\rho)-\rho$. For $\lam\in\frakt^\ast$ we will
denote by $\frakS(\lam)$ the stabilizer of $\lam$ in $\frakS$ under
the dot action. Let $l:\frakS\ra\N$ be the length function.

\subsection{The parabolic Verma modules and their deformations}\label{ss:paraverma}

The subset $\Pi_0$ of $\Pi$ defines a standard parabolic Lie
subalgebra of $\g$, which is given by
$$\frakq=\g_0\otimes\C[t]\oplus\C\bfone\oplus\C\partial.$$ It has a
Levi subalgebra $$\frakl=\g_0\oplus \C\bfone\oplus\C\partial.$$ The
parabolic Verma modules of $\calU_\kappa$ associated with $\frakq$
are given as follows. Let $\lam$ be an element in
$$\Lam^+=\{\lam\in{}_\kappa\!\frakt^\ast\,|\, \pair{\lam:\al}\in\N,\ \forall\ \al\in\Pi_0^+\}.$$ Then there is a unique finite dimensional
simple $\g_0$-module $V(\lam)$ of highest weight $a(\lam)$. It can
be regard as a $\frakl$-module by letting
$h\in\C\bfone\oplus\C\partial$ act by the scalar $\lam(h)$. It is
further a $\frakq$-module if we let the nilpotent radical of
$\frakq$ act trivially. The \emph{parabolic Verma module} of highest
weight $\lam$ is given by
$$M_\kappa(\lam)=\calU(\g)\otimes_{\calU(\frakq)}V(\lam).$$
It has a unique simple quotient, which we denote by
$L_\kappa(\lam)$.

Recall that $R=\C[[s]]$ and $\wp$ is its maximal ideal. Set
$$\bfc=c+s\qquad\textrm{and}\qquad\bfk=\kappa+s.$$ They are elements in $R$. Write
$\calU_\bfk$ for the quotient of $\calU(\g_R)$ by the two-sided ideal
generated by $\bfone-\bfc$. So if $M$ is a $\calU_\bfk$-module, then
$M(\wp)$ is a $\calU_\kappa$-module. Now, note that $R$ admits a
$\frakq_R$-action such that $\g_0\otimes\C[t]$ acts
trivially and $\frakt$ acts by the weight $s\omega_0$. Denote this $\frakq_R$-module by
$R_{s\omega_0}$. For $\lam\in\Lam^+$ the \emph{deformed parabolic
Verma module} $M_\bfk(\lam)$ is the $\g_R$-module induced from the
$\frakq_R$-module $V(\lam)\otimes R_{s\omega_0}$. It is a
$\calU_\bfk$-module of highest weight $\lam+s\omega_0$, and we have a
canonical isomorphism
$$M_\bfk(\lam)(\wp)\cong M_\kappa(\lam).$$
We will abbreviate $\lam_s=\lam+s\omega_0$ and will write
$${}_\bfk\!\frakt^\ast=\{\lam_s\,|\,\lam\in{}_\kappa\!\frakt^\ast\}.$$
\begin{lemma}\label{lem:MKsimple}
The $\g_K$-module $M_\bfk(\lam)_K=M_\bfk(\lam)\otimes_RK$ is simple.
\end{lemma}
\begin{proof}
Assume that $M_\bfk(\lam)_K$ is not simple. Then it contains a nontrivial submodule. This
submodule must have a highest weight vector of weight $\mu_s$ for some
$\mu\in\Lam^+$, $\mu\neq\lam$. By the linkage principle, there
exists $w\in\frakS$ such that $\mu_s=w\cdot\lam_s$. Therefore $w$
fixes $\omega_0$, so it belongs to $\frakS_0$. But then we must have
$w=1$, because $\lam$, $\mu\in\Lam^+$. So $\lam=\mu$. This is a contradiction.
\end{proof}

\subsection{The Jantzen filtration of parabolic Verma modules}\label{ss:jantverma}

For $\lam\in\Lam^+$ the Jantzen
filtration of $M_\kappa(\lam)$ is given as follows. Let $\sigma$ be the
$R$-linear anti-involution on $\g_R$ such that
$$\sigma(\xi\otimes t^n)={}^t\xi\otimes t^{-n},\quad\sigma(\bfone)=\bfone,\quad\sigma(\partial)=\partial.$$
Here $\xi\in\g_0$ and ${}^t\xi$ is the transposed matrix. Let $\g_R$
act on $\Hom_R(M_\bfk(\lam),R)$ via $(xf)(v)=f(\sigma(x)v)$ for
$x\in\g_R$, $v\in M_\bfk(\lam)$. Then
\begin{equation}\label{eq:dualdeform}
\bfD M_\bfk(\lam)=\bigoplus_{\mu\in{}_\bfk\!\frakt^\ast}\Hom_R
\bigl(M_\bfk(\lam)_\mu,
R\bigr)
\end{equation}
is a $\g_R$-submodule of $\Hom_R(M_\bfk(\lam),R)$. It is
the \emph{deformed dual parabolic Verma module} with highest weight
$\lam_s$. The $\lam_s$-weight spaces of $M_\bfk(\lam)$ and $\bfD
M_\bfk(\lam)$ are both free $R$-modules of rank one. Any isomorphism
between them yields, by the universal property of Verma modules, a
$\g_R$-module morphism
$$\phi: M_\bfk(\lam)\ra \bfD M_\bfk(\lam)$$ such that $\phi(\wp)\neq 0$. The Jantzen filtration $\bigl(J^iM_\kappa(\lam)\bigr)$ of $M_\kappa(\lam)$ defined by \cite{J} is the filtration given by Definition \ref{df:jantzen} using the morphism $\phi$ above.

\subsection{The deformed parabolic category $\calO$}\label{ss:plan}

The \emph{deformed parabolic category} $\calO$, denoted by
$\calO_\bfk$, is the category of $\calU_\bfk$-modules $M$ such that
\begin{itemize}
  \item[$\bullet$] $M=\bigoplus_{\lam\in{}_\bfk\!\frakt^\ast}M_\lam$ with
  $M_\lam\in R\modu$,
  \item[$\bullet$] for any $m\in M$ the $R$-module $\calU(\frakq_R)m$ is finitely
  generated.
\end{itemize}
It is an abelian category and contains deformed parabolic Verma
modules. Replacing $\bfk$ by $\kappa$ and $R$ by $\C$ we get the
usual parabolic category $\calO$, denoted $\calO_\kappa$.

Recall the map $z$ in (\ref{eq:z}). For any integer $r$ set
$${}^r_\kappa\!\frakt^\ast=\{\mu\in{}_\kappa\!\frakt^\ast\,|\,
r-z(\mu)\in\Z_{\geqs 0}\}.$$ Define ${}^r_\bfk\!\frakt^\ast$ in the
same manner. Let ${}^r\!\calO_\kappa$ (resp. ${}^r\!\calO_\bfk$) be the
Serre subcategory of $\calO_\kappa$ (resp. $\calO_\bfk$) consisting
of objects $M$ such that $M_\mu\neq 0$ implies that $\mu$ belongs to
${}^r_\kappa\!\frakt^\ast$ (resp. ${}^r_\bfk\!\frakt^\ast$). Write
${}^r\!\Lam^+=\Lam^+\cap{}^r_\kappa\!\frakt^\ast$. We have the following
lemma.

\begin{lemma}\label{lem:fiebig}
(a) For any finitely generated projective object $P$ in
${}^r\!\calO_\bfk$ and any $M\in {}^r\!\calO_\bfk$ the $R$-module
$\Hom_{{}^r\!\calO_\bfk}(P,M)$ is finitely generated and the canonical
map
$$\Hom_{{}^r\!\calO_\bfk}(P,M)(\wp)\ra\Hom_{{}^r\!\calO_\kappa}(P(\wp),M(\wp))$$
is an isomorphism. Moreover, if $M$ is free over $R$, then
$\Hom_{{}^r\!\calO_\bfk}(P,M)$ is also free over $R$.

(b) The assignment $M\mapsto M(\wp)$ yields a functor
$${}^r\!\calO_\bfk\ra{}^r\!\calO_\kappa.$$
This functor gives a bijection between the isomorphism classes of
simple objects and a bijection between the isomorphism classes of
indecomposable projective objects.
\end{lemma}

For any $\lam\in{}^r\!\Lam^+$ there is a unique finitely generated
projective cover ${}^r\!P_\kappa(\lam)$ of $L_\kappa(\lam)$ in
${}^r\!\calO_\kappa$, see \cite[Lemma 4.12]{RW}. Let $L_\bfk(\lam)$,
${}^r\!P_\bfk(\lam)$ be respectively the simple object and the
indecomposable projective object in ${}^r\!\calO_\bfk$ that map respectively to $L_\kappa(\lam)$,
$P_\kappa(\lam)$ by the bijections in Lemma \ref{lem:fiebig}(b).
\begin{lemma}\label{lem:fiebigrw}
  The object ${}^r\!P_\bfk(\lam)$ is, up to isomorphism, the unique finitely generated projective cover of $L_\bfk(\lam)$ in ${}^r\!\calO_\bfk$. It has a filtration by deformed parabolic Verma modules. In particular, it is a free $R$-module.
\end{lemma}

The proof of Lemmas \ref{lem:fiebig}, \ref{lem:fiebigrw} can be
given by imitating \cite[Section 2]{F}. There, Fiebig proved the
analogue of these results for the (nonparabolic) deformed category
$\calO$ by adapting arguments of \cite{RW}. The proof here goes in
the same way, because the parabolic case is also treated in
\cite{RW}. We left the details to the reader. Note that the deformed parabolic category $\calO$ for reductive Lie algebras has also been studied in \cite{St}.

\subsection{The highest weight category
$\calE_\bfk$}\label{ss:calEk}

Fix a positive integer $n\leqs m$. Let $\calP_n$ denote the set of partitions of $n$. Recall that a \emph{partition} $\lam$ of $n$ is a sequence of integers $\lam_1\geqs\ldots\geqs\lam_m\geqs 0$ such that $\sum_{i=1}^m\lam_i=n$. We associate $\lam$ with the element $\sum_{i=1}^m\lam_i\ep_i\in\frakt^\ast_0$, which we denote again by $\lam$. We will identify $\calP_n$ with a subset of $\Lam^+$ by the following inclusion
\begin{equation}\label{eq:calPn}
  \calP_n\ra\Lam^+,\quad \lam\mapsto
  \lam+c\omega_0-\frac{\pair{\lam:\lam+2\rho_0}}{2\kappa}\delta.
\end{equation}
We will also fix an integer $r$ large enough such that $\calP_n$ is contained in ${}^r\!\Lam^+$. Equip $\Lam^+$ with the partial order $\preceq$ given by $\lam\preceq\mu$ if and only if there exists $w\in\frakS$
such that $\mu=w\cdot\lam$ and $a(\mu)-a(\lam)\in\N\Pi_0^+$. Let $\unlhd$ denote the dominance order on $\calP_n$ given by
$$\lam\unlhd\mu\ \iff\ \sum_{j=1}^i\lam_j\leqs\sum_{j=1}^i\mu_j,\quad\forall\ 1\leqs i\leqs m.$$ Note that for $\lam$, $\mu\in\calP_n$ we have
\begin{equation}\label{eq:compareorder}
\lam\preceq\mu\Longrightarrow\lam\unlhd\mu,
\end{equation}
because $\lam\preceq\mu$ implies that
$\mu-\lam\in\N\Pi_0^+$, which implies that
$$\sum_{j=1}^i\mu_j-\sum_{j=1}^i\lam_j=\pair{\mu-\lam,
\ep_1+\cdots+\ep_i}\geqs 0,\quad \forall\ 1\leqs i\leqs m.$$

Now consider the following subset of ${}^r\!\Lam^+$ $$E=\{\mu\in{}^r\!\Lam^+\,|\,\mu=w\cdot\lam\text{ for some
}w\in\frakS,\lam\in\calP_n\}.$$
\begin{lemma}
  The set $E$ is finite.
\end{lemma}
\begin{proof}
  Since $\calP_n$ is finite, it is enough to show that for each
  $\lam\in\calP_n$ the set $\frakS\cdot\lam\cap{}^r\!\Lam^+$ is
  finite. Note that for $w\in\frakS_0$ and $\tau\in\Z\Pi_0$ we have
  $z(w\tau\cdot\lam)=z(\tau\cdot\lam)$. By
(\ref{eq:weylaction}) we have
  $$z(\tau\cdot\lam)=z(\lam)-\frac{\kappa}{2}(||\tau+\frac{\lam+\rho}{\kappa}||^2-||\frac{\lam+\rho}{\kappa}||^2).$$
  If $z(\tau\cdot\lam)\leqs r$, then
  $$||\tau+\frac{\lam+\rho}{\kappa}||^2\leqs\frac{2}{-\kappa}(r-z(\lam))+||\frac{\lam+\rho}{\kappa}||^2.$$
  There exists only finitely many $\tau\in\Z\Pi_0$ which satisfies
  this condition, hence the set $E$ is finite.
\end{proof}
Let $\calE_\kappa$ be the full subcategory of ${}^r\!\calO_\kappa$
consisting of objects $M$ such that
$$\mu\in{}^r\!\Lam^+,\ \mu\notin E\quad\Longrightarrow\quad\Hom_{{}^r\!\calO_\kappa}({}^r\!P_\kappa(\mu),M)=0.$$
Note that since ${}^r\!P_\kappa(\mu)$ is projective in ${}^r\!\calO_\kappa$, an object $M\in{}^r\!\calO_\kappa$ is in $\calE_\kappa$ if and only if each simple subquotient of $M$ is isomorphic to $L_\kappa(\mu)$ for $\mu\in E$. In particular $\calE_\kappa$ is abelian and it is a Serre
subcategory of ${}^r\!\calO_\kappa$. Further $\calE_\kappa$ is also an artinian category. In fact, each object $M\in\calE_\kappa$ has a
finite length because $E$ is finite and for each $\mu\in E$ the
multiplicity of $L_\kappa(\mu)$ in $M$ is finite because $\dim_\C M_\mu<\infty$. Let $\g'$ denote the Lie subalgebra of $\g$ given by
$$\g'=\g_0\otimes\C[t,t^{-1}]\oplus\C\bfone.$$ Forgetting the $\partial$-action yields an equivalence of categories from $\calE_\kappa$ to a category of
$\g'$-modules, see \cite[Proposition 8.1]{So} for details. Since $\kappa$ is negative, this category of $\g'$-modules is equal to the category studied in \cite{KL3}.

\iffalse
\footnote{In fact, for
$\mu=w\cdot\lam\in E$ with $\lam\in\calP_n$, $w\in\frakS$, we have
$||\mu+\rho||^2=||\lam+\rho||^2=||\rho||^2$, where the second
equality follows from the identification (\ref{eq:calPn}). Therefore
$\calE_\kappa$ is contained in the full subcategory of
$\calO_\kappa$ consisting of objects $M$ over which the action of
the generalized Casimir operator is nilpotent. The latter is known
to be equivalent to the category $\calO$ of $\g'$ of level $c$.
See e.g., \cite[Proposition 8.1]{So} for details.}
\fi

\begin{lemma}\label{lem:KL}
  (a) For $\lam\in E$, $\mu\in {}^r\!\Lam^+$ such that
  $[M_\kappa(\lam):L_\kappa(\mu)]\neq 0$ we have $\mu\in E$ and $\mu\preceq\lam$.

  (b) The module ${}^r\!P_\kappa(\lam)$ admits a filtration by
  $\calU_\kappa$-modules
  $${}^r\!P_\kappa(\lam)=P_0\supset P_1\supset\cdots\supset P_l=0$$ such
  that $P_0/P_1$ is isomorphic to $M_\kappa(\lam)$ and
  $P_i/P_{i+1}\cong M_\kappa(\mu_i)$ for some $\mu_i\succ\lam$.

  (c) The category $\calE_\kappa$ is a highest weight
  $\C$-category with standard objects $M_\kappa(\lam)$, $\lam\in E$. The indecomposable projective objects in $\calE_\kappa$ are the modules ${}^r\!P_\kappa(\lam)$ with $\lam\in E$.
\end{lemma}
\begin{proof}
  Let $\bfU_v$ be the quantized enveloping algebra of
  $\g_0$ with parameter $v=\exp(2\pi i/\kappa)$. Then the Kazhdan-Lusztig's
  tensor equivalence \cite[Theorem IV.38.1]{KL3} identifies
  $\calE_\kappa$ with a full subcategory of the category of
  finite dimensional $\bfU_v$-modules. It maps the module
  $M_\kappa(\lam)$ to the Weyl module of $\bfU_v$ with highest weight
  $a(\lam)$. Since $v$ is a root of unity, part (a) follows
  from the strong linkage principle for $\bfU_v$,
  see \cite[Theorem 3.1]{An}. Part (b) follows from
  (a) and \cite[Proposition I.3.9]{KL3}. Finally, part (c) follows directly from parts (a), (b).
\end{proof}
Now, let us consider the deformed version. Let $\calE_\bfk$ be the full subcategory of ${}^r\!\calO_\bfk$
consisting of objects $M$ such that
$$\mu\in{}^r\!\Lam^+,\ \mu\notin E\quad\Longrightarrow\quad\Hom_{{}^r\!\calO_\bfk}({}^r\!P_\bfk(\mu),M)=0.$$

\begin{lemma}\label{lem:calEk}
  An object $M\in{}^r\!\calO_\bfk$ belongs to $\calE_\bfk$ if and only
  if $M(\wp)$ belongs to $\calE_\kappa$. In particular $M_\bfk(\lam)$ and ${}^r\!P_\bfk(\lam)$ belong to $\calE_\bfk$ for $\lam\in E$.
\end{lemma}
\begin{proof}
  By Lemma \ref{lem:fiebig}(a) for any $\mu\in{}^r\!\Lam^+$ the $R$-module
  $\Hom_{{}^r\!\calO_\bfk}({}^r\!P_\bfk(\mu),M)$ is finitely generated and we
  have $$\Hom_{{}^r\!\calO_\bfk}({}^r\!P_\bfk(\mu),M)(\wp)=
  \Hom_{{}^r\!\calO_\kappa}({}^r\!P_\kappa(\mu),M(\wp)).$$
  Therefore $\Hom_{{}^r\!\calO_\bfk}({}^r\!P_\bfk(\mu),M)$ is nonzero if
  and only if $\Hom_{{}^r\!\calO_\kappa}({}^r\!P_\kappa(\mu),M(\wp))$ is
  nonzero by Nakayama's lemma. So the first statement follows from the
  definition of $\calE_\bfk$ and $\calE_\kappa$. The rest follows
  from Lemma \ref{lem:KL}(c).
\end{proof}
Let $$P_\bfk(E)=\bigoplus_{\lam\in E}{}^r\!P_\bfk(\lam), \quad
P_\kappa(E)=\bigoplus_{\lam\in E}{}^r\!P_\kappa(\lam).$$ We have the
following corollary.
\begin{cor}
  (a) The category $\calE_\bfk$ is abelian.

  (b) For $M\in\calE_\bfk$ there exists a positive integer $d$ and a surjective map
  $$P_\bfk(E)^{\oplus d}\lra M.$$
%In particular $M$ is a finitely generated $\calU_\bfk$-module.

  (c) The functor $\Hom_{{}^r\!\calO_\bfk}(P_\bfk(E),-)$ yields an
  equivalence of $R$-categories
  $$\calE_\bfk\cong\End_{{}^r\!\calO_\bfk}(P_\bfk(E))^{\op}\modu.$$
\end{cor}
\begin{proof}
  Let $M\in\calE_\bfk$, $N\in{}^r\!\calO_\bfk$. First assume that
  $N\subset M$. For $\mu\in{}^r\!\Lam^+$ if $\Hom_{{}^r\!\calO_\bfk}({}^r\!P_\bfk(\mu),N)\neq
  0$, then $\Hom_{{}^r\!\calO_\bfk}({}^r\!P_\bfk(\mu),M)\neq
  0$, so $\mu$ belongs to $E$. Hence $N$ belongs to $\calE_\bfk$.
  Now, if $N$ is a quotient of $M$, then $N(\wp)$ is a quotient of
  $M(\wp)$. Since $M(\wp)$ belongs to $\calE_\kappa$, we also have
  $N(\wp)\in\calE_\kappa$. Hence $N$ belongs to $\calE_\bfk$ by
  Lemma \ref{lem:calEk}. This proves part (a). Let us
  concentrate on (b). Since $M\in\calE_\bfk$ we have
  $M(\wp)\in\calE_\kappa$. The category $\calE_\kappa$ is artinian
  with $P_\kappa(E)$ a projective generator. Hence there exists a positive integer $d$ and a surjective map
  $$f: P_\kappa(E)^{\oplus d}\lra M(\wp).$$
  Since $P_\bfk(E)^{\oplus d}$ is projective in ${}^r\!\calO_\bfk$,
  this map lifts to a map of $\calU_\bfk$-modules $\tilde{f}:P_\bfk(E)^{\oplus
  d}\ra M$ such that the following diagram commute
  $$\xymatrix{P_\bfk(E)^{\oplus
  d}\ar[r]^{\tilde{f}}\ar@{>>}[d] & M\ar@{>>}[d]\\ P_\kappa(E)^{\oplus d}\ar@{>>}[r]^{f} & M(\wp).}$$
  Now, since the map $\tilde{f}$ preserves weight spaces and all the weight spaces of $P_\bfk(E)^{\oplus r}$ and $M$ are finitely generated
  $R$-modules, by Nakayama's lemma, the surjectivity of $f$
  implies that $\tilde{f}$ is surjective. This proves (b). %Thus the first statement in
 %(b) is proved. It implies $M$ is finitely generated over $\calU_\bfk$
  %because this property is satisfied for $P_\bfk(E)$ by Lemma
  %\ref{lem:fiebigrw}.
  Finally part (c) is a direct consequence
  of parts (a), (b) by Morita theory.
\end{proof}
\begin{prop}
  The category $\calE_\bfk$ is a highest weight $R$-category with
  standard modules $M_\bfk(\mu)$, $\mu\in E$.
\end{prop}
\begin{proof}
  Note that $\End_{{}^r\!\calO_\bfk}(P_\bfk(E))^{\op}$ is a finite
  projective $R$-algebra by Lemmas \ref{lem:fiebig},
  \ref{lem:fiebigrw}. Since $\calE_\kappa$ is a highest weight
  $\C$-category by Lemma \ref{lem:KL}(c), the result follows from Proposition
  \ref{prop:specialization}.
\end{proof}

\subsection{The highest weight category $\calA_\bfk$}

By definition $\calP_n$ is a subset of $E$. Let
$\calA_\bfk$ be the full subcategory of $\calE_\bfk$ consisting of the objects $M$ such that
$$\Hom_{{}^r\!\calO_\bfk}(M_\bfk(\lam), M)=0,\quad\forall\ \lam\in E,\ \lam\notin\calP_n.$$ We define the subcategory
$\calA_\kappa$ of $\calE_\kappa$ in the same way. Let
$$\Delta_\bfk=\{M_\bfk(\lam)\,|\,\lam\in\calP_n\},\quad
\Delta_\kappa=\{M_\kappa(\lam)\,|\,\lam\in\calP_n\}.$$ Recall that
$E\subset{}^r\!\Lam^+$ is equipped with the partial order $\preceq$,
and that $\calP_n\subset E$. We have the following lemma
  \begin{lemma}\label{lem:ideal}
  The set $\calP_n$ is an ideal in $E$, i.e., for $\lam\in E$,
  $\mu\in\calP_n$, if $\lam\preceq\mu$ then we have $\lam\in\calP_n$.
  \end{lemma}
\begin{proof}
  Let $\lam\in E$ and $\mu\in\calP_n$ and assume that
  $\lam\preceq\mu$. Recall that $E\subset{}_\kappa\!\frakt^\ast$,
  so we can write $a(\lam)=\sum_{i=1}^m\lam_i\ep_i$.
  Since $E\subset{}^r\!\Lam^+$
  we have $\lam_i\in\Z$ and $\lam_i\geqs\lam_{i+1}$.
  We need to show that $\lam_m\in\N$.
  Since $\lam\preceq\mu$ there exist $r_i\in\N$ such that
  $a(\mu)-a(\lam)=\sum_{i=1}^{m-1}r_i\al_i$.
  Therefore we have $\lam_m=\mu_m+r_{m-1}\geqs 0$.
\end{proof}
Now, we can prove the following proposition.
\begin{prop}
  The category $(\calA_\bfk, \Delta_\bfk)$ is a highest weight $R$-category with respect to the partial order $\unlhd$ on $\calP_n$. The highest weight category
  $(\calA_\bfk(\wp), \Delta_\bfk(\wp))$ given by base change is equivalent to
  $(\calA_\kappa, \Delta_\kappa)$.
\end{prop}
\begin{proof}
  Since $\calE_\bfk$ is a highest weight $R$-category
  and $\calP_n$ is an ideal of $E$, \cite[Proposition
  4.14]{R} implies that $(\calA_\bfk, \Delta_\bfk)$ is a highest
  weight $R$-category with respect to the
  partial order $\preceq$ on $\calP_n$. By (\ref{eq:compareorder}) this implies that $(\calA_\bfk, \Delta_\bfk)$ is also a highest
weight $R$-category with respect to $\unlhd$. Finally, the equivalence
$\calA_\bfk(\wp)\cong\calA_\kappa$ follows from the equivalence
$\calE_\bfk(\wp)\cong\calE_\kappa$ and loc.~cit.
\end{proof}

\subsection{Costandard objects of $\calA_\bfk$}\label{ss:duality}

Consider the (contravariant) \emph{duality} functor
$\bfD $ on $\calO_\bfk$
given by
\begin{equation}\label{ss:dualfunctor}
\bfD M=\bigoplus_{\mu\in{}_\bfk\!\frakt^\ast}\Hom_R(M_\mu,R),
\end{equation}
where the action of $\calU_\bfk$ on $\bfD  M$ is given as in Section \ref{ss:jantverma}, with the module $M_\bfk(\lam)$ there replaced by $M$. Similarly, we define the (contravariant) \emph{duality} functor $\bfD$ on $\calO_\kappa$ by
\begin{equation}\label{ss:dualfunctorcc}
\bfD M=\bigoplus_{\mu\in{}_\bfk\!\frakt^\ast}\Hom(M_\mu,\C),
\end{equation}
with the $\calU_\kappa$-action given in the same way. This functor
fixes the simple modules in $\calO_\kappa$. Hence it restricts to a duality functor on
$\calA_\kappa$, because $\calA_\kappa$ is a Serre subcategory of
$\calO_\kappa$. Therefore $(\calA_\kappa,\Delta_\kappa)$ is a
highest weight category with duality in the sense of \cite{CPS}. It
follows from \cite[Proposition 1.2]{CPS} that the costandard module
$M_\kappa(\lam)\spcheck$ in $\calA_\kappa$ is isomorphic to $\bfD
M_\kappa(\lam)$.

\begin{lemma}\label{lem:dualverma}
  The costandard module $M_\bfk(\lam)\spcheck$ in $\calA_\bfk$ is
  isomorphic to $\bfD  M_\bfk(\lam)$ for any $\lam\in\calP_n$.
\end{lemma}
\begin{proof}
  By definition we have a canonical isomorphism
  $$(\bfD  M_\bfk(\lam))(\wp)\cong \bfD(M_\kappa(\lam))\cong M_\kappa(\lam)\spcheck.$$
  Recall from Lemma \ref{lem:MKsimple} that   $M_\bfk(\lam)_K$ is a simple $\calU_{\bfk,K}$-module. Therefore we
  have $(\bfD  M_\bfk(\lam))_K\cong M_\bfk(\lam)_K$. So the lemma
  follows from Lemma \ref{lem:costandsp} applied to the
  highest weight category $(\calA_\bfk,\Delta_\bfk)$ and the set $\{\bfD
M_\bfk(\lam)\,|\,\lam\in\calP_n\}$.
\end{proof}

\subsection{Comparison of the Jantzen filtrations}

By Definition \ref{df:jantzen} for any $\lam\in\calP_n$ there is a
Jantzen filtration of $M_\kappa(\lam)$ associated with the highest
weight category $(\calA_\bfk,\Delta_\bfk)$. Lemma
\ref{lem:dualverma} implies that this Jantzen filtration coincides
with the one given in Section \ref{ss:jantverma}.

\subsection{The $v$-Schur algebra}\label{ss:schur}

\emph{In this section let $R$ denote an arbitrary integral domain.} Let $\bfv$ be an invertible element in $R$. The Hecke algebra $\scrH_\bfv$ over $R$ is a $R$-algebra, which is free as a $R$-module with basis $\{T_w\,|\,w\in\frakS_0\}$, the multiplication is given by
\begin{eqnarray*}
  &T_{w_1}T_{w_2}=T_{w_1w_2},&\text{ if }\ l(w_1w_2)=l(w_1)+l(w_2),\\
  &\quad(T_{s_i}+1)(T_{s_i}-\bfv)=0,&\qquad 1\leqs i\leqs m-1.
\end{eqnarray*}
Next, recall that a composition of $n$ is a sequence
$\mu=(\mu_1,\ldots,\mu_d)$ of positive integers such that
$\sum_{i=1}^d\mu_i=n$. Let $\calX_n$ be the set of compositions of
$n$. For $\mu\in\calX_n$ let $\frakS_\mu$ be the subgroup of
$\frakS_0$ generated by $s_i$ for all $1\leqs i\leqs d-1$ such that
$i\neq\mu_1+\cdots+\mu_j$ for any $j$. Write
\begin{equation*}
x_\mu=\sum_{w\in\frakS_\mu}T_w\quad\text{ and }\quad
y_\mu=\sum_{w\in\frakS_\mu}(-\bfv)^{-l(w)}T_w.
\end{equation*}
The $v$-\emph{Schur algebra} $\bfS_{\bfv}$ of parameter $\bfv$ is the endomorphism algebra of the right $\scrH_\bfv$-module
$\bigoplus_{\mu\in\calX_n}x_\mu\scrH_\bfv$. We will abbreviate
$$\calA_\bfv=\bfS_\bfv\modu.$$ Consider the composition $\varpi$ of $n$ such that $\varpi_i=1$ for $1\leqs i\leqs n$. Then
$x_{\varpi}\scrH_\bfv=\scrH_\bfv$. So the Hecke algebra $\scrH_\bfv$
identifies with a subalgebra of $\bfS_{\bfv}$ via the canonical
isomorphism $\scrH_\bfv\cong\End_{\scrH_\bfv}(\scrH_\bfv)$.

For $\lam\in\calP_n$ let $\lam'$ be the transposed partition of
$\lam$. Let $\varphi_\lam$ be the element in
$\bfS_{\bfv}$ given by $\varphi_\lam(h)=x_\lam h$ for $h\in
x_\varpi\scrH_\bfv$ and $\varphi_\lam(x_\mu\scrH_\bfv)=0$ for any composition $\mu\neq\varpi$. Then there is a particular element $w_\lam\in\frakS_0$ associated
with $\lam$ such that the \emph{Weyl module} $W_\bfv(\lam)$ is the
left ideal in $\bfS_\bfv$ generated by the element
$$z_\lam=\varphi_\lam T_{w_\lam}y_{\lam'}\in\bfS_{\bfv}.$$ See \cite{JM} for
details. We will write
$$\Delta_\bfv=\{W_\bfv(\lam)\,|\,\lam\in\calP_n\}.$$

\subsection{The Jantzen filtration of Weyl modules}

Now, set again $R=\C[[s]]$. Fix
$$v=\exp(2\pi
i/\kappa)\in\C\qquad\textrm{and}\qquad\bfv=\exp(2\pi i/\bfk)\in R.$$
Below we will consider the $v$-Schur algebra over $\C$ with the
parameter $v$, and the $v$-Schur algebra over $R$ with the parameter
$\bfv$. The category $(\calA_v, \Delta_v)$ is a highest weight
$\C$-category. Write $L_v(\lam)$ for the simple quotient of
$W_v(\lam)$. The canonical algebra isomorphism
$\bfS_\bfv(\wp)\cong\bfS_v$ implies that $(\calA_\bfv,\Delta_\bfv)$
is a highest weight $R$-category and there is a canonical
equivalence
\begin{equation*}
(\calA_{\bfv}(\wp),\Delta_{\bfv}(\wp))\cong(\calA_v,\Delta_v).
\end{equation*}
We define the Jantzen filtration $(J^iW_v(\lam))$ of $W_v(\lam)$ by applying
Definition \ref{df:jantzen} to $(\calA_\bfv,\Delta_\bfv)$. This
filtration coincides with the one defined in \cite{JM}, because the
contravariant form on $W_\bfv(\lam)$ used in \cite{JM}'s definition
is equivalent to a morphism from $W_\bfv(\lam)$ to the dual standard
module $W_\bfv(\lam)\spcheck=\Hom_R(W_\bfv(\lam),R)$.

\subsection{Equivalence of $\calA_\bfk$ and $\calA_\bfv$}

In this section we will show that the highest weight $R$-categories $\calA_\bfk$ and $\calA_\bfv$ are equivalent. The proof uses rational double affine Hecke algebras and Suzuki's functor. Let us
first give some reminders. Let $\h=\C^n$, let $y_1,\ldots,y_n$ be its standard basis and
$x_1,\ldots, x_n\in\h^\ast$ be the dual basis. Let $\bfH_{1/\kappa}$
be the rational double affine Hecke algebra associated with
$\frakS_n$ with parameter $1/\kappa$. It is the quotient of the smash product of
the tensor algebra $T(\h\oplus\h^\ast)$ with $\C\frakS_n$
by the relations
$$[y_i,x_i]=1+\frac{1}{\kappa}\sum_{j\neq i}s_{ij},\quad
[y_i,x_j]=\frac{-1}{\kappa}s_{ij},\quad 1\leqs i,j\leqs n,\ i\neq j.$$ Here
$s_{ij}$ denotes the element of $\frakS_n$ that permutes $i$ and
$j$. Denote by $\calB_\kappa$ the category $\calO$ of $\bfH_{1/\kappa}$ as defined in \cite{GGOR}. It is a highest weight $\C$-category. Let $\{B_\kappa(\lam)\,|\,\lam\in\calP_n\}$ be the set of standard modules.

Now, let $V=\C^m$ be the dual of the vectorial representation of
$\g_0$. For any object $M$ in $\calA_\kappa$ consider the action of the Lie algebra $\g_0\otimes\C[z]$ on the vector space
$$T(M)=V^{\otimes n}\otimes M \otimes \C[\h]$$ given by
$$(\xi\otimes z^a)(v\otimes m\otimes f)=\sum_{i=1}^n\xi_{(i)}(v)\otimes m\otimes
x_i^af+v\otimes (-1)^a(\xi\otimes t^{-a})m \otimes f$$ for
$\xi\in\g_0$, $a\in\N$, $v\in V^{\otimes n}$, $m\in M$,
$f\in\C[\h]$. Here $\xi_{(i)}$ is the operator on $V^{\otimes n}$
that acts on the $i$-th copy of $V$ by $\xi$ and acts
on the other copies of $V$ by identity. Suzuki defined a natural action of
$\bfH_{1/\kappa}$ on the space of coinvariants
$$\frakE_\kappa(M)=H_0(\g_0\otimes\C[z],T(M)).$$
The assignment $M\mapsto\frakE_\kappa(M)$ gives a right exact
functor
\begin{equation*}
  \frakE_\kappa:\calA_\kappa\ra\calB_\kappa.
\end{equation*}
See \cite{Su} or \cite[Section 2]{VV} for details. We have
\begin{equation*}
  \frakE_\kappa(M_\kappa(\lam))=B_\kappa(\lam),
\end{equation*}
and $\frakE_\kappa$ is an equivalence of highest weight categories
\cite[Theorem A.5.1]{VV}.

Next, we consider the rational double affine Hecke algebra
$\bfH_{1/\bfk}$ over $R$ with parameter $1/\bfk$. The category
$\calO$ of $\bfH_{1/\bfk}$ is defined in the obvious way. It is a
highest weight $R$-category. We will denote it by $\calB_\bfk$. The standard modules will be denoted by $B_\bfk(\lam)$. The
Suzuki functor over $R$
$$\frakE_{\bfk}: \calA_\bfk\ra\calB_\bfk,\quad M\mapsto H_0(\g_0\otimes\C[z],T(M))$$
is defined in the same way. It has the following properties.

\begin{lemma}\label{lem:suzuki}
(a) We have $\frakE_\bfk(M_\bfk(\lam))=B_\bfk(\lam)$ for $\lam\in\calP_n.$

(b) The functor $\frakE_\bfk$ restricts to an exact functor
$\calA^\Delta_\bfk\ra\calB^\Delta_\bfk.$

(c) The functor $\frakE_\bfk$ maps a projective generator of
$\calA_\bfk$ to a projective generator of $\calB_\bfk$.
\end{lemma}
\begin{proof}
The proof of part (a) is the same as in the nondeformed case. For part (b), since $\frakE_\bfk$ is right exact over $\calA_\bfk$, it is
enough to prove that for any injective morphism $f:M\ra N$ with $M$,
$N\in\calA^\Delta_\bfk$ the map
$$\frakE_\bfk(f):\frakE_\bfk(M)\ra\frakE_\bfk(N)$$ is injective.
Recall from Lemma \ref{lem:MKsimple} that the
$\calU_{\bfk,K}$-module $M_\bfk(\lam)_K$ is simple for any $\lam$.
So the functor
$$\frakE_{\bfk,K}:\calA_{\bfk,K}\ra\calB_{\bfk,K}$$
is an equivalence. Hence the map
$$\frakE_\bfk(f)\otimes_RK:\frakE_{\bfk,K}(M_K)\ra
\frakE_{\bfk,K}(N_K)$$ is injective. Since both $\frakE_\bfk(M)$ and
$\frakE_\bfk(N)$ are free $R$-modules, this implies that
$\frakE_\bfk(f)$ is also injective. Now, let us concentrate on
(c). Let $P$ be a projective generator of $\calA_\bfk$. Then
$P(\wp)$ is a projective generator of $\calA_\kappa$. Since
$\frakE_\kappa$ is an equivalence of categories, we have
$\frakE_\kappa(P(\wp))$ is a projective generator of $\calB_\kappa$.
By (b) the object $\frakE_\bfk(P)$ belongs to
$\calB^\Delta_\bfk$, so it is free over $R$. Therefore by the
Universal Coefficient Theorem we have
$$(\frakE_\bfk(P))(\wp)\cong\frakE_\kappa(P(\wp)).$$
Hence $\frakE_\bfk(P)$ is a projective object of $\calB_\bfk$. Note that
for any $\lam\in\calP_n$ there is a surjective map $P\ra
M_\bfk(\lam)$. The right exact functor $\frakE_\bfk$ sends it to a surjective map
$\frakE_\bfk(P)\ra B_\bfk(\lam)$. This proves that $\frakE_\bfk(P)$
is a projective generator of $\calB_\bfk$.
\end{proof}

\begin{prop}\label{prop:equiv}
Assume that $\kappa\leqs -3$. Then there exists an equivalence
of highest weight $R$-categories
$$\calA_\bfk\simra\calA_\bfv,$$
which maps $M_\bfk(\lam)$ to $W_\bfv(\lam)$ for any
$\lam\in\calP_n$.
\end{prop}
\begin{proof}
We first give an equivalence of highest weight categories
$$\Phi:\calA_\bfk\ra\calB_\bfk$$ as follows. Let $P$ be a projective generator of
$\calA_\bfk$. Then $Q=\frakE_\bfk(P)$ is a projective generator of
$\calB_\bfk$ by Lemma \ref{lem:suzuki}(c). By Morita theory we
have equivalences of categories
\begin{eqnarray*}
\Hom_{\calA_\bfk}(P,-):\calA_\bfk\simra\End_{\calA_\bfk}(P)^{\op}\modu,\\
\Hom_{\calB_\bfk}(Q,-):\calB_\bfk\simra\End_{\calB_\bfk}(Q)^{\op}\modu.
\end{eqnarray*}
We claim that the algebra homomorphism
\begin{equation}\label{eq:isopq}
\End_{\calA_\bfk}(P)\ra\End_{\calB_\bfk}(Q),\quad f\mapsto\frakE_\bfk(f),
\end{equation}
is an isomorphism. To see this, note that we have
$$Q(\wp)=\frakE_\kappa(P(\wp)),\quad(\End_{\calA_\bfk}(P))(\wp)=\End_{\calA_\kappa}(P(\wp)),\quad
(\End_{\calB_\bfk}(Q))(\wp)=\End_{\calB_\kappa}(Q(\wp)).$$ Since
$\frakE_\kappa$ is an equivalence, it yields an isomorphism
$$\End_{\calA_\kappa}(P(\wp))\simra\End_{\calB_\kappa}(Q(\wp)),\quad f\mapsto\frakE_\kappa(f).$$
Since both $\End_{\calA_\bfk}(P)$ and $\End_{\calB_\bfk}(Q)$ are
finitely generated free $R$-modules, by Nakayama's lemma the morphism
(\ref{eq:isopq}) is an isomorphism. In particular, it yields an
equivalence of categories
$$\End_{\calA_\bfk}(P)^{\op}\modu\cong\End_{\calB_\bfk}(Q)^{\op}\modu.$$
Combined with the other two equivalences above, we get an
equivalence of categories
$$\Phi: \calA_\bfk\ra\calB_\bfk.$$
It remains to show that
$$\Phi(M_\bfk(\lam))\cong B_\bfk(\lam),\quad\lam\in\calP_n.$$
Note that the functor $\frakE_\bfk$ yields a morphism of finitely generated $R$-modules
\begin{eqnarray*}
\Hom_{\calB_\bfk}(Q,\Phi(M_\bfk(\lam)))
&=&\End_{\calB_\bfk}(Q)^{\op}\otimes_{\End_{\calA_\bfk}(P)^{\op}}\Hom_{\calA_\bfk}(P,M_\bfk(\lam))\\
&\ra&\Hom_{\calB_\bfk}(Q,\frakE_\bfk(M_\bfk(\lam)))\\
&=&\Hom_{\calB_\bfk}(Q,B_\bfk(\lam)).
\end{eqnarray*}
Let us denote it by $\varphi$. Note also that we have isomorphisms
\begin{eqnarray*}
\Hom_{\calA_\bfk}(P,M_\bfk(\lam))(\wp)
&=&\Hom_{\calA_\kappa}(P(\wp),M_\kappa(\lam)),\\
\Hom_{\calB_\bfk}(Q,B_\bfk(\lam))(\wp)
&=&\Hom_{\calB_\kappa}(Q(\wp),B_\kappa(\lam)),
\end{eqnarray*}
and note that $\frakE_\kappa$ is an equivalence of categories. So the map
$\varphi(\wp)$ is an isomorphism. Further $\Hom_{\calB_\bfk}(Q,B_\bfk(\lam))$ is free over $R$, so Nakayama's lemma implies
that $\varphi$ is also an isomorphism. The preimage of $\varphi$
under the equivalence $\Hom_{\calB_\bfk}(Q,-)$ yields an isomorphism
$$\Phi(M_\bfk(\lam))\simeq B_\bfk(\lam).$$
Finally, if $v\neq -1$, i.e., $\kappa\leqs -3$, then by \cite[Theorem 6.8]{R}
the categories $\calB_\bfk$ and $\calA_\bfv$ are equivalent highest
weight $R$-categories with $B_\bfk(\lam)$ corresponding to
$W_\bfv(\lam)$. This equivalence composed with $\Phi$ gives the desired equivalence in the proposition.
\end{proof}

\begin{cor}\label{cor:jantid}
Assume that $\kappa\leqs -3$. Then for any $\lam$,
$\mu\in\calP_n$ and $i\in\N$ we have
\begin{equation}\label{eq:samefilt}
[J^iM_\kappa(\lam)/J^{i+1}M_\kappa(\lam):L_\kappa(\mu)]=
[J^iW_v(\lam)/J^{i+1}W_v(\lam):L_v(\mu)].
\end{equation}
\end{cor}
\begin{proof}
This follows from the proposition above and Proposition
\ref{prop:jantiso}.
\end{proof}

To prove the main theorem, it remains to compute the left hand side of (\ref{eq:samefilt}). This will be done by generalizing the approach of \cite{BB} to the affine parabolic
case. To this end, we first give some reminders on $\scrD$-modules on affine flag varieties.

\vskip1cm

\section{Generalities on $\scrD$-modules on ind-schemes}\label{s:indscheme}

In this section, we first recall basic notion for $\scrD$-modules on (possibly singular) schemes. We will also discuss twisted $\scrD$-modules and holonomic $\scrD$-modules. Then we introduce the notion of $\scrD$-modules on ind-schemes following \cite{BD} and \cite{KV}.

\subsection{Reminders on $\scrD$-modules}\label{ss:dmodreminder}

Unless specified otherwise, all the schemes will be assumed to be of finite type over $\C$, quasi-separated and quasi-projective. Although a large number of statements are true in a larger generality, we will only use them for quasi-projective schemes. For any scheme $Z$, let $\scrO_Z$ be the structure sheaf over $Z$. We write $\bfO(Z)$ for the category of quasi-coherent $\scrO_Z$-modules on $Z$. Note that we abbreviate
$\scrO_Z$-module for sheaf of $\scrO_Z$-modules over $Z$. For
$f:Z\ra Y$ a morphism of schemes, we write $f_*$, $f^\ast$ for the
functors of direct and inverse images on $\bfO(Z)$, $\bfO(Y)$. If
$f$ is a closed embedding and $\scrM\in \bfO(Y)$, we consider the
quasi-coherent $\scrO_Z$-module
$$f^!\scrM=f^{-1}\homo_{\scrO_Y}(f_\ast\scrO_Z,\scrM).$$
It is the restriction to $Z$ of the subsheaf of $\scrM$ consisting
of sections supported scheme-theoretically on $f(Z)\subset Y$.

Let $Z$ be a smooth scheme. Let $\scrD_Z$ be the ring of
differential operators on $Z$. We denote by $\bfM(Z)$ the category of right $\scrD_Z$-modules that are
quasi-coherent as $\scrO_Z$-modules. It is an abelian category. Let $\Omega_Z$ denote the sheaf of differential forms of highest degree on $Z$. The category of right $\scrD_Z$-modules is equivalent to the category of left $\scrD_Z$-modules via $\scrM\mapsto\Omega_Z\otimes_{\scrO_Z}\scrM$. Let $i:Y\ra Z$ be a morphism of smooth schemes. We consider the $(\scrD_Y, i^{-1}\scrD_Z)$-bimodule
$$\scrD_{Y\ra Z}=i^\ast\scrD_Z=\scrO_Y\otimes_{i^{-1}\scrO_{Z}}i^{-1}\scrD_{Z}.$$
We define the following functors
\begin{eqnarray*}
  &i^\ast:\bfM(Z)\ra\bfM(Y),\quad&\scrM\mapsto
  \Omega_Y\otimes_{\scrO_Y}\bigl(\scrD_{Y\ra Z}\otimes_{\scrD_Z}(\Omega_Z\otimes_{\scrO_Z}\scrM)\bigr),\\
  &i_\bullet: \bfM(Y)\ra\bfM(Z), \quad&\scrM\mapsto i_\ast(\scrM\otimes_{\scrD_Y}\scrD_{Y\ra Z}).
\end{eqnarray*}
For any $\scrM\in\bfM(Y)$ let $\scrM^{\scrO}$ denote the underlying $\scrO_Y$-module of $\scrM$. Then we have $$i^\ast(\scrM^{\scrO})=i^\ast(\scrM)^{\scrO}.$$
If $i$ is a locally closed affine embedding, then the functor $i_\bullet$ is exact. For any closed subscheme $Z'$ of $Z$, we denote by $\bfM(Z,Z')$ the full subcategory of $\bfM(Z)$ consisting of $\scrD_Z$-modules supported set-theoretically on $Z'$. If $i:Y\ra Z$ is a closed embedding of smooth varieties, then by a theorem of Kashiwara, the functor $i_\bullet$ yields an equivalence of categories
\begin{equation}\label{eq:equivkas}
  \bfM(Y)\cong\bfM(Z,Y).
\end{equation}
We refer to \cite{HTT} for more details about $\scrD$-modules on smooth schemes

Now, let $Z$ be a possibly singular scheme. We consider the abelian category $\bfM(Z)$ of right $\scrD$-modules on $Z$ with a faithful \emph{forgetful} functor
$$\bfM(Z)\ra\bfO(Z),\quad\scrM\mapsto\scrM^{\scrO}$$
as in \cite[7.10.3]{BD}. If $Z$ is smooth, it is equivalent to the category $\bfM(Z)$ above, see \cite[7.10.12]{BD}. For any closed embedding
$i:Z\ra X$ there is a left exact functor
$$i^!:\bfM(X)\ra\bfM(Z)$$ such that $(i^!(\scrM))^\scrO=i^!(\scrM^{\scrO})$ for all $\scrM$. It admits an exact left adjoint functor
$$i_\bullet:\bfM(Z)\ra\bfM(X).$$
In the smooth case these functors coincide with the one before.
If $X$ is smooth, then $i_\bullet$ and $i^!$ yield mutually inverse equivalences of categories
\begin{equation}\label{eq:kassing}
  \bfM(Z)\cong\bfM(X,Z).
\end{equation}
such that $i^!\circ i_\bullet=\Id$, see \cite[7.10.11]{BD}. Note that when $Z$ is smooth, this is Kashiwara's equivalence (\ref{eq:equivkas}). We will always consider $\scrD$-modules on a (possibly singular) scheme $Z$ which is given an embedding into a smooth scheme. Finally, if $j:Y\ra Z$ is a locally closed affine embedding and $Y$ is smooth, then we have the following exact functor
\begin{equation}\label{eq:ibullet}
  j_\bullet=i^!\circ(i\circ j)_\bullet: \bfM(Y)\ra\bfM(Z).
\end{equation}
Its definition is independent of the choice of $i$.

\subsection{Holonomic $\scrD$-modules}\label{ss:holoD-mod}

Let $Z$ be a scheme. If $Z$ is smooth, we denote by $\bfM_h(Z)$ the category of holonomic $\scrD_Z$-modules, see e.g., \cite[Definition 2.3.6]{HTT}. Otherwise, let $i:Z\ra X$ be a closed embedding into a smooth
scheme $X$. We define $\bfM_h(Z)$ to be the full subcategory of $\bfM(Z)$ consisting of objects $\scrM$ such that $i_\bullet\scrM$ is holonomic. The category $\bfM_h(Z)$ is abelian. There
is a (contravariant) \emph{duality} functor on $\bfM_h(Z)$ given by
$$\bbD:\bfM_h(Z)\ra\bfM_h(Z),\quad\scrM\mapsto i^!\bigl(\Omega_X\otimes_{\scrO_X}\ext^{\dim
X}_{\scrD_X}(i_\bullet\scrM,\scrD_X)\bigr).$$
For a locally closed affine embedding $i:Y\ra
Z$ with $Y$ a smooth scheme, the functor $i_\bullet$ given
by (\ref{eq:ibullet}) maps $\bfM_h(Y)$ to $\bfM_h(Z)$. We put $$i_!=\bbD\circ i_\bullet\circ\bbD:\ \bfM_h(Y)\ra \bfM_h(Z).$$
There is a canonical morphism of functors
$$\psi: i_!\ra i_\bullet.$$ The \emph{intermediate extension} functor is
given by
$$i_{!\bullet}:\bfM_h(Y)\ra\bfM_h(Z),\quad
\scrM\mapsto\Im(\psi(\scrM):i_!\scrM\ra i_\bullet\scrM).$$
Note that the functors $i_\bullet$, $i_!$ are exact. Moreover, if the embedding $i$ is closed, then $\psi$ is an isomorphism of functors $i_!\cong i_\bullet$.
%Finally, note that if $i$ is an open embedding and $Z$ is a smooth scheme, then we have $i^!=i^\ast$. In particular $i^!$ is exact and
%\begin{equation}
%\label{eq:adjointtriple}
%(i_!, i^!=i^\ast, i_\bullet)
%\end{equation}
%form a triple of adjoint functors between the categories $\bfM_h(Y)$ and $\bfM_h(Z)$. In this case, we have
%$$(i_\bullet(\scrM))^{\scrO}=i_\ast(\scrM^{\scrO}),\quad \forall\ \scrM\in\bfM_h(Z),$$
%see e.g. \cite{HTT} for more details.

\subsection{Weakly equivariant $\scrD$-modules}

Let $T$ be a linear group. For any $T$-scheme $Z$ there is an abelian category $\bfM^T(Z)$ of weakly $T$-equivariant right $\scrD$-modules on $Z$ with a faithful forgetful functor
\begin{equation}\label{eq:forgetT}
\bfM^T(Z)\ra\bfM(Z).
\end{equation}
If $Z$ is smooth, an object $\scrM$ of $\bfM^T(Z)$ is an object $\scrM$ of $\bfM(Z)$ equipped with a structure of
$T$-equivariant $\scrO_Z$-module such that the action map
$\scrM\otimes_{\scrO_Z}\scrD_Z\ra\scrM$ is $T$-equivariant. For any $T$-scheme $Z$ with a $T$-equivariant closed embedding $i:Z\ra X$ into a smooth $T$-scheme $X$, the functor $i_\bullet$ yields
an equivalence
$\bfM^T(Z)\cong\bfM^T(X,Z)$,
where $\bfM^T(X,Z)$ is the subcategory of $\bfM^T(X)$ consisting of objects supported set-theoretically on $Z$.

\subsection{Twisted $\scrD$-modules}\label{ss:twisted}

Let $T$ be a torus, and let $\frakt$ be its Lie algebra. Let
$\pi:Z^\dag\ra Z$ be a right $T$-torsor over the scheme $Z$. For any
object $\scrM\in\bfM^T(Z^\dag)$ the $\scrO_Z$-module
$\pi_\ast(\scrM^\scrO)$ carries a $T$-action. Let
$$\scrM^\dag=\pi_\ast(\scrM^\scrO)^T$$ be the $\scrO_Z$-submodule of $\pi_\ast(\scrM^\scrO)$ consisting of
the $T$-invariant local sections. We have
$$\Gamma(Z,\scrM^\dag)=\Gamma(Z^\dag,\scrM)^T.$$
For any weight $\lam\in\frakt^\ast$ we define the categories
$\bfM^{\tilde{\lam}}(Z)$, $\bfM^{\lam}(Z)$ as follows.

First, assume that $Z$ is a smooth scheme. Then $Z^\dag$ is also
smooth. So we have a sheaf of algebras on $Z$ given by
$$\scrD_Z^\dag=(\scrD_{Z^\dag})^\dag,$$ and $\scrM^\dag$ is a
right $\scrD_Z^\dag$-module for any $\scrM\in\bfM^T(Z^\dag)$.
For any open subscheme $U\subset Z$ the $T$-action on $\pi^{-1}(U)$
yields an algebra homomorphism
\begin{equation}\label{eq:taction}
\delta_r:\ \calU(\frakt)\ra\Gamma(U,\scrD_Z^\dag),
\end{equation}
whose image lies in the center of the right hand side. Thus there is
also an action of $\calU(\frakt)$ on $\scrM^\dag$ commuting with the
$\scrD_Z^\dag$-action. For $\lam\in\frakt^\ast$ let
$\frakm_\lam\subset\calU(\frakt)$ be the ideal generated by
$$\{h+\lam(h)\,\,|\,\, h\in\frakt\}.$$
We define $\bfM^\lam(Z)$ (resp. $\bfM^{\tilde{\lam}}(Z)$) to be the
full subcategory of $\bfM^T(Z^\dag)$ consisting of the objects
$\scrM$ such that the action of $\frakm_\lam$ on $\scrM^\dag$ is
zero (resp. nilpotent). In particular $\bfM^\lam(Z)$ is a full
subcategory of $\bfM^{\tilde\lam}(Z)$ and both categories are abelian. We will write
\begin{eqnarray}\label{eq:gammadag}
  \Gamma(Z,\scrM)&=&\Gamma(Z,\scrM^\dag),\quad\forall\ \scrM\in\bfM^{\tilde{\lam}}(Z).
\end{eqnarray}

Now, let $Z$ be any scheme. We say that a $T$-torsor $\pi:Z^\dag\ra Z$ is \emph{admissible} if there exists a $T$-torsor $X^\dag\ra X$ with
$X$ smooth and a closed embedding $i:Z\ra X$ such that $Z^\dag\cong
X^\dag\times_XZ$ as a $T$-scheme over $Z$. We will only use admissible $T$-torsors. Let $\bfM^\lam(X,Z)$, $\bfM^{\tilde\lam}(X,Z)$ be respectively the subcategories of $\bfM^\lam(X)$, $\bfM^{\tilde\lam}(X)$ consisting of objects supported on $Z^\dag$.
We define $\bfM^\lam(Z)$, $\bfM^{\tilde{\lam}}(Z)$ to be the full
subcategories of $\bfM^T(Z^\dag)$ consisting of objects $\scrM$ such that $i_\bullet(\scrM)$ belongs to $\bfM^\lam(X,Z)$, $\bfM^{\tilde\lam}(X,Z)$ respectively. Their definition only depends on $\pi$.

\begin{rk}
Let $Z$ be a smooth scheme. Let $\bfM(\scrD_Z^\dag)$ be the
category of right $\scrD_Z^\dag$-modules on $Z$ that are
quasi-coherent as $\scrO_Z$-modules. The functor
\begin{equation}\label{eq:equivdag}
    \bfM^T(Z^\dag)\simra\bfM(\scrD_Z^\dag),\quad\scrM\mapsto\scrM^\dag
\end{equation}
is an equivalence of categories. A quasi-inverse is given by
$\pi^\ast$, see e.g., \cite[Lemma 1.8.10]{BB}.
\end{rk}

\begin{rk}\label{rk:omega}
We record the following fact for a further use. For any smooth $T$-torsor $\pi:Z^\dag\ra Z$, the exact sequence of
relative differential $1$-forms
$$\pi^\ast(\Omega^1_{Z})\lra\Omega^1_{Z^\dag}\lra\Omega^1_{Z^\dag/Z}\lra
0$$ yields an isomorphism
$$\Omega_{Z^\dag}=\pi^\ast(\Omega_{Z})\otimes_{\scrO_{Z^\dag}}
\Omega_{Z^\dag/Z}.$$
Since $\pi$ is a $T$-torsor we have indeed
$$\Omega_{Z^\dag/Z}=\scrO_{Z^\dag}$$
as a line bundle. Therefore we have an isomorphism of $\scrO_{Z^\dag}$-modules
$$\Omega_{Z^\dag}=\pi^\ast(\Omega_{Z}).$$
Below, we will identify them whenever needed.
\end{rk}

\subsection{Twisted holonomic $\scrD$-modules and duality functors}\label{ss:twistholonome}

Let $\pi:Z^\dag \ra Z$ be an admissible $T$-torsor. We define $\bfM^T_h(Z^\dag)$ to be the full subcategory of $\bfM^T(Z^\dag)$ consisting of objects $\scrM$ whose image via the functor (\ref{eq:forgetT}) belongs to $\bfM_h(Z^\dag)$. We define the categories $\bfM^{\lam}_h(Z)$, $\bfM^{\tilde\lam}_h(Z)$ in the same manner.

Assume that $Z$ is smooth. Then the category $\bfM^T(Z^\dag)$ has enough injective objects, see e.g., \cite[Proposition 3.3.5]{Kas} and the references there. We define a (contravariant) duality functor on $\bfM^T_h(Z^\dag)$ by
$$\bbD': \bfM^T_h(Z^\dag)\ra\bfM^T_h(Z^\dag),\quad\scrM\mapsto
\Omega_{Z^\dag}\otimes_{\scrO_{Z^\dag}}\ext^{\dim Z^\dag}_{\bfM^T(Z^\dag)}(\scrM,\scrD_{Z^\dag}).$$
We may write $\bbD'=\bbD'_Z$. Note that by Remark \ref{rk:omega} and the equivalence (\ref{eq:equivdag}) we have
\begin{equation}
(\bbD'_Z\scrM)^\dag=\Omega_{Z}\otimes_{\scrO_{Z}}\ext^{\dim Z^\dag}_{\scrD_Z^\dag}(\scrM^\dag,\scrD_Z^\dag),\quad\forall\ \scrM\in\bfM^T(Z^\dag).
\end{equation}
For any $\lam\in\frakt^\ast$ the functor $\bbD'$ restricts to (contravariant) equivalences of categories
\begin{equation}\label{eq:bbD'lam}
\bbD': \bfM^{\tilde\lam}_h(Z)\ra\bfM^{\widetilde{-\lam}}_h(Z),\qquad
\bbD': \bfM^{\lam}_h(Z)\ra\bfM^{-\lam}_h(Z),\end{equation}
see e.g., \cite[Remark 2.5.5(iv)]{BB}. In particular, if $\lam=0$ then $\bbD'$ yields a duality on $\bfM^{0}_h(Z^\dag)$. It is compatible with the duality functor $\bbD$ on $\bfM_h(Z)$ defined in Section \ref{ss:holoD-mod} via the equivalence
$$\Phi:\bfM^{0}_h(Z)\ra  \bfM_h(Z),\quad \scrM\mapsto\scrM^\dag.$$
given by (\ref{eq:equivdag}). More precisely, we have the following lemma
\begin{lemma}\label{lem:interduality}
We have
$\Phi\circ\bbD'=\bbD\circ\Phi.$
\end{lemma}
The proof is standard, and is left to the reader.

\vskip1mm

Similarly, for an admissible $T$-torsor $\pi:Z^\dag\ra Z$ with an embedding $i$ into a smooth $T$-torsor $X^\dag\ra X$, we define the functor
$$\bbD': \bfM^T_h(Z^\dag)\ra \bfM^T_h(Z^\dag),\quad \scrM\mapsto i^!\bbD'_X(i_\bullet(\scrM)).$$
Its definition only depends on $\pi$. The equivalence (\ref{eq:bbD'lam}) and Lemma \ref{lem:interduality} hold again.

\vskip1mm

A weight $\lam\in\frakt^\ast$ is \emph{integral} if it is given by the differential of a character $e^\lam: T\ra\C^\ast.$ For such a $\lam$ we consider the invertible sheaf
$\scrL^\lam_Z\in\bfO(Z)$ defined by
$$\Gamma(U,\scrL^\lam_Z)=\{\gamma\in\Gamma(\pi^{-1}(U),
\scrO_{Z^\dag})\,|\, \gamma(xh^{-1})=e^\lam(h)\gamma(x),\quad
(x,h)\in \pi^{-1}(U)\times T\}$$ for any open set $U\subset Z$.
We define the following \emph{translation functor}
$$\Theta^\lam:\bfM^T_h(Z^\dag)\ra\bfM^T_h(Z^\dag),\quad \scrM\mapsto\scrM\otimes_{\scrO_{Z^\dag}} \pi^\ast(\scrL^\lam_{Z}).$$
It is an equivalence of categories. A quasi-inverse is given by $\Theta^{-\lam}$.
For any $\mu\in\frakt^\ast$ the restriction of $\Theta^\lam$ yields equivalences of categories
\begin{equation}\label{eq:thetalam}
\Theta^\lam: \bfM_h^{\tilde{\mu}}(Z)\ra \bfM_h^{\widetilde{\mu+\lam}}(Z),\quad \Theta^\lam: \bfM_h^{\mu}(Z)\ra \bfM_h^{\mu+\lam}(Z).
\end{equation}
We define the \emph{duality} functor on $\bfM_h^{\tilde{\lam}}(Z)$ to be
$$\bbD:\bfM_h^{\tilde{\lam}}(Z)\ra \bfM_h^{\tilde{\lam}}(Z),\quad\scrM\mapsto
\Theta^{2\lam}\circ\bbD'(\scrM).$$
It restricts to a duality functor on $\bfM_h^\lam(Z)$, which we denote again by $\bbD$. To avoid any confusion, we may write $\bbD=\bbD^\lam$. The equivalence $\Theta^\lam$ intertwines the duality functors, i.e., we have
\begin{equation}\label{eq:interthetaduality}
  \bbD^{\lam+\mu}\circ\Theta^\lam=\Theta^\lam\circ\bbD^\mu.
\end{equation}

For any locally closed affine embedding of $T$-torsors $i: Z\ra Y$ with $Z$ smooth, we define the functor
$$i_!=\bbD\circ i_\bullet\circ \bbD:\bfM^{\tilde\lam}_h(Z)\ra \bfM^{\tilde{\lam}}_h(Y).$$
As in Section \ref{ss:holoD-mod}, we have a morphism of exact functors $\psi:i_!\ra i_\bullet$ which is an isomorphism if $i$ is a closed embedding. The intermediate extension functor $i_{!\bullet}$ is defined in the same way.

\begin{rk}
  Assume that $Z$ is smooth. Let $\scrM\in\bfM^\lam_h(Z)$. Put $\mu=0$ in (\ref{eq:thetalam}). Using the equivalence $\Phi$ we see that $\scrM^\dag$ is a right module over the sheaf of algebras
  $$\scrD^\lam_Z=\scrL^{-\lam}_Z\otimes_{\scrO_Z}
  \scrD_Z\otimes_{\scrO_Z}\scrL^\lam_Z.$$
  Further, we have
  $$\bbD(\scrM)^\dag=\Omega_{Z}\otimes_{\scrO_{Z}}
  \otimes_{\scrO_{Z}}\scrL^{2\lam}_Z\otimes_{\scrO_Z}\ext^{\dim Z}_{\scrD^\lam_{Z}}(\scrM^\dag,\scrD^\lam_{Z})$$
   by Lemma \ref{lem:interduality} and (\ref{eq:interthetaduality}), compare \cite[(2.1.2)]{KT1}.
\end{rk}

\subsection{Injective and projective limit of categories}\label{ss:pilimcat}

Let us introduce the following notation. Let $A$ be a filtering poset. For any inductive system of
categories $(\calC_\al)_{\al\in A}$ with functors
$i_{\al\beta}:\calC_\al\ra\calC_\beta$, $\al\leqs\beta$, we denote
by $\tilim\calC_\al$ its \emph{inductive limit}, i.e., the category
whose objects are pairs $(\al, M_\al)$ with $\al\in A$,
$M_\al\in\calC_\al$ and
$$\Hom_{\tilim\calC_\al}((\al, M_\al),
(\beta,N_\beta))=\ilim_{\gamma\geqs\al,\beta}\Hom_{\calC_\gamma}(i_{\al\gamma}(M_\al),
i_{\beta\gamma}(N_\beta)).$$
For any projective system of categories $(\calC_\al)_{\al\in A}$
with functors $j_{\al\beta}:\calC_\beta\ra\calC_\al$,
$\al\leqs\beta$, we denote by $\tplim\calC_\al$ its \emph{projective
limit}, i.e., the category whose objects are systems consisting of
objects $M_\al\in\calC_\al$ given for all $\al\in A$ and
isomorphisms $j_{\al\beta}(M_\beta)\ra M_\al$ for each
$\al\leqs\beta$ and satisfying the compatibility condition for each
$\al\leqs\beta\leqs\gamma$. Morphisms are defined in the obvious
way. See e.g., \cite[3.2, 3.3]{KV}.

\subsection{The $\scrO$-modules on ind-schemes}\label{ss:O-modules}

An \emph{ind-scheme} $X$ is a filtering
inductive system of schemes $(X_\al)_{\al\in A}$ with closed
embeddings $i_{\al\beta}:X_\al\ra X_\beta$ for $\al\leqs \beta$ such
that $X$ represents the ind-object $``\ilim"X_\al$. See
\cite[1.11]{KS} for details on ind-objects. Below we will simply
write $\ilim$ for $``\ilim"$, hoping this does not create any
confusion. The categories $\bfO(X_\al)$ form a projective system via
the functors $i_{\al\beta}^!:\bfO(X_\beta)\ra\bfO(X_\al).$ Following
\cite[7.11.4]{BD} and \cite[3.3]{KV} we define the category of
$\scrO$-modules on $X$ as
$$\bfO(X)=\tplim\bfO(X_\al).$$
It is an abelian category. An object $\scrM$ of $\bfO(X)$ is represented by
$$\scrM=(\scrM_\al,\ \varphi_{\al\beta}: i_{\al\beta}^!\scrM_\beta\ra \scrM_\al)$$ where $\scrM_\al$ is an object of
$\bfO(X_\al)$ and $\varphi_{\al\beta}$, $\al\leqs\beta$, is an
isomorphism in $\bfO(X_\al)$.

Note that any object $\scrM$ of $\bfO(X)$ is an inductive limit of
objects from $\bfO(X_\al)$. More precisely, we first identify
$\bfO(X_\al)$ as a full subcategory of $\bfO(X)$ in the following
way: since the poset $A$ is filtering, to any $\scrM_\al\in\bfO(X_\al)$ we
may associate a canonical object $(\scrN_\beta)$ in $\bfO(X)$ such that
$\scrN_{\beta}=i_{\al\beta\ast}(\scrM_\al)$ for $\al\leqs\beta$ and
the structure isomorphisms $\varphi_{\beta\gamma}$, $\beta\leqs\gamma$, are the obvious
ones. Let us denote this object in $\bfO(X)$ again by $\scrM_\al$.
Given any object $\scrM\in\bfO(X)$ represented by
$\scrM=(\scrM_\al,\varphi_{\al\beta})$, these $\scrM_\al\in\bfO(X)$,
$\al\in A$, form an inductive system via the canonical morphisms
$\scrM_\al\ra\scrM_\beta$. Then, the ind-object $\ilim\scrM_\al$ of
$\bfO(X)$ is represented by $\scrM$. So, we define the space of
global sections of $\scrM$ to be the inductive limit of vector
spaces
\begin{equation}\label{eq:gammaindM}
  \Gamma(X,\scrM)=\ilim\Gamma(X_\al,\scrM_\al).
\end{equation}

We will also use the category $\hat\bfO(X)$ defined as the limit of
the projective system of categories
$(\bfO(X_\al),i_{\al\beta}^\ast)$, see \cite[7.11.3]{BD} or
\cite[3.3]{KV}. Note that the canonical isomorphisms
$i_{\al\beta}^\ast\scrO_{X_\beta}=\scrO_{X_\al}$ yield an object
$(\scrO_{X_\al})_{\al\in A}$ in $\hat\bfO(X)$. We denote this
object by $\scrO_X$. An object $\scrF\in\hat\bfO(X)$ is said to be
\emph{flat} if each $\scrF_\al$ is a flat $\scrO_{X_\al}$-module.
Such a $\scrF$ yields an exact functor
\begin{equation}\label{eq:tensor}
\bfO(X)\ra\bfO(X),\quad\scrM\mapsto
\scrM\otimes_{\scrO_X}\scrF=(\scrM_\al\otimes_{\scrO_{X_\al}}\scrF_\al)_{\al\in A}.
\end{equation}
For $\scrF\in\hat\bfO(X)$ the vector spaces
$\Gamma(X_\al,\scrF_\al)$ form a projective system with the
structure maps induced by the functors $i^*_{\al\beta}$. We set
\begin{equation}\label{eq:gammaF}
\Gamma(X,\scrF)=\plim\Gamma(X_\al,\scrF_\al).
\end{equation}

\subsection{The $\scrD$-modules on ind-schemes}\label{ss:D-mod}

The category of $\scrD$-modules on the ind-scheme $X$ is defined as
the limit of the inductive system of categories $(\bfM(X_\al),
i_{\al\beta\bullet})$, see e.g., \cite[3.3]{KV}. We will denote it
by $\bfM(X)$. Since $\bfM(X_\al)$ are abelian categories and
$i_{\al\beta\bullet}$ are exact functors, the category $\bfM(X)$ is
abelian. Recall that an object of $\bfM(X)$ is represented by a pair
$(\al,\scrM_\al)$ with $\al\in  A$, $\scrM_\al\in\bfM(X_\al)$.
There is an exact and faithful forgetful functor
\begin{equation*}
\bfM(X)\ra\bfO(X),\quad \scrM=(\al,\scrM_\al)\ra
\scrM^\scrO=(i_{\al\beta\bullet}(\scrM_\al)^{\scrO})_{\beta\geqs\al}.
\end{equation*}
The \emph{global sections} functor on $\bfM(X)$ is defined by
$$\Gamma(X,\scrM)=\Gamma(X,\scrM^\scrO).$$

Next, we say that $X$ is a \emph{$T$-ind-scheme} if $X=\ilim X_\al$ with each
$X_\al$ being a $T$-scheme and $i_{\al\beta}:X_\al\ra X_\beta$ being
$T$-equivariant. We define $\bfM^T(X)$ to be the abelian category given by the limit
of the inductive system of categories
$(\bfM^T(X_\al),i_{\al\beta\bullet})$. The functors (\ref{eq:forgetT}) for each $X_\al$
yield an exact and faithful functor
\begin{equation}
\bfM^T(X)\ra\bfM(X).
\end{equation}
The functor $\Gamma$ on $\bfM^T(X)$ is
given by the functor $\Gamma$ on $\bfM(X)$.

Finally, given a $T$-ind-scheme $X=\ilim X_\al$ let $\pi:X^\dag\ra X$ be a $T$-torsor over $X$, i.e., $\pi$ is the limit of an inductive
system of $T$-torsors $\pi_\al: X^\dag_\al\ra X_\al$. We say that
$\pi$ is \emph{admissible} if each of the $\pi_\al$ is admissible.
Assume this is the case. Then the categories $\bfM^{\lam}(X_\al)$,
$\bfM^{\tilde{\lam}}(X_\al)$ form,
respectively, two inductive systems of categories via
$i_{\al\beta\bullet}$. Let
$$\bfM^\lam(X)=\tilim\bfM^{\lam}(X_\al),\qquad
\bfM^{\tilde\lam}(X)=\tilim\bfM^{\tilde\lam}(X_\al).$$ They are
abelian subcategories of $\bfM^T(X^\dag)$. For any object
$\scrM=(\al,\scrM_\al)$ of $\bfM^T(X^\dag)$, the
$\scrO_{X_\beta}$-modules $(i_{\al\beta\bullet}\scrM_\al)^\dag$ with
$\beta\geqs\al$ give an object of $\bfO(X)$. We will denote it by
$\scrM^\dag$. The functor
$$\bfM^T(X^\dag)\ra\bfO(X),\quad \scrM\mapsto\scrM^\dag$$
is exact and faithful. For $\scrM\in\bfM^{\tilde\lam}(X)$ we will write
\begin{equation}\label{eq:gammaM}
  \Gamma(\scrM)=\Gamma(X,\scrM^\dag).
\end{equation}
Note that it is also equal to $\Gamma(X^\dag,\scrM)^T$. We will consider the following categories
$$\bfM^{\tilde\lam}_h(X)=\tilim\bfM^{\tilde\lam}_h(X_\al),\quad \bfM^{\lam}_h(X)=\tilim\bfM^{\lam}_h(X_\al).$$ Let $Y$ be a smooth scheme.
A locally closed affine embedding $i:Y\ra X$ is the composition of
an affine open embedding $i_1:Y\ra\olY$ with a closed embedding $i_2:\olY\ra
X$. For such a morphism the functor
$i_\bullet:\bfM^{\tilde\lam}_h(Y)\ra\bfM^{\tilde\lam}_h(X)$ is defined by $i_\bullet=i_{2\bullet}\circ i_{1\bullet}$, and the
functor $i_!:\bfM^{\tilde\lam}_h(Y)\ra\bfM^{\tilde\lam}_h(X)$ is defined by $i_!=i_{2\bullet}\circ i_{1!}$.

\subsection{The sheaf of differential operators on a formally smooth ind-scheme.}

Let $X$ be a formally smooth\footnote{See \cite[7.11.1]{BD} and the references there for the definition of formally smooth.} ind-scheme. Fix $\beta\geqs\al$ in $ A$. Let $\diff_{\beta,\al}$ be the
$\scrO_{X_\beta\times X_\al}$-submodule of $\homo_\C(\scrO_{X_\beta},
i_{\al\beta\ast}\scrO_{X_\al})$ consisting of local sections supported
set-theoretically on the diagonal $X_\al\subset X_\beta\times
X_\al$. Here $\homo_\C(\scrO_{X_\beta},
i_{\al\beta\ast}\scrO_{X_\al})$ denotes the sheaf of morphisms between the sheaves of $\C$-vector spaces associated with $\scrO_{X_\beta}$ and $i_{\al\beta\ast}\scrO_{X_\al}$. As a left $\scrO_{X_\al}$-module $\diff_{\beta,\al}$ is quasi-coherent, see e.g., \cite[Section 1.1]{BB}. So it is an object in $\bfO(X_\al)$. For $\beta\leqs\gamma$ the functor $i_{\beta\gamma\ast}$ and the canonical map
$\scrO_{X_\gamma}\ra i_{\beta\gamma\ast}\scrO_{X_\beta}$ yield a morphism of $\scrO_{X_\al}$-modules
$$\homo_\C(\scrO_{X_\beta}, i_{\al\beta\ast}\scrO_{X_\al})\ra\homo_\C(\scrO_{X_\gamma}, i_{\al\gamma\ast}\scrO_{X_\al}).$$
It induces a morphism
$\diff_{\beta,\al}\ra \diff_{\gamma,\al}$ in $\bfO(X_\al)$. The $\scrO_{X_\al}$-modules $\diff_{\beta,\al}$, $\beta\geqs\al$, together with these maps form an inductive system. Let
$$\diff_\al=\ilim_{\beta\geqs\al}\diff_{\beta,\al}\in\bfO(X_\al).$$
The system consisting of the $\diff_\al$'s and the
canonical isomorphisms $i_{\al\beta}^\ast\diff_\beta\ra\diff_\al$
is a flat object in $\hat\bfO(X)$, see \cite[7.11.11]{BD}. We will
call it the \emph{sheaf of differential operators} on $X$ and denote
it by $\scrD_X$. It carries canonically a structure of
$\scrO_X$-bimodules, and a structure of algebra given by
$$\diff_{\gamma,\beta}\otimes_{\scrO_{X_\beta}}\diff_{\beta,\al}\ra\diff_{\gamma,\al},\quad (g,f)\mapsto g\circ f,\quad\al\leqs\beta\leqs\gamma.$$
Any object $\scrM\in\bfM(X)$ admits a canonical right $\scrD_X$-action given by a morphism
\begin{equation}\label{eq:daction}
\scrM\otimes_{\scrO_X}\scrD_X\ra \scrM
\end{equation}
in $\bfO(X)$ which is compatible with the multiplication in $\scrD_X$.

\vskip1cm

\section{Localization theorem for affine Lie algebras of negative level}\label{s:localization}

In this section we first consider the affine localization theorem which relates right $\scrD$-modules on the affine flag variety (an ind-scheme) to a category of modules over the affine Lie algebra with integral weights and a negative level. Then we compute the $\scrD$-modules corresponding to Verma and parabolic Verma modules. All the constructions here hold for a general simple linear group. We will only use the case of $SL_m$, since the multiplicities on the left hand side of (\ref{eq:samefilt}) that we want to compute are the same for $\mathfrak{sl}_m$ and $\mathfrak{gl}_m$. We will use, for $\mathfrak{sl}_m$, the same notation as in Section \ref{s:affineO} for $\mathfrak{gl}_m$. In
particular $\g_0=\mathfrak{sl}_m$ and $\frakt_0^\ast$ is now given the basis consisting of the
weights $\ep_i-\ep_{i+1}$ with $1\leqs i\leqs m-1$. We identify
$\calP_n$ as a subset of $\frakt_0^\ast$ via the map
\begin{equation*}
\calP_n\ra\frakt^\ast_0,\quad\lam=(\lam_1,\ldots,\lam_m)\mapsto
\sum_{i=1}^{m}(\lam_i-n/m)\ep_i.
\end{equation*}
Finally, we will modify slightly
the definition of $\g$ by extending $\C[t,t^{-1}]$ to $\C((t))$,
i.e., from now on we set
$$\g=\g_0\otimes\C((t))\oplus\C\bfone\oplus\C\partial.$$
The bracket is given in the same way as before. We will again denote
by $\frakb$, $\frakn$, $\frakq$, etc., the corresponding Lie
subalgebras of $\g$.

\subsection{The affine Kac-Moody group}

Consider the group ind-scheme $LG_0=G_0(\C((t)))$ and the group
scheme $L^+G_0=G_0(\C[[t]])$. Let $I\subset L^+G_0$ be the Iwahori
subgroup. It is the preimage of $B_0$ via the canonical map
$L^+G_0\ra G_0$. For $z\in\C^\ast$ the loop rotation $t\mapsto zt$
yields a $\C^\ast$-action on $LG_0$. Write
$$\widehat{LG}_0=\C^\ast\ltimes LG_0.$$ Let $G$ be the Kac-Moody
group associated with $\g$. It is a group ind-scheme which is a
central extension
$$1\ra\C^\ast\ra G\ra \widehat{LG}_0\ra 1,$$
see e.g., \cite[Section 13.2]{Ku2}. There is an obvious projection
$\pr:G\ra LG_0.$ We set
$$B=\pr^{-1}(I),\quad
Q=\pr^{-1}(L^+G_0),\quad T=\pr^{-1}(T_0).$$ Finally, let $N$ be the
prounipotent radical of $B$. We have
$$\g=\Lie(G),\quad
\frakb=\Lie(B),\quad\frakq=\Lie(Q),\quad\frakt=\Lie(T),\quad\frakn=\Lie(N).$$

\subsection{The affine flag variety}\label{ss:affineflag}

Let $X=G/B$ be the affine flag variety. It is a formally smooth
ind-scheme. The enhanced affine flag variety $X^\dag=G/N$ is a
$T$-torsor over $X$ via the canonical projection
\begin{equation}
\pi: X^\dag\ra X.
\end{equation}
The $T$-action on $X^\dag$ is given by $gN\mapsto gh^{-1}N$ for
$h\in T$, $g\in G$. The $T$-torsor $\pi$ is admissible, see the end of Section \ref{ss:compKT}. The ind-scheme $X^\dag$ is
also formally smooth. For any subscheme $Z$ of $X$ we will write
$Z^\dag=\pi^{-1}(Z)$. The $B$-orbit decomposition of $X$ is
\begin{equation*}
  X=\bigsqcup_{w\in\frakS}X_w,\qquad X_w=B\dot{w}B/B,
\end{equation*}
where $\dot{w}$ is a representative of $w$ in the normalizer of $T$ in
$G$. Each $X_w$ is an affine space of dimension $l(w)$. Its closure
$\olX_w$ is an irreducible projective variety. We have
\begin{equation*}
  \olX_w=\bigsqcup_{w'\leqs w}X_{w'},\qquad
  X=\ilim_w\olX_w.
\end{equation*}

\subsection{Localization theorem}\label{ss:functorgamma}

Recall the sheaf of differential operators $\scrD_{X^\dag}\in\hat\bfO(X^\dag)$.
The space of sections of $\scrD_{X^\dag}$ is defined as in (\ref{eq:gammaF}). The left action of $G$ on $X^\dag$ yields an algebra homomorphism
\begin{equation}\label{eq:g}
  \delta_l:\ \calU(\g)\ra\Gamma(X^\dag,\scrD_{X^\dag}).
\end{equation}
Since the $G$-action on $X^\dag$ commutes with the right $T$-action, the image of the map above lies in the $T$-invariant part of $\Gamma(X^\dag,\scrD_{X^\dag})$.
So for $\scrM\in\bfM^T(X^\dag)$ the $\scrD_{X^\dag}$-action on $\scrM$ given by (\ref{eq:daction}) induces a $\g$-action\footnote{More precisely, here by $\g$-action we mean the $\g$-action on the associated sheaf of vector spaces $(\scrM^{\dag})^{\C}$, see Step $1$ of the proof of Proposition \ref{prop:trans} for details.} on $\scrM^\dag$ via $\delta_l$. In particular the vector space
$\Gamma(\scrM)$ as defined in (\ref{eq:gammaM}) is a $\g$-module. Let $\bfM(\g)$ be the category of $\g$-modules. We say that a weight $\lam\in\frakt^\ast$ is \emph{antidominant} (resp.
\emph{dominant}, \emph{regular}) if for any $\al\in\Pi^+$ we have
$\pair{\lam:\al}\leqs 0$ (resp. $\pair{\lam:\al}\geqs 0$,
$\pair{\lam:\al}\neq 0$).

\begin{prop}\label{prop:BD}

  (a) The functor
  $$\Gamma:\bfM^\lam(X)\ra\bfM(\g),\quad\scrM\mapsto\Gamma(\scrM)$$
  is exact if $\lam+\rho$ is antidominant. %It is faithful if $\lam+\rho$ is
  %antidominant and regular.

  (b) The functor
  $$\Gamma:\bfM^{\tilde{\lam}}(X)\ra\bfM(\g),\quad\scrM\mapsto\Gamma(\scrM)$$
  is exact if $\lam+\rho$ is antidominant.
\end{prop}
\begin{proof}
A proof of part (a) is sketched in \cite[Theorem 7.15.6]{BD}. A
detailed proof can be given using similar technics as in the proof
of the Proposition \ref{prop:trans} below. This is left to the
reader. See also \cite[Theorem 2.2]{FG} for another proof of this
result. Now, let us concentrate on part (b). Let
$\scrM=(\al,\scrM_\al)$ be an object in $\bfM^{\tilde{\lam}}(X)$. By
definition the action of $\frakm_\lam$ on $\scrM^\dag$ is nilpotent.
Let $\scrM_n$ be the maximal subobject of $\scrM$ such that the
ideal $(\frakm_\lam)^n$ acts on $\scrM^\dag_n$ by zero. We have
$\scrM_{n-1}\subset\scrM_n$ and $\scrM=\ilim\scrM_n$. Write
$R^k\Gamma(X,-)$ for the $k$-th derived functor of the global
sections functor $\Gamma(X,-)$. Given $n\geqs 1$, suppose that
$$R^k\Gamma(X,\scrM^\dag_n)=0,\qquad\forall\,k>0.$$ Since
$\scrM_{n+1}/\scrM_n$ is an object of $\bfM^\lam(X)$, by part (a)
we have
$$R^k\Gamma(X,(\scrM_{n+1}/\scrM_n)^\dag)=0,\quad\forall\,k>0.$$ The
long exact sequence for $R\Gamma(X,-)$ applied to the short exact
sequence
$$0\lra\scrM^\dag_n\lra\scrM^\dag_{n+1}\lra(\scrM_{n+1}/\scrM_n)^\dag\lra 0$$
implies that $R^k\Gamma(X,\scrM^\dag_{n+1})=0$ for any $k>0$.
Therefore by induction the vector space
$R^k\Gamma(X,\scrM^\dag_{n})$ vanishes for any $n\geqs 1$ and $k>0$.
Finally, since the functor $R^k\Gamma(X,-)$ commutes with direct
limits, see e.g., \cite[Lemma B.6]{TT}, we have
$$R^k\Gamma(X,\scrM^\dag)=\ilim R^k\Gamma(X,\scrM^\dag_n)=0,\qquad\forall\,\, k>0.$$
\end{proof}

\subsection{The category $\tilde{\calO}_\kappa$ and Verma modules}\label{ss:verma}

For a $\frakt$-module $M$ and $\lam\in\frakt^\ast$ let
\begin{equation}\label{eq:genweight}
  M_{\tilde{\lam}}=\{m\in M\,|\,(h-\lam(h))^Nm=0,\ \forall\ h\in\frakt,\ N\gg 0\}.
\end{equation}
We call a $\frakt$-module $M$ a \textit{generalized weight module} if it satisfies the conditions
$$M=\bigoplus_{\lam\in{}_\kappa\!\frakt^\ast}M_{\tilde{\lam}},$$
$$\dim_\C M_{\tilde{\lam}}<\infty,\qquad\forall\ \lam\in\frakt^\ast.$$
Its character $\ch(M)$ is defined as the formal sum
\begin{equation}\label{eq:ch}
\ch(M)=\sum_{\lam\in\frakt^\ast}\dim_\C(M_{\tilde{\lam}})e^\lam.
\end{equation}
Let $\tilde{\calO}$ be the category consisting of the $\calU(\g)$-modules $M$ such that
\begin{itemize}
  \item[$\bullet$] as a $\frakt$-module $M$ is a generalized weight module,
  \item[$\bullet$] there exists a finite subset $\Xi\subset\frakt^\ast$ such that $M_{\tilde{\lam}}\neq 0$ implies that $\lam\in\Xi+\sum_{i=0}^{m-1}\Z_{\leqs 0}\al_i$.
\end{itemize}
\vskip2mm
The category $\tilde{\calO}$ is abelian. We define the duality functor $\bfD$ on $\tilde{\calO}$ by
\begin{equation}\label{eq:duality}
  \bfD M=\bigoplus_{\lam\in\frakt^\ast} \Hom(M_{\tilde{\lam}},\C),
\end{equation}
with the action of $\g$ given by the involution $\sigma$, see Section \ref{ss:jantverma}. Let $\tilde{\calO}_\kappa$ be the full subcategory of $\tilde{\calO}$ consisting of the $\g$-modules on which $\bfone-(\kappa-m)$ acts locally nilpotently. The category $\tilde{\calO}_\kappa$ is also abelian. It is stable under the duality functor, because $\sigma(\bfone)=\bfone$. The category $\calO_\kappa$ is a Serre
subcategory of $\tilde{\calO}_\kappa$.
For $\lam\in{}_\kappa\!\frakt^\ast$ we consider the \emph{Verma module}
$$N_\kappa(\lam)=\calU(\g)\otimes_{\calU(\frakb)}\C_\lam.$$
Here $\C_\lam$ is the one dimensional $\frakb$-module such that $\frakn$ acts trivially and $\frakt$ acts by $\lam$. It is an object of $\tilde{\calO}_\kappa$. Let $L_\kappa(\lam)$ be the unique simple quotient of $N_\kappa(\lam)$. We have $\bfD L_\kappa(\lam)=L_\kappa(\lam)$ for any $\lam$. A simple subquotient of a module $M\in\tilde{\calO}_\kappa$ is isomorphic to $L_\kappa(\lam)$ for some $\lam\in{}_\kappa\!\frakt^\ast$. The classes $[L_\kappa(\lam)]$ form a basis of the vector space $[\tilde{\calO}_\kappa]$, because the characters of the $L_\kappa(\lam)$'s are linearly independent.

Denote by $\Lam$ the set of integral weights in ${}_\kappa\!\frakt^\ast$. Let $\lam\in\Lam$ and $w\in\frakS$. Recall the line bundle $\scrL^\lam_{X_w}$ from Section \ref{ss:twistholonome}. Let
\begin{equation}\label{eq:Aw}
  \scrA^\lam_w=\Omega_{X^\dag_w}\otimes_{\scrO_{X^\dag_w}}
  \pi^\ast(\scrL_{X_w}^\lam).
\end{equation}
It is an object of $\bfM^\lam_h(X_w)$ and
$$\bbD(\scrA^\lam_w)=\scrA^\lam_w.$$
Let $i_w:X^\dag_w\ra X^\dag$ be the canonical embedding. It is locally closed and affine. We have the
following objects in $\bfM^\lam_h(X)$,
\begin{equation*}
  \scrA_{w!}^\lam=i_{w!}(\scrA^\lam_w), \quad \scrA_{w!\bullet}^\lam=i_{w!\bullet}(\scrA^\lam_w), \quad
  \scrA_{w\bullet}^\lam=i_{w\bullet}(\scrA^\lam_w)
\end{equation*}
We will consider the Serre subcategory $\bfM^\lam_0(X)$ of $\bfM_h^\lam(X)$ generated by the simple objects $\scrA^\lam_{w!\bullet}$ for $w\in\frakS$. It is an artinian category. Since $\bbD(\scrA^\lam_{w!\bullet})=\scrA^\lam_{w!\bullet}$, the category $\bfM^\lam_0(X)$ is stable under the duality. We have the following proposition.

\begin{prop}\label{prop:nonregular}
  Let $\lam\in \Lam$ such that $\lam+\rho$ is antidominant. Then
  \vskip2mm
  (a) $\Gamma(\scrA^\lam_{w!})=N_\kappa(w\cdot\lam),$
  \vskip2mm
  (b) $\Gamma(\scrA^\lam_{w\bullet})=\bfD N_\kappa(w\cdot\lam),$
  \vskip2mm
  (c) $\Gamma(\scrA^\lam_{w!\bullet})=\begin{cases}
     L_\kappa(w\cdot\lam) &\text{if $w$ is the shortest element in
     }w\frakS(\lam),\\
     0 &\text{else.}
   \end{cases}$
  \end{prop}

The proof of the proposition will be given in Appendix \ref{s:appendix}. It relies on results of Kashiwara-Tanisaki \cite{KT1} and uses translation functors for the affine category $\calO$.
\vskip3mm

\subsection{The parabolic Verma modules.}

Let $\QS$ be the set of the longest
representatives of the cosets $\frakS_0\backslash\frakS$. Let $w_0$
be the longest element in $\frakS_0$. Recall the following basic
facts.

\begin{lemma}\label{lem:longest}
  For $w\in\frakS$ if $w\cdot\lam\in\Lam^+$ for some $\lam\in\Lam$ with $\lam+\rho$ antidominant, then $w\in\QS$. Further, if $w\in\QS$ then we have

  (a) the element $w$ is the unique element $v$ in $\frakS_0w$ such that $\Pi^+_0\subset-v(\Pi^+)$,

  (b) for any $v\in\frakS_0$ we have $l(vw)=l(w)-l(v)$,

  (c) the element $w_0w$ is the shortest element in $\frakS_0w$.
\end{lemma}

The $Q$-orbit decomposition of $X$ is given by
$$X=\bigsqcup_{w\in{}^{\sss{Q}}\negmedspace\frakS}Y_w,\qquad Y_w=Q\dot{w}B/B.$$
Each $Y_w$ is a smooth subscheme of $X$, and $X_w$ is open and dense
in $Y_w$. The closure of $Y_w$ in $X$ is an irreducible projective variety of dimension $l(w)$ given by
\begin{equation*}
  \olY_w=\bigsqcup_{w'\in{}^{\sss{Q}}\negmedspace\frakS,\,w'\leqs w}Y_{w'}.
\end{equation*}
Recall that $Y^\dag_w=\pi^{-1}(Y_w)$. The canonical embedding $j_w: Y^\dag_w\ra X^\dag$ is locally closed and affine, see Remark \ref{rk:bordfunction}(b). For $\lam\in\Lam$ and $w\in\QS$ let
\begin{equation}\label{eq:Bw}
  \scrB^\lam_w=\Omega_{Y^\dag_w}\otimes_{\scrO_{Y^\dag_w}}\pi^\ast(\scrL_{Y_w}^\lam).
\end{equation}
We have the following objects in $\bfM^\lam_h(X)$
\begin{equation*}
  \scrB_{w!}^\lam=j_{w!}(\scrB^\lam_w), \quad \scrB_{w!\bullet}^\lam=j_{w!\bullet}(\scrB^\lam_w), \quad
  \scrB_{w\bullet}^\lam=j_{w\bullet}(\scrB^\lam_w).
\end{equation*}
Now, consider the canonical embedding $r:X_w^\dag\ra Y_w^\dag$. Since $r$ is open and affine, we have $r^\ast=r^!$ and the functors
$$(r_!, r^!=r^\ast, r_\bullet)$$
form a triple of adjoint functors between the categories $\bfM_h^{\tilde\lam}(Y_w)$ and $\bfM_h^{\tilde\lam}(X_w)$. Note that $r^\ast(\scrB^\lam_w)\cong\scrA^\lam_w$.
\begin{lemma}\label{lem:vermaquotient}
For $\lam\in\Lam$ and $w\in\QS$ the following holds.

  (a) The adjunction map $r_!r^\ast\ra\Id$ yields a surjective morphism in $\bfM_h^{\lam}(Y_w)$
  \begin{equation}\label{eq:eqsurjective}
    r_!(\scrA^\lam_w)\ra\scrB^\lam_w.
  \end{equation}

  (b) The adjunction map $\Id\ra r_\bullet r^\ast$ yields an injective morphism in $\bfM_h^{\lam}(Y_w)$
  \begin{equation}\label{eq:eqinjective}
    \scrB^\lam_w\ra r_\bullet(\scrA^\lam_w).
  \end{equation}
\end{lemma}
\begin{proof}
We will only prove part (b). Part (a) follows from (b) by applying the duality functor $\bbD$. To prove (b), it is enough to show that the $\scrO_{Y^\dag_w}$-module morphism
\begin{equation}\label{eq:oo}
(\scrB^\lam_w)^{\scrO}\ra \bigl(r_\bullet r^\ast (\scrB^\lam_w)\bigr)^{\scrO}
\end{equation} is injective.
Since $r$ is an open embedding, the right hand side is equal to $r_\ast r^\ast\bigl((\scrB^\lam_w)^{\scrO}\bigr).$ Now, consider the closed embedding $$i: Y^\dag_w - X^\dag_w\lra
Y^\dag_w.$$
The morphism (\ref{eq:oo}) can be completed into the following exact sequence in $\bfO(Y^\dag_w)$,
$$0\ra i_\ast i^!\bigl((\scrB^\lam_w)^{\scrO}\bigr)\ra(\scrB^\lam_w)^{\scrO}\ra r_\ast r^\ast\bigl((\scrB^\lam_w)^{\scrO}\bigr),$$
see e.g., \cite[Proposition 1.7.1]{HTT}. The $\scrO_{Y^\dag_w}$-module $(\scrB^\lam_w)^{\scrO}$ is locally free. So it has no subsheaf supported on the closed subscheme $Y^\dag_w - X^\dag_w$. We deduce that $i_\ast i^!\bigl((\scrB^\lam_w)^{\scrO}\bigr)=0$. Hence the morphism (\ref{eq:oo}) is injective.
\end{proof}
\begin{lemma}\label{lem:isosimple}
  For $\lam\in \Lam$ and $w\in \QS$ we have
  \begin{equation*}
    \scrA^\lam_{w!\bullet}\cong\scrB^\lam_{w!\bullet}.
  \end{equation*}
\end{lemma}
\begin{proof}
  By applying the exact functor $j_{w\bullet}$ to the map (\ref{eq:eqinjective}) we see that $\scrB^\lam_{w\bullet}$ is a subobject of
  $\scrA^\lam_{w\bullet}$ in $\bfM_h^\lam(X)$.
  In particular $\scrB^\lam_{w!\bullet}$ is a simple subobject of
  $\scrA^\lam_{w\bullet}$. So it is isomorphic to
  $\scrA^\lam_{w!\bullet}$.
\end{proof}

\begin{prop}\label{prop:localization}
Let $\lam\in\Lam$ such that $\lam+\rho$ is antidominant, and
let $w\in\QS$.

(a) If there exists $\al\in\Pi^+_0$ such that
$\pair{w(\lam+\rho):\al}=0$, then
$$\Gamma(\scrB_{w!}^\lam)=0.$$

(b) We have
$$\pair{w(\lam+\rho):\al}\neq 0,\quad\forall\ \al\in\Pi^+_0\quad \iff\quad
w\cdot\lam\in\Lam^+.$$ In this case, we have
\begin{equation*}
\Gamma(\scrB_{w!}^\lam)=M_\kappa(w\cdot\lam),\qquad
\Gamma(\scrB_{w\bullet}^\lam)=\bfD M_\kappa(w\cdot\lam).
\end{equation*}

(c) We have
\begin{equation*}
  \Gamma(\scrB^\lam_{w!\bullet})=\begin{cases}
     L_\kappa(w\cdot\lam) &\text{ if $w$ is the shortest element in
     }w\frakS(\lam),\\
     0 &\text{ else. }
  \end{cases}
\end{equation*}
\end{prop}
\begin{proof}
  The proof is inspired by the proof in the finite type case, see e.g., \cite[Theorem G.2.10]{Mi}. First, by Kazhdan-Lusztig's algorithm (see Remark \ref{rk:charpara}), the following equality holds in $[\bfM^\lam_0(X)]$
  \begin{equation}\label{eq:charr}
  [\scrB^\lam_{w!}]=\sum_{y\in\frakS_0}(-1)^{l(y)}[\scrA^\lam_{yw!}].
  \end{equation}
  Since $\lam+\rho$ is antidominant, the functor $\Gamma$ is exact on $\bfM^\lam_0(X)$ by Proposition \ref{prop:BD}(a). Therefore we have the following equalities in $[\tilde{\calO}_\kappa]$
\begin{eqnarray}
    [\Gamma(\scrB^\lam_{w!})]&=&\sum_{y\in\frakS_0}(-1)^{l(y)}[\Gamma(\scrA^\lam_{yw!})]\nonumber\\
    &=&\sum_{y\in\frakS_0}(-1)^{l(y)}[N_\kappa(yw\cdot\lam)].\label{eq:charpverma}
  \end{eqnarray}
Here the second equality is given by Proposition \ref{prop:nonregular}. Now, suppose that there exists $\al\in\Pi^+_0$ such that
  $\pair{w(\lam+\rho):\al}=0$. Let $s_\al$ be the corresponding reflection in
  $\frakS_0$. Then we have $$s_\al w\cdot\lam=w\cdot\lam\,.$$ By Lemma \ref{lem:longest}(b) we have $l(w)=l(s_\al w)+1$. So the right hand side of
  (\ref{eq:charpverma}) vanishes. Therefore we have $\Gamma(\scrB^\lam_{w!})=0$. This proves part (a).
Now, let us concentrate on part (b). Note that $\lam+\rho$ is antidominant. Thus by Lemma \ref{lem:longest}(a) we have $\pair{w(\lam+\rho):\al}\in\N$ for any $\al\in\Pi^+_0$. Hence
$$\pair{w(\lam+\rho):\al}\neq 0\quad\iff\quad\pair{w(\lam+\rho):\al}\geqs 1\quad\iff\quad\pair{w\cdot\lam:\al}\geqs 0.$$
By consequence $\pair{w(\lam+\rho):\al}\neq 0$ for all $\al\in\Pi^+_0$ if and only if $w\cdot\lam$ belongs to $\Lam^+$. In this case, the right hand side of (\ref{eq:charpverma}) is equal to $[M_\kappa(w\cdot\lam)]$ by the BGG-resolution. We deduce that
  \begin{equation}\label{eq:isogroth}
    [\Gamma(\scrB^\lam_{w!})]=[M_\kappa(w\cdot\lam)].
  \end{equation}
Now, applying the exact functor $j_!$ to the surjective morphism in (\ref{eq:eqsurjective}) yields a quotient map
$\scrB^\lam_{w!}\ra\scrA^\lam_{w!}$ in $\bfM^\lam(X)$. The exactness of $\Gamma$ implies that $\Gamma(\scrB^\lam_{w!})$ is a quotient
  of $N_\kappa(w\cdot\lam)=\Gamma(\scrA^\lam_{w!})$. Since $M_\kappa(w\cdot\lam)$ is the maximal $\frakq$-locally-finite quotient of
  $N_\kappa(w\cdot\lam)$ and $\Gamma(\scrB^\lam_{w!})$ is $\frakq$-locally finite, we deduce that $\Gamma(\scrB^\lam_{w!})$ is a quotient of
  $M_\kappa(w\cdot\lam)$. So the first equality in
  part (b) follows from (\ref{eq:isogroth}). The proof of the second one is similar. Finally, part (c) follows from Lemma \ref{lem:isosimple} and Proposition \ref{prop:nonregular}.
\end{proof}

\begin{rk}\label{rk:condition}
  Note that if $w\in\QS$ is a shortest element in $w\frakS(\lam)$, then we have $\pair{w(\lam+\rho):\al}\neq 0$ for all $\al\in\Pi^+_0$. Indeed, if there exists $\al\in\Pi^+_0$ such that $\pair{w(\lam+\rho):\al}=0$. Let $s'=w^{-1}s_\al w$. Then $s'$ belongs to $\frakS(\lam)$. Therefore we have $l(ws')>l(w)$. But $ws'=s_\al w$ and $s_\al\in\frakS_0$, by Lemma \ref{lem:longest} we have $l(ws')=l(s_\al w)<l(w)$. This is a contradiction.
\end{rk}

\section{The geometric construction of the Jantzen filtration}\label{s:geojant}

In this part, we give the geometric construction of the Jantzen filtration in the affine parabolic case by generalizing the result of \cite{BB}.

\subsection{Notation}\label{ss:overR}

Let $R$ be any noetherian $\C$-algebra. To any abelian category $\calC$ we
associate a category $\calC_R$ whose objects are the pairs
$(M,\mu_M)$ with $M$ an object of $\calC$ and
$\mu_M:R\ra\End_{\calC}(M)$ a ring homomorphism. A morphism
$(M,\mu_M)\ra(N,\mu_N)$ is a morphism $f:M\ra N$ in $\calC$ such
that $\mu_N(r)\circ f=f\circ\mu_M(r)$ for $r\in R$. The category
$\calC_R$ is also abelian. We have a faithful forgetful functor
\begin{equation}\label{eq:forget}
  for:\calC_R\ra\calC,\qquad (M,\mu_M)\ra M.
\end{equation} Any functor $F:
\calC\ra\calC'$ gives rise to a functor
$$F_R:\calC_R\ra\calC'_R,\quad(M,\mu_M)\mapsto(F(M),\,\mu_{F(M)})$$
such that $\mu_{F(M)}(r)=F(\mu_M(r))$ for $r\in R$. The functor
$F_R$ is $R$-linear. If $F$ is exact, then $F_R$ is also exact. We
have $for\circ F_R=F\circ for$. Given an inductive system of categories $(\calC_\al,i_{\al\beta})$, it yields an inductive
system $((\calC_\al)_R, (i_{\al\beta})_R)$, and we have a canonical
equivalence
$$(\tilim\calC_\al)_R=\tilim((\calC_\al)_R).$$

\subsection{The function $f_w$}\label{ss:ff}

Let $Q'=(Q,Q)$ be the commutator subgroup of $Q$. It acts
transitively on $Y_w^\dag$ for $w\in\QS$. We have the following lemma.

\begin{lemma}\label{lem:f}
  For any $w\in\QS$ there exists a regular function
  $f_w:\olY^\dag_w\ra\C$ such that $f_w^{-1}(0)=\olY_w^\dag-Y_w^\dag$ and
  \begin{equation*}
    f_w(qxh^{-1})=e^{w^{-1}\omega_0}(h)f_w(x),\qquad q\in
    Q',\ x\in Y^\dag_w,\ h\in T.
  \end{equation*}
\end{lemma}
\begin{proof}
  Let $V$ denote the simple $\g$-module of highest weight $\omega_0$.
  It is integrable, hence it admits an action of $G$. Let $v_0\in V$ be a nonzero vector in the weight space $V_{\omega_0}$. It is fixed under the action of $Q'$. So the map
  $$\varphi: G\ra V,\quad g\mapsto g^{-1}v_0$$
  maps $Q\dot{w}B$ to $B\dot{w}^{-1}v_0$ for any $w\in\QS$. Let
  $V(w^{-1})$ be the $\calU(\frakb)$-submodule of $V$ generated by the weight space
  $V_{w^{-1}\omega_0}$. We have $B\dot{w}^{-1}v_0\subset
  V(w^{-1})$. Recall that for $w'\in\QS$ we have
  \begin{eqnarray}
    w'<w&\Longleftrightarrow&(w')^{-1}<w^{-1}\nonumber\\
    &\Longleftrightarrow& (\dot{w}')^{-1}v_0\in\frakn
    \dot{w}^{-1}v_0,\label{eq:norder}
  \end{eqnarray}
  see e.g., \cite[Proposition 7.1.20]{Ku2}.
  Thus, if $w'\leqs w$ then $\varphi(Q\dot{w}'B)\subset V(w^{-1})$. The $\C$-vector space $V(w^{-1})$ is finite
  dimensional. We choose a linear form $l_w: V(w^{-1})\ra\C$
  such that $$l_w(\dot{w}^{-1}v_0)\neq
  0\quad\text{and}\quad l_w(\frakn \dot{w}^{-1}v_0)=0.$$
  Set $\tilde{f}_w=l_w\circ\varphi$. Then
  for $q\in Q'$, $h\in T$, $u\in N$ we have
  \begin{eqnarray*}
    \tilde{f}_w(q\dot{w}h^{-1}u)&=&l_w(u^{-1}h\dot{w}^{-1}v_0)\\
    &=&e^{w^{-1}\omega_0}(h)l_w(\dot{w}^{-1}v_0)\\
    &=&e^{w^{-1}\omega_0}(h)\tilde{f}_w(\dot{w}^{-1}).
  \end{eqnarray*}
  A similar calculation together with (\ref{eq:norder}) yields
  that $\tilde{f}_w(Q\dot{w}'B)=0$ for $w'<w$. Hence
  $\tilde{f}_w$ defines a regular function on $\bigsqcup_{w'\leqs
  w}Q\dot{w}'B$ which is invariant under the right action of
  $N$. By consequence it induces a regular function $f_w$ on
  $\olY^\dag_w$ which has the required properties.
\end{proof}

\begin{rk}\label{rk:bordfunction}
  (a) The function $f_w$ above is completely determined by its value on the point $\dot{w}N/N$,
  hence is unique up to scalar.

  (b) The lemma implies that the embedding $j_w: Y^\dag_w\ra X^\dag$ is affine.

  (c) The function $f_w$ is an analogue of the function defined in \cite[Lemma 3.5.1]{BB} in the finite type case. Below we will use it to define the Jantzen filtration on $\scrB^\lam_{w!}$. Note that \cite{BB}'s function is defined on the whole enhanced flag variety (which is a smooth scheme). Although our $f_w$ is only defined on the singular scheme $\olY^\dag_w$, this does not create any problem, because the definition of the Jantzen filtration is local (see Section \ref{ss:geojant}), and each point of $\olY_w^\dag$ admits a neighborhood $V$ which can be embedded into a smooth scheme $U$ such that $f_w$ extends to $U$. The choice of such an extension will not affect the filtration, see \cite[Remark 4.2.2(iii)]{BB}.
\end{rk}

\subsection{The $\scrD$-module $\scrB^{(n)}$}\label{ss:scrB}

Fix $\lam\in \Lam$ and $w\in\QS$. In the rest of Section \ref{s:geojant}, we will
abbreviate
$$j=j_w,\quad f=f_w,\quad\scrB=\scrB^\lam_w,\quad \scrB_!=\scrB^\lam_{w!},\quad \text{etc.}$$
Following \cite{BB} we introduce the deformed version of $\scrB$.
Recall that $R=\C[[s]]$ and $\wp$ is the maximal ideal. Let $x$ denote a coordinate on $\C$. For each integer $n>0$ set
$R^{(n)}=R(\wp^n)$. Consider the left $\scrD_{\C^\ast}$-module
$$\scrI^{(n)}=(\scrO_{\C^\ast}\otimes R^{(n)})x^s.$$
It is a rank one $\scrO_{\C^\ast}\otimes R^{(n)}$-module generated
by a global section $x^s$ such that the action of $\scrD_{\C^\ast}$
is given by $x\partial_x(x^s)=s(x^s)$. The restriction of $f$ yields a map $Y^\dag_w\ra\C^\ast$. Thus $f^\ast\scrI^{(n)}$ is a
left $\scrD_{Y^\dag_w}\otimes R^{(n)}$-module. So we get a right $\scrD_{Y^\dag_w}\otimes R^{(n)}$-module
$$\scrB^{(n)}=\scrB\otimes _{\scrO_{Y^\dag_w}}f^\ast\scrI^{(n)}.$$
\begin{lemma}\label{lem:basic}
  The right $\scrD_{Y^\dag_w}\otimes R^{(n)}$-module $\scrB^{(n)}$
  is an object of $\bfM^{\tilde\lam}_h(Y_w)$.
\end{lemma}
\begin{proof}
Since $R^{(n)}$ is a $\C$-algebra of dimension $n$ and $\scrB$ is
locally free of rank one over $\scrO_{Y^\dag_w}$, the
$\scrO_{Y^\dag_w}$-module $\scrB^{(n)}$ is locally free of rank $n$.
Hence it is a holonomic $\scrD_{Y^\dag_w}$-module. Note
that the $\scrD_{\C^\ast}$-module $\scrI^{(n)}$ is weakly $T$-equivariant such that $x^s$ is a $T$-invariant global section. Since the map $f$ is $T$-equivariant, the $\scrD_{Y^\dag_w}$-module $f^\ast\scrI^{(n)}$ is also weakly $T$-equivariant. Let $f^s$ be the global section of $f^\ast\scrI^{(n)}$ given by the image of $x^s$ under the inclusion
$$\Gamma(\C^\ast,\scrI^{(n)})\subset \Gamma(Y_w^\dag,f^\ast\scrI^{(n)}).$$
Then $f^s$ is $T$-invariant. It is nowhere vanishing on $Y^\dag_w$. Thus it yields an isomorphism of $\scrO_{Y^\dag_w}\otimes R^{(n)}$-modules
\begin{equation*}
 f^\ast\scrI^{(n)}\cong\scrO_{Y^\dag_w}\otimes R^{(n)}.
\end{equation*}
By consequence we have the following isomorphism
\begin{eqnarray*}
  (\scrB^{(n)})^\dag&=&
  \pi_\ast(\pi^\ast(\Omega_{Y_w}\otimes_{\scrO_{Y_w}}\scrL^\lam_{Y_w})\otimes_{\scrO_{Y^\dag_w}}f^\ast\scrI^{(n)})^T\\
  &=&\Omega_{Y_w}\otimes_{\scrO_{Y_w}}\scrL^\lam_{Y_w}\otimes_{\scrO_{Y_w}}\pi_\ast(f^\ast\scrI^{(n)})^T\\
  &\cong&\Omega_{Y_w}\otimes_{\scrO_{Y_w}}\scrL^\lam_{Y_w}\otimes_{\scrO_{Y_w}}(\scrO_{Y_w}\otimes
  R^{(n)}).
\end{eqnarray*}
See Remark \ref{rk:omega} for the first equality. Next, recall from (\ref{eq:taction}) that the right $T$-action on $Y^\dag_w$ yields a morphism of Lie algebras
$$\delta_r: \frakt\ra \Gamma(Y^\dag_w,\scrD_{Y^\dag_w}).$$
The right $\scrD^\dag_{Y_w}$-module structure of $f^\ast(\scrI^{(n)})$ is such that $$(f^s\cdot\delta_r(h))(m)=sw^{-1}\omega_0(-h)f^s(m),\quad\forall\ m\in Y^\dag_w.$$ So the action of the element
$$h+\lam(h)+sw^{-1}\omega_0(h)\in\calU(\frakt)\otimes R^{(n)}$$ on $(\scrB^{(n)})^\dag$ via the map $\delta_r$ vanishes. Since the multiplication by $s$ on $(\scrB^{(n)})^\dag$ is nilpotent, the action of the ideal $\frakm_\lam$ is also
nilpotent. Therefore $\scrB^{(n)}$ belongs to the category
$\bfM^{\tilde\lam}_h(Y_w)$.
\end{proof}
It follows from the lemma that we have the following objects in $\bfM^{\tilde{\lam}}_h(X)$
$$\scrB^{(n)}_!=j_!(\scrB^{(n)}),\quad
\scrB^{(n)}_{!\bullet}=j_{!\bullet}(\scrB^{(n)}),
\quad\scrB^{(n)}_\bullet=j_\bullet(\scrB^{(n)}).$$

\subsection{Deformed parabolic Verma modules}

Fix $\lam\in\Lam$ and $w\in\QS$ as before. Let $n>0$. If $w\cdot\lam\in\Lam^+$ we will abbreviate
$$M_\kappa=M_\kappa(w\cdot\lam),\quad M_\bfk=M_\bfk(w\cdot\lam),
\quad M^{(n)}_\bfk=M_\bfk(\wp^n),\quad \bfD M^{(n)}_\bfk=(\bfD
M_\bfk)(\wp^n).$$
Note that the condition is satisfied when $w$ is a shortest element in $w\frakS(\lam)$, see Remark \ref{rk:condition}.

\begin{prop}\label{prop:parvermadeformed}
Assume that $\lam+\rho$ is antidominant and that $w$ is a shortest element in $w\frakS(\lam)$. Then there are isomorphisms of
  $\g_{R^{(n)}}$-modules
  \begin{equation*}
    \Gamma(\scrB^{(n)}_!)=M^{(n)}_\bfk,\quad
    \Gamma(\scrB^{(n)}_\bullet)=\bfD M^{(n)}_\bfk.
  \end{equation*}
\end{prop}
The proof will be given in Appendix \ref{s:appendixb}.

\subsection{The geometric Jantzen filtration}\label{ss:geojant}

Now, we define the Jantzen filtration on $\scrB_!$ following \cite[Sections 4.1,4.2]{BB}. Recall that $\scrB^{(n)}$ is an object of $\bfM^{\tilde\lam}_h(Y_w)$. Consider the map $$\mu:R^{(n)}\ra
\End_{\bfM^{\tilde\lam}_h(Y_w)}(\scrB^{(n)}),\qquad
\mu(r)(m)=rm,$$ where $m$ denotes a local section of $\scrB^{(n)}$.
Then the pair $(\scrB^{(n)},\mu)$ is an object of the category
$\bfM^{\tilde\lam}_h(Y_w)_{R^{(n)}}$ defined in Section
\ref{ss:overR}. We will abbreviate $\scrB^{(n)}=(\scrB^{(n)},\mu)$.
Fix an integer $a\geqs 0$. Recall the morphism of functors
$\psi:j_!\ra j_\bullet$. We consider the morphism
$$\psi(a,n):\scrB^{(n)}_!\ra \scrB^{(n)}_\bullet$$ in the category
$\bfM^{\tilde\lam}_h(X)_{R^{(n)}}$ given by the composition of the chain of maps
\begin{equation}\label{eq:nearby}
\xymatrix@C=0.5cm{
   \scrB^{(n)}_! \ar[rr]^{j_!(\mu(s^a))} && \scrB^{(n)}_!
   \ar[rr]^{\psi(\scrB^{(n)})} && \scrB^{(n)}_\bullet}.
\end{equation}
The category $\bfM^{\tilde\lam}_h(X)_{R^{(n)}}$ is abelian. The
obvious projection $ R^{(n)}\ra R^{(n-1)}$ yields a canonical map
$$\Coker(\psi(a,n))\ra \Coker(\psi(a,n-1)).$$
By \cite[Lemma 2.1]{B} this map is an isomorphism when $n$ is
sufficiently large. We define
\begin{equation}\label{eq:pia}
\pi^a(\scrB)=\Coker(\psi(a,n)), \quad n\gg 0.
\end{equation}
This is an object of $\bfM^{\tilde\lam}_h(X)_{R^{(n)}}$. We
view it as an object of $\bfM^{\tilde\lam}_h(X)$ via the
forgetful functor (\ref{eq:forget}). Now, let us consider the maps
$$\al:\scrB_!\ra\pi^1(\scrB),\quad\beta:\pi^1(\scrB)\ra\pi^0(\scrB)$$
in $\bfM^{\tilde\lam}_h(X)$ given as follows. First, since
$$\pi^0(\scrB)=\Coker(\psi(\scrB^{(n)}))\quad\text{and}\quad
\pi^1(\scrB)=\Coker(\psi(\scrB^{(n)})\circ j_!(\mu(s))),$$
there is a canonical projection $\pi^1(\scrB)\ra\pi^0(\scrB)$. We define $\beta$ to be this map. Next, the morphism $\psi(\scrB^{(n)})$ maps
$j_!(s(\scrB^{(n)}))$ to $\Im(\psi(1,n))$. Hence it induces a
map $$j_!(\scrB^{(n)}/s(\scrB^{(n)}))\ra \pi^1(\scrB),\qquad
n\gg 0.$$ Composing it with the isomorphism
$\scrB\cong\scrB^{(n)}/s(\scrB^{(n)})$ we get the map $\al$. Let $\mu^1$ denote the
$R^{(n)}$-action on $\pi^1(\scrB)$. Then by \cite{B} the sequence
\begin{equation}\label{eq:exact1}
0 \lra \scrB_! \overset{\al}\lra \pi^1(\scrB) \overset{\beta}\lra
\pi^0(\scrB) \lra 0,
\end{equation}
is exact and $\al$ induces an
isomorphism
$$\scrB_!\ra\Ker(\mu^1(s):\pi^1(\scrB)\ra\pi^1(\scrB)).$$
The \emph{Jantzen filtration} of $\scrB_!$ is defined by
\begin{equation}\label{eq:jantgeo}
J^i(\scrB_!)=\Ker(\mu^1(s))\cap\Im(\mu^1(s)^i), \qquad\forall\,\, i\geqs 0.
\end{equation}

\subsection{Comparison of the Jantzen filtrations}\label{ss:gammajant}

Fix $\lam\in \Lam$ and $w\in\QS$. Consider the Jantzen filtration $(J^iM_\kappa)$ on $M_\kappa$ as defined in Section \ref{ss:jantverma}. The following proposition compares it with the geometric Jantzen filtration on $\scrB_!$.
\begin{prop}\label{prop:gammajant}
  Assume that $\lam+\rho$ is antidominant and that $w$ is a shortest element in $w\frakS(\lam)$. Then we have
  $$J^iM_\kappa=\Gamma(J^i\scrB_!),\quad\forall\,\, i\geqs 0.$$
\end{prop}
\begin{proof}
By Proposition \ref{prop:localization}(b) and Proposition \ref{prop:parvermadeformed} we have
$$\Gamma(\scrB_!)=M_\kappa,\quad\Gamma(\scrB^{(n)}_!)=M_\bfk^{(n)},\quad
\Gamma(\scrB^{(n)}_\bullet)=\bfD M_\bfk^{(n)}.$$
So the map
$$\phi^{(n)}=\Gamma(\psi(\scrB^{(n)})):\Gamma(\scrB^{(n)}_!)\ra\Gamma(\scrB^{(n)}_\bullet).$$
identifies with a $\g_{R^{(n)}}$-module homomorphism
$$\phi^{(n)}:M_\bfk^{(n)}\ra \bfD M_\bfk^{(n)}.$$ Consider the
projective systems $(M_\bfk^{(n)})$, $(\bfD M_\bfk^{(n)})$, $n>0$,
induced by the quotient map $R^{(n)}\ra R^{(n-1)}$. Their limits are
respectively $M_\bfk$ and $\bfD M_\bfk$. The morphisms $\phi^{(n)}$,
$n>0$, yield a morphism of $\g_R$-modules
$$\phi=\plim\phi^{(n)}:M_\bfk\ra \bfD M_\bfk$$
such that $$\phi(\wp)=\phi^{(1)}=\Gamma(\psi(\scrB)).$$
The functor $\Gamma$ is exact by Proposition \ref{prop:BD}. So
the image of $\phi(\wp)$ is $\Gamma(\scrB_{!\bullet})$. It is non zero by Proposition \ref{prop:localization}(c). Hence $\phi$
satisfies the condition of Definition \ref{df:jantzen} and we have
$$J^iM_\kappa=\big(\{x\in M_\bfk\,|\,\phi(x)\in s^i\bfD M_\bfk\}+sM_\bfk\big)/sM_\bfk.$$
By Lemma \ref{lem:MKsimple} and Remark \ref{rk:jantzen}
the map $\phi$ is injective. So the equality
above can be rewritten as
$$J^iM_\kappa=\big(\phi(M_\bfk)\cap
s^i\bfD M_\bfk+s\phi(M_\bfk)\big)/s\phi(M_\bfk).$$ Now, for $a\geqs 0$ let
$$\phi(a,n): M_\bfk^{(n)}\ra \bfD M_\bfk^{(n)}$$ be the
$\g_{R^{(n)}}$-module homomorphism given by the composition
\begin{equation}
\xymatrix{
   M_\bfk^{(n)} \ar[r]^{\mu(s^a)} & M_\bfk^{(n)}
   \ar[r]^{\phi^{(n)}} & \bfD M_\bfk^{(n)}.}
\end{equation}
Then we have $\Gamma(\psi(a,n))=\phi(a,n)$. Since $\Gamma$ is exact,
we have $$\Coker(\phi(a,n))=\Gamma\bigl(\Coker(\psi(a,n))\bigr).$$
So the discussion in Section \ref{ss:geojant} and the
exactness of $\Gamma$ yields that the canonical map $$\Coker(\phi(a,n))\ra
\Coker(\phi(a,n-1))$$ is an isomorphism if $n$ is large enough. We
deduce that
$$\bfD M_\bfk/s^aM_\bfk=\Coker(\phi(a,n))=\Gamma(\pi^a(\scrB)),\quad n\gg 0,$$
see (\ref{eq:pia}). The action of $\mu(s)$ on $\bfD M_\bfk/s^a\phi(M_\bfk)$ is
nilpotent, because $\mu(s)$ is nilpotent on $\bfD M^{(n)}_\bfk$. Further $\Gamma$ maps the exact sequence
(\ref{eq:exact1}) to an exact sequence
\begin{equation}\label{eq:exact2}
0\lra M_\kappa\lra \bfD M_\bfk/s\phi(M_\bfk)\lra \bfD
M_\bfk/\phi(M_\bfk)\lra 0,
\end{equation}
and the first map yields an isomorphism
$$M_\kappa=\Ker\big(\mu(s):\bfD M_\bfk/s\phi(M_\bfk)
\ra \bfD M_\bfk/s\phi(M_\bfk)\big).$$ Note that since $\bfD M_\bfk$
is a free $R$-module, for $x\in\bfD M_\bfk$ if $sx\in s\phi(M_\bfk)$
then $x\in\phi(M_\bfk)$. So by (\ref{eq:jantgeo}) and the exactness of $\Gamma$, we have for $i\geqs 0$,
\begin{eqnarray*}
  \Gamma(J^i\scrB_!)&=&\Ker(\mu(s))\cap\Im(\mu(s)^i)\\
  &=&\big(\phi(M_\bfk)\cap
s^i\bfD M_\bfk+s\phi(M_\bfk)\big)/s\phi(M_\bfk)\\
&=&J^iM_\kappa.
\end{eqnarray*}
The proposition is proved.
\end{proof}

\section{Proof of the main theorem}\label{s:proof}

\subsection{Mixed Hodge modules}

Let $Z$ be a smooth scheme. Let $\MHM(Z)$ be the category of mixed
Hodge modules on $Z$ \cite{Sa}. It is an abelian category. Each
object $\scrM$ of $\MHM(Z)$ carries a canonical filtration
$$W^\bullet\scrM=\cdots W^k\scrM\supset W^{k-1}\scrM\cdots,$$ called the \emph{weight filtration}. For each
$k\in\Z$ the \emph{Tate twist} is an auto-equivalence
$$(k):\MHM(Z)\ra\MHM(Z),\quad\scrM\mapsto\scrM (k)$$
such that $W^\bullet(\scrM (k))=W^{\bullet+2k}(\scrM)$. Let
$\Perv(Z)$ be the category of perverse sheaves on $Z$ with
coefficient in $\C$. There is an exact forgetful functor
$$\varrho:\MHM(Z)\ra\Perv(Z).$$
%For any irreducible object $\scrL\in\Perv(Z)$, up to a Tate twist
%there is at most one irreducible object $\tilde\scrL\in\MHM(Z)$ such
%that $\varrho(\tilde\scrL)=\scrL$.
For any locally closed affine embedding $i:Z\ra Y$ we have exact
functors
$$i_!,\,\,i_\bullet:\MHM(Z)\ra\MHM(Y)$$ which
correspond via $\varrho$ to the same named functors on the
categories of perverse sheaves.

If $Z$ is not smooth we embed it into a smooth variety $Y$ and we
define $\MHM(Z)$ as the full subcategory of $\MHM(Y)$ consisting of
the objects supported on $Z$. It is independent of the choice of the
embedding for the same reason as for $\scrD$-modules.

\subsection{The graded multiplicities of $\scrB^\lam_{x!\bullet}$ in $\scrB^\lam_{w!}$}\label{ss:hodge}

Now, let us calculate the multiplicities of a simple object
$\scrB^\lam_{x!\bullet}$ in the successive quotients of the Jantzen
filtration of $\scrB^\lam_{w!}$ for $x$, $w\in\QS$ with $x\leqs w$.

We fix once for all an element $v\in\frakS$, and we consider the Serre subcategory $\bfM^\lam_0(\olX_v)$ of
$\bfM^\lam_{h}(\olX_v)$ generated by the objects
$\scrA^\lam_{w!\bullet}$ with $w\leqs v$, $w\in\frakS$. The De Rham functor yields
an exact fully faithful functor
\begin{equation*}
  \DR:\bfM_{0}^\lam(\olX_v)\lra\Perv(\olX_v).
\end{equation*}
See e.g., \cite[Section 4]{KT1}. Let $\MHM_0(\olX_v)$ be the full
subcategory of $\MHM(\olX_v)$ consisting of objects whose image by
$\varrho$ belong to the image of the functor $\DR$. There exists a
unique exact functor
\begin{equation*}
  \eta:\MHM_0(\olX_v)\ra\bfM_{0}^\lam(\olX_v)
\end{equation*}
such that $\DR\circ\eta=\varrho$. An object $\scrM$ in $\MHM_0(\olX_v)$ is \emph{pure of weight} $i$ if we have $W^k\scrM/W^{k-1}\scrM=0$ for any $k\neq i$. For any $w\in\frakS$, $w\leqs v$, there is a unique simple object $\tilde\scrA^\lam_w$ in $\MHM(X_w)$ pure of weight $l(w)$ such that $\eta(\tilde\scrA^\lam_w)=\scrA^\lam_w$, see e.g. \cite{KT2}. Let
$$\tilde\scrA^\lam_{w!}=(i_w)_!(\tilde\scrA^\lam_w),\qquad
\tilde\scrA^\lam_{w!\bullet}=(i_w)_{!\bullet}(\tilde\scrA^\lam_w).$$
They are objects of $\MHM_0(\olX_v)$ such that
$$\eta(\tilde\scrA^\lam_{w!})=\scrA^\lam_{w!},\qquad\eta(\tilde\scrA^\lam_{w!\bullet})=\scrA^\lam_{w!\bullet}.$$

Now, assume that $w\in\QS$ and $w\leqs v$.
Recall that $\scrB^\lam_w\in\bfM^\lam(Y_w)$, and that
$\scrB^\lam_{w!}\in\bfM^\lam(X)$ can be viewed as an object of
$\bfM^\lam(\olX_v)$. We define similarly the objects
$\tilde\scrB^\lam_w\in\MHM(Y_w)$ and $\tilde\scrB^\lam_{w!}$,
$\tilde\scrB^\lam_{w!\bullet}\in\MHM_0(\olX_v)$ such that
$$\eta(\tilde\scrB^\lam_{w!})=\scrB^\lam_{w!},\qquad\eta(\tilde\scrB^\lam_{w!\bullet})=\scrB^\lam_{w!\bullet}.$$
The object $\tilde\scrB^\lam_{w!}$ has a canonical weight
filtration $W^\bullet$. We set $J^{k}\scrB^\lam_{w!}=\scrB^\lam_{w!}$ for $k<0$. The following proposition is due to Gabber
and Beilinson-Bernstein \cite[Theorem 5.1.2, Corollary 5.1.3]{BB}.

\begin{prop}\label{prop:gabber}
We have $\eta(W^{l(w)-k}\tilde\scrB^\lam_{w!})=J^{k}\scrB^\lam_{w!}$ in
$\bfM^\lam_0(\olX_v)$ for all $k\in\Z$.
\end{prop}

So the problem that we posed at the beginning of the section reduces to calculate the multiplicities of
$\tilde\scrB^\lam_{x!\bullet}$ in $\tilde\scrB^\lam_{w!}$ in the
category $\MHM_0(\olX_v)$. Let $q$ be a formal parameter. The Hecke
algebra $\scrH_q(\frakS)$ of $\frakS$ is a $\Z[q,q^{-1}]$-algebra
with a $\Z[q,q^{-1}]$-basis $\{T_w\}_{w\in\frakS}$ whose
multiplication is given by
$$T_{w_1}T_{w_2}=T_{w_1w_2},\qquad\text{if}\ l(w_1w_2)=l(w_1)+l(w_2),$$
$$(T_{s_i}+1)(T_{s_i}-q)=0, \quad 0\leqs i\leqs m-1.$$
On the other hand, the Grothendieck group $[\MHM_0(\olX_v)]$ is a
$\Z[q,q^{-1}]$-module such that
$$q^{k}[\scrM]=[\scrM(-k)],\qquad k\in\Z,\,\,\scrM\in\MHM_0(\olX_v).$$
For $x\in\frakS$ with $x\leqs v$ consider the closed embedding $$c_x:\pt\ra
\olX_v,\quad\pt\mapsto\dot{x}B/B.$$ There is an injective
$\Z[q,q^{-1}]$-module homomorphism, see e.g., \cite[(5.4)]{KT2},
\begin{eqnarray*}
  \Psi:[\MHM_0(\olX_v)]&\lra&\scrH_q(\frakS),\\
  {[\scrM]}&\longmapsto&\sum_{x\leqs v}\sum_{k\in\Z}(-1)^k[H^kc_x^\ast(\scrM)]T_x.
\end{eqnarray*}

The desired multiplicities are given by the following lemma.
\begin{lemma}\label{lem:simpar}
  For $w\in\QS$ we have
  \begin{equation*}
    \Psi([\tilde\scrB^\lam_{w!\bullet}])=\sum_{x\in\QS,x\leqs
    w}(-1)^{l(w)-l(x)}P_{x,w}\Psi([\tilde\scrB^\lam_{x!}]),
  \end{equation*}
  where $P_{x,w}\in\Z[q,q^{-1}]$ is the Kazhdan-Lusztig polynomial.
\end{lemma}
\begin{proof}
Since the choice for the element $v$ above is arbitrary, we may assume that $w\leqs v$. By the definition of $\Psi$ we have
\begin{equation}\label{eq:T_w}
  \Psi([\tilde\scrA^\lam_{w!}])=(-1)^{l(w)}T_w.
\end{equation}
By \cite{KL2}, \cite{KT1}, we have
\begin{equation}\label{eq:KL}
  \Psi([\tilde\scrA^\lam_{w!\bullet}])=(-1)^{l(w)}\sum_{x\in\frakS}P_{x,w}T_x.
\end{equation}
Next, for $x\in\QS$ with $x\leqs v$
we have
\begin{eqnarray}
  \Psi([\tilde\scrB^\lam_{x!}])&=&\sum_{y\in\frakS_0}\sum_{k\in\Z}(-1)^k[H^kc_{yx}^\ast(\tilde\scrB^\lam_{x!})]T_{yx}\nonumber\\
  &=&(-1)^{l(x)}\sum_{y\in\frakS_0}T_{yx}.\label{eq:charpara}
\end{eqnarray}
Since by Lemma \ref{lem:isosimple} we have
$$\tilde\scrA^\lam_{w!\bullet}=\tilde\scrB^\lam_{w!\bullet},$$
the following equalities hold
  \begin{eqnarray*}
    \Psi([\tilde\scrB^\lam_{w!\bullet}])&=&\Psi([\tilde\scrA^\lam_{w!\bullet}])\\
    &=&(-1)^{l(w)}\sum_{x\in{}^{\sss{Q}}\negmedspace\frakS,x\leqs w}\sum_{y\in\frakS_0}P_{yx,w}T_{yx}\\
    &=&(-1)^{l(w)}\sum_{x\in{}^{\sss{Q}}\negmedspace\frakS,x\leqs w}P_{x,w}\sum_{y\in\frakS_0}T_{yx}\\
    &=&\sum_{x\in{}^{\sss{Q}}\negmedspace\frakS,x\leqs
    w}(-1)^{l(w)-l(x)}P_{x,w}\Psi([\tilde\scrB^\lam_{x!}]).
  \end{eqnarray*}
Here the third equality is given by the well known identity:

$$P_{yx,w}=P_{x,w},\quad y\in\frakS_0,\
x\in\QS,\ x\leqs w.$$
\end{proof}

\begin{rk}\label{rk:charpara}
Let $x\in\QS$. Since $\Psi$ is injective, the equation
(\ref{eq:charpara}) yields that
\begin{equation*}
  [\tilde\scrB^\lam_{x!}]=\sum_{y\in\frakS_0}(-1)^{l(y)}[\tilde\scrA^\lam_{yx!}].
\end{equation*}
By applying the functor $\eta$ we get the following equality in
$[\bfM^\lam_0(X)]$
\begin{equation}\label{eq:char}
  [\scrB^\lam_{x!}]=\sum_{y\in\frakS_0}(-1)^{l(y)}[\scrA^\lam_{yx!}].
\end{equation}
\end{rk}

\medskip

\subsection{Proof of Theorem \ref{mthm}.}

Recall from (\ref{eq:calPn}) that we view $\calP_n$ as a subset of $\Lam^+$. By Corollary \ref{cor:jantid},
Theorem \ref{mthm} is a consequence of the following theorem.

\begin{thm}\label{thm2}
Let $\lam$, $\mu$ be partitions of $n$. Then for any negative integer $\kappa$ we have
\begin{equation}
d_{\lam'\mu'}(q)=\sum_{i\geqs 0}[J^iM_\kappa(\lam)/J^{i+1}M_\kappa(\lam):
L_\kappa(\mu)]q^i.
\end{equation}
Here $d_{\lam'\mu'}(q)$ is the polynomial defined in the introduction with $v=\exp(2\pi i/\kappa)$.
\end{thm}
\begin{proof}
Let $\nu\in\Lam$ such that $\nu+\rho$ is antidominant. We may assume that $\mu$, $\lam$ belong to the same orbit of $\nu$ under the dot action of $\frakS$, see (\ref{eq:blockdecomposition}).
For any $\mu\in\Lam^+\cap(\frakS\cdot\nu)$ let $w(\mu)_\nu$ be the shortest element in the set
$$w(\mu)_\nu\frakS(\nu)=\{w\in\frakS\,|\, \mu=w\cdot\nu\}.$$ Note
that $w(\mu)_\nu\frakS(\nu)$ is contained in
$\QS$ by Lemma \ref{lem:longest}. We fix $v\in\frakS$ such that $v\geqs w(\gamma)_\nu$ for any $\gamma\in\calP_n$. Let $q^{1/2}$ be a formal variable. We identify $q=(q^{1/2})^2$. Let $\tilde{P}_{x,w}$ be the Kazhdan-Lusztig polynomial normalized as follows $$P_{x,w}(q)=q^{(l(w)-l(x))/2}\tilde{P}_{x,w}(q^{-1/2}).$$
Let $\tilde{Q}_{x,w}$ be the inverse Kazhdan-Lusztig polynomial given by
$$\sum_{x\in\frakS}\tilde{Q}_{x,z}(-q)\tilde{P}_{x,w}(q)=\delta_{z,w},\quad z, w\in\frakS.$$
Then by (\ref{eq:T_w}), (\ref{eq:KL}) we have
$$[\tilde{\scrA}^\nu_{x!}]=\sum_{w\in\frakS}q^{(l(x)-l(w))/2}
\tilde{Q}_{x,w}(q^{-1/2})[\tilde{\scrA}^\nu_{w!\bullet}],\quad\forall\ x\in\frakS.$$
By Remark \ref{rk:charpara} we see that
\begin{equation}\label{eq:tildeq}
  [\tilde{\scrB}^\nu_{x!}]=\sum_{w\in\QS}\Bigl(\sum_{s\in\frakS_0}(-1)^{l(s)}
q^{(l(sx)-l(w))/2}\tilde{Q}_{sx,w}(q^{-1/2})\Bigr)[\tilde{\scrB}^\nu_{w!\bullet}],\quad\forall\ x\in\QS.
\end{equation}
Now, let
$$[\bfM^\nu_0(\olX_v)]_q=[\bfM^\nu_0(\olX_v)]\otimes_\Z\Z[q^{1/2},q^{-1/2}], \qquad[\tilde{\calO}_\kappa]_q=[\tilde{\calO}_\kappa]\otimes_\Z\Z[q^{1/2},q^{-1/2}].$$ We have a $\Z[q,q^{-1}]$-module homomorphism
\begin{eqnarray*}
  \vep: [\MHM_0(\olX_v)]&\lra&[\bfM^\nu_0(\olX_v)]_q,\\
  {[\scrM]}&\longmapsto&\sum_{i\in\Z}[\eta(W^i\scrM/W^{i-1}\scrM)]q^{i/2}.
\end{eqnarray*}
Note that $\vep([\tilde\scrB^\nu_{x!\bullet}])=q^{l(x)/2}[\scrB^\nu_{x!\bullet}]$ and by Proposition \ref{prop:gabber} we have
$$\vep([\tilde\scrB^\nu_{x!}])=
\sum_{i\in\N}[J^i\scrB^\nu_{x!}/J^{i+1}\scrB^\nu_{x!}]q^{(l(x)-i)/2},\quad
x\in\QS,\ x\leqs v.$$
Next, let
$$[M_\kappa(\lam)]_q=\sum_{i\in\N}[J^iM_\kappa(\lam)/J^{i+1}M_\kappa(\lam)]q^{-i/2}.$$
Then by Proposition \ref{prop:gammajant}, we have
\begin{equation*}
\Gamma\vep([\tilde\scrB^\nu_{w(\lam)_\nu!}])=q^{l(w(\lam)_\nu)/2}[M_\kappa(\lam)]_q.
\end{equation*}
On the other hand, by (\ref{eq:tildeq}) we  have
\begin{eqnarray*}
&&\Gamma\vep([\tilde\scrB^\nu_{w(\lam)_\nu!}])\\
&=&\sum_{w\in\QS}\Bigl(\sum_{s\in\frakS_0}(-1)^{l(s)}
q^{(l(sw(\lam)_\nu)-l(w))/2}\tilde{Q}_{sw(\lam)_\nu,w}(-q^{1/2})\Bigr)
\Gamma\vep[\tilde{\scrB}^\nu_{w!\bullet}]\\
&=&\sum_{\mu\in\calP_n}\Bigl(\sum_{s\in\frakS_0}(-1)^{l(s)}
q^{(l(sw(\lam)_\nu)-l(w(\mu)_\nu))/2}\tilde{Q}_{sw(\lam)_\nu,w(\mu)_\nu}(q^{-1/2})\Bigr)
q^{l(w(\mu)_\nu)/2}[L_\kappa(\mu)].
\end{eqnarray*}
Here in the second equality we have used Proposition \ref{prop:localization}(c) and the fact that $\calP_n$ is an ideal in $\Lam^+$. Note that $l(sw(\lam)_\nu)=l(w(\lam)_\nu)-l(s)$ for $s\in\frakS_0$ by Lemma \ref{lem:longest}(b). We deduce that
$$[M_\kappa(\lam)]_q=\sum_{\mu\in\calP_n}\Bigl(\sum_{s\in\frakS_0}(-q^{-1/2})^{l(s)}
\tilde{Q}_{sw(\lam)_\nu,w(\mu)_\nu}(q^{-1/2})\Bigr)[L_\kappa(\mu)].$$
By \cite[Proposition 5]{L}, we have
$$d_{\lam',\mu'}(q^{-1/2})=\sum_{s\in\frakS_0}(-q^{-1/2})^{l(s)}
\tilde{Q}_{sw(\lam)_\nu,w(\mu)_\nu}(q^{-1/2}),$$
see also the beginning of the proof of Proposition 6 in loc. cit., and \cite[Lemma 2.2]{LT2} for instance. We deduce that
  \begin{equation*}
    \sum_{i\in\N}[J^iM_\kappa(\lam)/J^{i+1}M_\kappa(\lam)]q^i=\sum_{\mu\in\calP_n}d_{\lam',\mu'}(q)[L_\kappa(\mu)].
  \end{equation*}
The theorem is proved.
\end{proof}

\medskip

\begin{rk}
The $q$-multiplicities of the Weyl modules $W_v(\lam)$ have also been considered in \cite{Ar} and \cite{RT}. Both papers are of combinatorial nature, and are very different from the approach used here. In \cite{Ar} Ariki defined a grading on the $q$-Schur algebra and he proved that the $q$-multiplicities of the Weyl module with respect to this grading is also given by the same polynomials $d_{\lam',\mu'}$. However, it not clear to us how to relate this grading to the Jantzen filtration.
\end{rk}
\begin{rk}
The \emph{radical filtration} $C^\bullet(M)$ of an object $M$ in an abelian category $\calC$ is given by putting $C^0(M)=M$ and $C^{i+1}(M)$ to be the radical of $C^i(M)$ for $i\leqs 0$. It follows from \cite[Lemma 5.2.2]{BB} and Proposition \ref{prop:gammajant} that the Jantzen filtration of $\scrB_!$ coincides with the radical filtration. If $\lam\in\Lam$ such that $\lam+\rho$ is antidominant and regular, then the exact functor $\Gamma$ is faithful, see \cite[Theorem 7.15.6]{BD}. In this case, we have $$\Gamma(C^\bullet(\scrB_!))=C^\bullet(\Gamma(\scrB_!))=C^\bullet M_\kappa(\lam).$$
So the Jantzen filtration on $M_\kappa(\lam)$ coincides with the radical filtration. If we have further $\lam\in\calP_n$ and $\kappa\leqs -3$, then by the equivalence in Proposition \ref{prop:equiv} we deduce that the Jantzen filtration of $W_v(\lam)$ also coincides with the radical filtration. This is compatible with recent result of Parshall-Scott \cite{PS}, where they computed the radical filtration of $W_v(\lam)$ under the same assumption of regularity here but without assuming $\kappa\leqs -3$. We conjecture that for any $\lam$ the Jantzen filtration on $M_\kappa(\lam)$ coincides with the radical filtration.
\end{rk}
\begin{rk}
The results of Sections \ref{s:localization}, \ref{s:geojant}, \ref{s:proof} hold for any standard parabolic subgroup $Q$ of $G$ with the same proof. In particular, it allows us to calculate the graded decomposition matrices associated with the Jantzen filtration of the parabolic Verma modules in more general cases.
\end{rk}

\appendix

\section{Kashiwara-Tanisaki's construction, translation functors and Proof of Proposition \ref{prop:nonregular}}\label{s:appendix}

The goal of this appendix is to prove Proposition \ref{prop:nonregular}. We first consider the case when $\lam+\rho$ is regular. In this case, the result is essentially due to Kashiwara and Tanisaki \cite{KT1}. However, the setting of loc. cit. is slightly different from the one used here. So we will first recall their construction and adapt it to our setting to complete the proof of the Proposition in the regular case. Next, we give a geometric construction of the translation functor for the affine category $\calO$ inspired from \cite{BG}, and apply it to deduce the result for singular blocks. We will use the same notation as in Section \ref{s:localization}.

\subsection{The Kashiwara affine flag variety}\label{ss:compKT}

Recall that $\Pi$ is the root system of $\g$ and $\Pi^+$ is the set
of positive root. Write $\Pi^-=-\Pi^+$. For $\al\in\Pi$ we write
$$\g_\al=\{x\in\g\,\,|\,\,[h,x]=\al(h)x,\quad\forall\ h\in\frakt\}.$$
For any subset $\Upsilon$ of $\Pi^+$, $\Pi^-$ we set respectively
$$\frakn(\Upsilon)=\bigoplus_{\al\in\Upsilon}\g_\al,\qquad\frakn^-(\Upsilon)=\bigoplus_{\al\in\Upsilon}\g_\al.$$ For
$\al=\sum_{i=0}^{m-1}h_i\al_i\in\Pi$ we write
$\hight(\al)=\sum_{i=0}^{m-1}h_i$ and for $l\in\N$ we set
$$\Pi^-_l=\{\al\in\Pi^-\,|\,\hight(\al)\leqs
-l\},\qquad\frakn^-_l=\frakn^-(\Pi^-_l).$$ Consider the group scheme
$L^-G_0=G_0(\C[[t^{-1}]])$. Let $B^-$ be the preimage of $B_0^-$ by
the map
$$L^-G_0\ra G_0,\qquad t^{-1}\mapsto 0,$$ where $B_0^-$ is the Borel
subgroup of $G_0$ opposite to $B_0$. Let $N^-$ be the prounipotent radical of $B^-$. Let $N^-_l\subset B^-$ be the
group subscheme given by
$$N^-_l=\plim_k\exp(\frakn^-_l/\frakn^-_k).$$

Let $\frakX$ be the Kashiwara affine flag variety, see \cite{K}. It
is a quotient scheme $\frakX=G_\infty/B$, where $G_\infty$ is a
coherent scheme with a locally free left action of $B^-$ and a
locally free right action of $B$. The scheme $\frakX$ is coherent,
prosmooth, non quasi-compact, locally of countable type, with a left
action of $B^-$. There is a right $T$-torsor
$$\pi: \frakX^\dag=G_\infty/N\ra\frakX.$$
For any subscheme $Z$ of $\frakX$ let $Z^\dag$ be its preimage by $\pi$. Let
$$\frakX=\bigsqcup_{w\in\frakS}\overset{\circ\,\,\,}{\frakX^w}.$$
be the $B^-$-orbit decomposition. Then $\frakX$ is covered by
the following open sets
$$\frakX^w=\bigsqcup_{v\leqs w}\overset{\circ\,\,\,}{\frakX^v}.$$
For each $w$ there is a canonical closed embedding
$\olX_w\ra\frakX^w$. Moreover, for any integer $l$ that is large
enough, the group $N^-_l$ acts locally freely on $\frakX^w$,
$\frakX^{w\dag}$, the quotients
$$\frakX^w_l=N^-_l\backslash\frakX^w,\qquad \frakX^{w\dag}_l=N^-_l\backslash\frakX^{w\dag}$$
are smooth schemes\footnote{For $l$ large enough the scheme $\frakX^w_l$ is separated (hence quasi-separated). To see this, one uses the fact that $\frakX^w$ is separated and applies \cite[Proposition C.7]{TT}. }, and the induced morphism
\begin{equation}\label{eq:closedembedding}
  \olX_w\ra\frakX^w_l
\end{equation} is a closed immersion. See \cite[Lemma 2.2.1]{KT1}. Further we have
$$\olX^\dag_w=\olX_w\times_{\frakX^w_l}\frakX^{w\dag}_l.$$
In particular, we get a closed
embedding of $\olX_w^\dag\ra\olX_w$ into the
$T$-torsor $\frakX^{w\dag}_l\ra\frakX^w_l.$ This implies that the
$T$-torsor $\pi: X^\dag\ra X$ is admissible. Finally, let
$$p_{l_1l_2}:\frakX^{w\dag}_{l_1}\ra \frakX^{w\dag}_{l_2},
\quad p_l:\frakX^{w\dag}\ra\frakX^{w\dag}_l,\quad l_1\geqs l_2$$ be
the canonical projections. They are affine morphisms.

\subsection{The category $\bfH^{\tilde{\lam}}(X)$.}

Fix $w, y\in\frakS$ with $y\geqs w$. For
$l_1\geqs l_2$ large enough, the functor
$$(p_{l_1l_2})_\bullet:\bfM^{\tilde{\lam}}_h(\frakX^y_{l_1},\olX_w)\ra\bfM^{\tilde{\lam}}_h(\frakX^y_{l_2},\olX_w)$$
yields a filtering projective system of categories, and we set
$$\bfH^{\tilde{\lam}}(\frakX^y,\olX_w)=\tplim_l\bfM^{\tilde{\lam}}_h(\frakX^y_l,\olX_w).$$
For $z\geqs y$ let $j_{yz}: \frakX^{y\dag}\ra\frakX^{z\dag}$ be the
canonical open embedding. It yields a map
$j_{yz}:\frakX^{y\dag}_l\ra\frakX^{z\dag}_l$ for each $l$. The
pull-back functors by these maps yield, by base change, a morphism of projective
systems of categories
$$(\bfM^{\tilde{\lam}}_h(\frakX^z_l,\olX_w))_l\ra(\bfM^{\tilde{\lam}}_h(\frakX^y_l,\olX_w))_l.$$ Hence we get a map
$$\bfH^{\tilde{\lam}}(\frakX^z,\olX_w)\ra\bfH^{\tilde{\lam}}(\frakX^y,\olX_w).$$
As $y$, $z$ varies these maps yield again a projective system of categories and we set
$$\bfH^{\tilde{\lam}}(\olX_w)=\tplim_{y\geqs w}\bfH^{\tilde{\lam}}(\frakX^y,\olX_w).$$ Finally,
for $w\leqs v$ the category $\bfH^{\tilde{\lam}}(\olX_w)$ is canonically
a full subcategory of $\bfH^{\tilde{\lam}}(\olX_{v})$. We define
$$\bfH^{\tilde{\lam}}(X)=\tilim_{w}\bfH^{\tilde{\lam}}(\olX_w).$$
This definition is inspired from \cite{KT1}, where the authors
considered the categories $\bfM^\lam_h(\frakX^y_l,\olX_w)$ instead of the categories $\bfM^{\tilde\lam}_h(\frakX^y_l,\olX_w)$. Finally, note that since the
category $\bfH^{\tilde{\lam}}(\olX_w)$ is equivalent to
$\bfM^{\tilde{\lam}}_h(\frakX^y_{l},\olX_w)$ for $y$, $l$ large
enough, and since the latter is equivalent to
$\bfM^{\tilde{\lam}}_h(\olX_w),$ see Section \ref{ss:dmodreminder},
we have an equivalence of categories
$$\bfH^{\tilde{\lam}}(X)\cong\bfM^{\tilde{\lam}}_h(X).$$

\subsection{The functors $\hat{\Gamma}$ and $\overline{\Gamma}$}

For an object $\scrM$ of $\bfH^{\tilde{\lam}}(X)$, there exists
$w\in\frakS$ such that $\scrM$ is an object of the subcategory
$\bfH^{\tilde{\lam}}(\olX_w)$. Thus $\scrM$ is represented by a
system $(\scrM^y_l)_{y\geqs w, l}$, with $\scrM^y_l\in
\bfM^{\tilde{\lam}}_h(\frakX^y_{l},\olX_w)$ and $l$ large enough.
For $l_1\geqs l_2$ there is a canonical map
$$(p_{l_1l_2})_\ast(\scrM^y_{l_1})\ra (p_{l_1l_2})_\bullet(\scrM^y_{l_1})=\scrM^y_{l_2}.$$
It yields a map (see (\ref{eq:gammadag}) for the notation)
$$\Gamma(\frakX^y_{l_1},\scrM^y_{l_1})\ra
\Gamma(\frakX^y_{l_2},\scrM^y_{l_2}).$$ Next, for $y$, $z\geqs w$
and $l$ large enough, we have a canonical isomorphism
$$\Gamma(\frakX^y_{l},\scrM^y_{l})=\Gamma(\frakX^{z}_{l},\scrM^{z}_{l}).$$
Following \cite{KT1}, we choose a $y\geqs w$ and we set
$$\hat\Gamma(\scrM)=\plim_l\Gamma(\frakX^y_{l},\scrM^y_{l}).$$
This definition does not depend on the choice of $w$, $y$. Now,
regard $\scrM$ as an object of $\bfM^{\tilde{\lam}}_h(X)$. Recall
the object $\scrM^\dag\in\bfO(X)$ from Section \ref{ss:twisted}.
Suppose that $\scrM^\dag$ is represented by a system
$(\scrM^\scrO_y)_{y\geqs w}$ with $\scrM^\scrO_y\in\bfO(\olX_y)$.
By definition we have $\scrM^\scrO_y=(i^!\scrM^y_l)^\dag$,
where $i$ denotes the closed embedding
$\olX^\dag_y\ra\frakX^{y\dag}_l$, see (\ref{eq:closedembedding}).
Therefore we have
\begin{eqnarray}
  \Gamma(\olX_y,\scrM^\scrO_y)&=&\Gamma(\olX_y,
  i^!(\scrM^y_l)^\dag)\nonumber\\
  &\subset&\Gamma(\frakX^y_l,\scrM^y_l).\label{eq:inclua}
\end{eqnarray}
Next, recall that we have
$$\Gamma(\scrM)=
\Gamma(X,\scrM^\dag)=\ilim_{y}\Gamma(\olX_y,\scrM^\scrO_y).$$ So by
first taking the projective limit on the right hand side of
(\ref{eq:inclua}) with respect to $l$ and then taking the inductive
limit on the left hand side with respect to $y$ we get an inclusion
$$\Gamma(\scrM)\subset\hat\Gamma(\scrM).$$
It identifies $\Gamma(\scrM)$ with the subset of $\hat\Gamma(\scrM)$
consisting of the sections supported on subschemes (of finite type)
of $X$.

The vector space $\hat\Gamma(\scrM)$ has a $\g$-action, see
\cite[Section 2.3]{KT1}. The vector space $\Gamma(\scrM)$ has also a
$\g$-action by Section \ref{ss:functorgamma}. The inclusion is
compatible with these $\g$-actions. Following loc.~cit., let
$$\overline{\Gamma}(\scrM)\subset\hat\Gamma(\scrM)$$
be the set of $\frakt$-finite elements. It is a $\g$-submodule of
$\hat\Gamma(\scrM)$.

\subsection{The regular case}\label{ss:proofKT} In this subsection we prove Proposition \ref{prop:nonregular} in the regular case. More precisely, we prove the following result.

\begin{prop}\label{prop:KT}
  Let $\lam\in\Lam$ be such that $\lam+\rho$ is antidominant and regular. Then we have isomorphisms of $\g$-modules
\begin{equation}\label{eq:KTiso}
  \Gamma(\scrA^\lam_{v!})=N_\kappa(v\cdot\lam),\quad
  \Gamma(\scrA^\lam_{v\bullet})=\bfD N_\kappa(v\cdot\lam),\quad
  \Gamma(\scrA^\lam_{v!\bullet})=L_\kappa(v\cdot\lam),\quad\forall\ v\in\frakS.
\end{equation}
\end{prop}
\begin{proof}
By \cite[Theorem 3.4.1]{KT1} under the assumption of the proposition we have isomorphisms of $\g$-modules.
\begin{equation*}
  \overline{\Gamma}(\scrA^\lam_{v!})=N_\kappa(v\cdot\lam),\qquad
  \overline{\Gamma}(\scrA^\lam_{v\bullet})=\bfD N_\kappa(v\cdot\lam),\qquad
  \overline{\Gamma}(\scrA^\lam_{v!\bullet})= L_\kappa(v\cdot\lam),\quad\forall v\in\frakS.
\end{equation*}
We must check that for $\sharp=!,\ \bullet,$ or $!\bullet$, the $\g$-submodules $\Gamma(\scrA^\lam_{v\sharp})$ and $\overline{\Gamma}(\scrA^\lam_{v\sharp})$ of $\hat\Gamma(\scrA^\lam_{v\sharp})$ are equal. Let us first prove this for $\sharp=\bullet$. We will do this in several steps.

\emph{Step 1.} Following \cite{KT1} we first define a particular section $\vartheta$ in $\hat\Gamma(\scrA_{v\bullet}^\lam).$ Let $\omega$ be a nowhere vanishing section of $\Omega_{X_{v}}$. It is unique up to a nonzero scalar. Let $t^\lam$ be the nowhere vanishing section of $\scrL^\lam_{X_{v}}$ such that $t^\lam(u\dot{v} b)=e^{-\lam}(b)$ for $u\in N$, $b\in B$. Then $\omega\otimes t^\lam$ is a nowhere vanishing section of $\scrA^{\lam,\dag}_{v}$ over $X_v$. Now, for $y\geqs v$ and $l$ large enough, let
$i^v_l:X_v\ra\frakX^{y}_l$ be the composition of the
locally closed embedding $X_v\ra\olX_y$ and the closed embedding
$\olX_y\ra\frakX^{y}_l$ in (\ref{eq:closedembedding}). We will denote the corresponding embedding $X^\dag_v\ra\frakX^{y\dag}_l$ again by $i^v_l$. Note that $\bigl(i^v_{l\bullet}(\scrA^\lam_v)\bigr)_l$ represents the object $\scrA^\lam_{v\bullet}$ in $\bfH^{\tilde\lam}(X).$ Therefore we have
$$\hat\Gamma(\scrA^\lam_{v\bullet})=
\plim_l\Gamma(\frakX^y_l,i^v_{l\bullet}(\scrA^\lam_v)).$$
Consider the canonical inclusion of $\scrO_{\frakX^y_l}$-modules
$$i^{v}_{l\ast}(\scrA^{\lam}_{v})^\dag\lra
i^{v}_{l\bullet}(\scrA^{\lam}_{v})^\dag.$$ Let
$\vartheta_l\in\Gamma(\frakX^y_l,
i^{v}_{l\bullet}(\scrA^{\lam}_{v})^\dag)$ be the image of
$\omega\otimes t^\lam$ under this map. The family $(\vartheta_l)$ defines an element
$$\vartheta\in \hat\Gamma(\scrA_{v\bullet}^{\lam}).$$
\vskip1mm
\emph{Step 2.} Let $V^y=y B^-\cdot B/B$. It is an affine open set in
$\frakX^y$. For $l$ large enough, let $V^y_l$ be the image
of $V^y$ in $\frakX^y_l$ via the canonical projection
$\frakX^y\ra\frakX^y_l$. Write $j^y_l: V^y_l\ra\frakX^y_l$ for the
inclusion. Note that $V^y_l\cong N^-/N^-_l$ as affine spaces. Therefore, if $l$ is large enough such that $\Pi^-_l\subset\Pi^-\cap v\Pi^-$, then the right $\scrD^\dag_{X_v}$-module structure on $\scrA_v^{\lam,\dag}$ yields an isomorphism of sheaves of $\C$-vector spaces over $V^y_l$
\begin{eqnarray*}
j^{y\ast}_l(i^{v}_{l\ast}\scrO_{X_{v}})\otimes
\calU(\frakn^-(\Pi^-\cap v(\Pi^-))/\frakn^-_l)&\simra& j^{y\ast}_l\bigl(i^{v}_{l\bullet}(\scrA_{v}^{\lam})^\dag\bigr),\\
f\otimes p&\mapsto& (\vartheta\cdot f)\cdot \delta_l(p).
\end{eqnarray*}
This yields an isomorphism of $\frakt$-modules
\begin{equation}\label{eq:9}
\Gamma(\frakX^y_l,
i^{v}_{l\bullet}(\scrA^{\lam}_v))
=\calU(\frakn^-/\frakn^-_l)\otimes
\C_{v\cdot\lam},
\end{equation}
see \cite[Lemma 3.2.1]{KT1}. By consequence we have an isomorphism of $\frakt$-modules
\begin{equation}\label{eq:hatgamman}
\hat\Gamma(\scrA^{\lam}_{v\bullet})
=(\plim_l\calU(\frakn^-/\frakn^-_l))\otimes
\C_{v\cdot\lam}.
\end{equation}

\emph{Step 3.} Now, let us prove
$\Gamma(\scrA_{v\bullet}^{\lam})=\overline{\Gamma}(\scrA_{v\bullet}^{\lam}).$ First, by (\ref{eq:9}) the space\\ $\Gamma(\frakX^y_l,i^v_{l\bullet}(\scrA^\lam_v))$ is $\frakt$-locally finite. So (\ref{eq:inclua}) implies that
$\Gamma(\scrA^{\lam}_{v\bullet})$ is the inductive limit of a system of $\frakt$-locally finite submodules. Therefore it is itself $\frakt$-locally finite. Hence we have
$$\Gamma(\scrA^{\lam}_{v\bullet})\subset \overline{\Gamma}(\scrA^{\lam}_{v\bullet}).$$
To see that this is indeed an equality, note that if
$m\in\hat\Gamma(\scrA^{\lam}_{v\bullet})$ is not $\frakt$-locally
finite, then by (\ref{eq:hatgamman}) the section $m$ is represented
by an element in
$$\plim_l\calU(\frakn^-/\frakn^-_l)\otimes
\C_{v\cdot\lam}$$ which does not come from
$\calU(\frakn^-)\otimes \C_{v\cdot\lam}$ via the
obvious map. Then one sees that $m$ can not be supported on a finite
dimensional scheme, i.e., it can not belong to
$\Gamma(\scrA^{\lam}_{v\bullet})$. This proves that
$$\Gamma(\scrA^{\lam}_{v\bullet})=
\overline{\Gamma}(\scrA^{\lam}_{v\bullet}).$$

\vskip2mm

Now, we can prove the rest two equalities in the proposition. Since $\lam+\rho$ is antidominant, by Proposition \ref{prop:BD}(a) the functor $\Gamma$ is exact on $\bfM^\lam(X)$. So $\Gamma(\scrA^{\lam}_{v!\bullet})$ is a $\g$-submodule of $\Gamma(\scrA^{\lam}_{v\bullet})$. Therefore all the elements in $\Gamma(\scrA^{\lam}_{v!\bullet})$ are $\frakt$-finite, i.e., we have
$$\Gamma(\scrA^{\lam}_{v!\bullet})\subset\overline{\Gamma}(\scrA^{\lam}_{v!\bullet}).$$
On the other hand, by \cite[Theorem 3.4.1]{KT1} we have
$$\overline{\Gamma}(\scrA^{\lam}_{v!\bullet})\subset\overline{\Gamma}(\scrA^{\lam}_{v\bullet}).$$
Therefore, Step 3 yields that each section in $\overline{\Gamma}(\scrA^{\lam}_{v!\bullet})$ is supported on a finite dimensional scheme, and hence belongs to $\Gamma(\scrA^{\lam}_{v!\bullet})$. We deduce that
\begin{equation}\label{eq:eqsimple}
  \Gamma(\scrA^{\lam}_{v!\bullet})
  =\overline{\Gamma}(\scrA^{\lam}_{v!\bullet}).
\end{equation}

Finally, since $\scrA^{\lam}_{v!}$ has a finite composition series whose constituents are given by $\scrA^{\lam}_{w!\bullet}$ for $w\leqs v$. Since both $\Gamma$ and $\overline{\Gamma}$ are exact functors on $\bfM^\lam_0(X)$, see Proposition \ref{prop:BD} and \cite[Corollary 3.3.3, Theorem 3.4.1]{KT1}. We deduce from (\ref{eq:eqsimple}) that $\Gamma(\scrA^{\lam}_{v!})$ is $\frakt$-locally finite, and the sections of $\overline{\Gamma}(\scrA^{\lam}_{v!})$ are supported on finite dimensional subschemes. Therefore we have
$$\Gamma(\scrA^{\lam}_{v!})=\overline{\Gamma}(\scrA^{\lam}_{v!}).$$
The proposition is proved.
\end{proof}

\subsection{Translation functors}

In order to compute the images of $\scrA^\lam_{v!}$ and $\scrA^\lam_{v\bullet}$ in the case when $\lam+\rho$ is not regular, we need the translation functors. For $\lam\in{}_\kappa\!\frakt^\ast$ such that $\lam+\rho$ is anti-dominant, we define $\tilde{\calO}_{\kappa,\lam}$ to be the Serre subcategory of $\tilde{\calO}_\kappa$ generated by
$L_\kappa(w\cdot\lam)$ for all $w\in\frakS$. The same argument as in the proof of \cite[Theorem 4.2]{DGK} yields that each $M\in\tilde{\calO}_\kappa$ admits a decomposition
\begin{equation}\label{eq:blockdecomposition}
M=\bigoplus M^\lam,\qquad M^\lam\in\tilde{\calO}_{\kappa,\lam},
\end{equation}
where $\lam$ runs over all the weights in ${}_\kappa\!\frakt^\ast$ such that $\lam+\rho$ is antidominant. The projection $$\pr_{\lam}:\tilde{\calO}_\kappa\ra\tilde{\calO}_{\kappa,\lam},\quad M\mapsto M^\lam,$$ is an exact functor. Fix two integral weights $\lam$, $\mu$ in $\frakt^\ast$ such that $\lam+\rho$, $\mu+\rho$ are antidominant and the integral weight $\nu=\lam-\mu$ is dominant. Assume that $\lam\in{}_\kappa\!\frakt^\ast$, then $\mu$ belongs to ${}_{\kappa'}\frakt^\ast$ for an integer $\kappa'<\kappa$. Let $V(\nu)$ be the simple $\g$-module
of highest weight $\nu$. Then for any $M\in\tilde{\calO}_{\kappa'}$ the module $M\otimes V(\nu)$ belongs to $\tilde{\calO}_\kappa$. Therefore we can define the following translation functor
$$\theta^{\nu}:\tilde{\calO}_{\kappa',\mu}\ra\tilde{\calO}_{\kappa,\lam},\quad M\mapsto\pr_{\lam}(M\otimes V(\nu)),$$
see \cite{Ku1}. Note that the subcategory $\tilde{\calO}_{\kappa,\lam}$ of $\tilde{\calO}$ is stable under the duality $\bfD$, because $\bfD$ fixes simple modules.
We have a canonical isomorphism of functors
\begin{equation}\label{eq:transdual}
  \theta^{\nu}\circ\bfD=\bfD\circ\theta^{\nu}.
\end{equation}
Indeed, it follows from (\ref{eq:duality}) that $\bfD(M\otimes V(\nu))=\bfD(M)\otimes\bfD(V(\nu))$ as $\g$-modules. Since $V(\nu)$ is simple, we have $\bfD V(\nu)=V(\nu)$. The equality (\ref{eq:transdual}) follows.

On the geometric side, recall the $T$-torsor $\pi: X^\dag\ra X$. For any integral weight $\lam\in\frakt^\ast$ the family
of line bundles $\scrL^\lam_{\olX_w}$ (see Section \ref{ss:twistholonome}) with $w\in\frakS$ form a
projective system of $\scrO$-modules under restriction, yielding a
flat object $\scrL^\lam$ of $\hat\bfO(X)$. Note that $\pi^\ast(\scrL^\lam)$ is a line bundle on $X^\dag$.
For integral weights $\lam$, $\mu$ in $\frakt^\ast$ the translation functor
$$\Theta^{\lam-\mu}: \bfM^\mu_0(X)\ra\bfM^\lam_0(X),\quad
\scrM\mapsto\scrM\otimes_{\scrO_{X^\dag}}\pi^\ast(\scrL^{\lam-\mu}),$$ is an equivalence
of categories. A quasi-inverse is given by $\Theta^{\mu-\lam}$. By the projection formula we have
\begin{equation}\label{eq:transstandard}
\Theta^{\lam-\mu}(\calA^\mu_{w\sharp})=\calA^\lam_{w\sharp},\quad\textrm{for}\quad
\sharp=!, \,\,!\bullet, \,\,\bullet.\end{equation}
Now, assume that $\mu+\rho$ is antidominant. Consider the exact functor
$$\Gamma:\bfM^\mu(X)\ra\bfM(\g),\quad\scrM\mapsto\Gamma(\scrM)$$
as in Proposition \ref{prop:BD}. Note that if $\mu+\rho$ is regular, then $\Gamma$ maps $\scrA^\mu_{v!\bullet}$ to $L_\kappa(v\cdot\mu)$ by Proposition \ref{prop:KT}. Since the subcategory $\tilde{\calO}_\kappa$ of $\bfM(\g)$ is stable under extension, the exact functor $\Gamma$ restricts to a functor
$$\Gamma:\bfM^\mu_0(X)\ra\tilde{\calO}_{\kappa,\mu}.$$
The next proposition is an
affine analogue of \cite[Proposition 2.8]{BG}.
\begin{prop}\label{prop:trans}
Let $\lam$, $\mu$ be integral weights in $\frakt^\ast$ such
that $\lam+\rho$, $\mu+\rho$ are antidominant and $\nu=\lam-\mu$ is
dominant. Assume further that $\mu+\rho$ is regular. Then the functors $$\theta^\nu\circ\Gamma:\ \bfM^\mu_0(X)\ra
\tilde{\calO}_{\kappa,\lam}\subset\bfM(\g)\quad\text{and}\quad \Gamma\circ\Theta^\nu:\ \bfM^\mu_0(X)\ra\bfM(\g)$$
are isomorphic.
\iffalse
Then the following diagram commutes up to a canonical
isomorphism of functors
\[\xymatrix{\bfM^{\mu}_0(X)\ar[r]^{\quad\Gamma}
\ar[d]_{\Theta^\mu_\lam} &\tilde{\calO}_{\kappa,\mu}\ar[d]^{\theta^\mu_\lam}\\
\bfM^\lam_0(X)\ar[r]^{\quad\Gamma} &\tilde{\calO}_{\kappa,\lam}.}\]\fi
\end{prop}
\begin{proof}
We will prove the proposition in several steps.

\emph{Step $1$.} First, we define a category $\Sh(X)$ of sheaves of $\C$-vector spaces on $X$ and we consider $\g$-modules in this category. To do this, for $w\in\frakS$ let $\Sh(\olX_w)$ be the category
of sheaves of $\C$-vector spaces on $\olX_w$. For $w\leqs x$ we have a closed embedding $i_{w,x}:\olX_w\ra\olX_x$, and an exact functor
$$i^!_{w,x}: \Sh(\olX_x)\ra \Sh(\olX_w),\quad\scrF\mapsto i^!_{w,x}(\scrF),$$ where $i^!_{w,x}(\scrF)$ is the
subsheaf of $\scrF$ consisting of the local sections supported
set-theoretically on $\olX_w$. We get a projective system of categories $$(\Sh(\olX_w),\ i^!_{w,x}).$$ Following \cite[7.15.10]{BD}
we define the category of sheaves of $\C$-vector spaces on $X$ to be the projective limit
$$\Sh(X)=\tplim\Sh(\olX_w).$$
This is an abelian category. By the same arguments as in the second
paragraph of Section \ref{ss:O-modules}, the category $\Sh(\olX_w)$
is canonically identified with a full subcategory of $\Sh(X)$, and each
object $\scrF\in\Sh(X)$ is a direct limit
$$\scrF=\ilim\scrF_w,\quad\scrF_w\in\Sh(\olX_w).$$
The space of global sections of an object of $\Sh(X)$ is given by
$$\Gamma(X,\scrF)=\ilim\Gamma(\olX_w,\scrF_w).$$
Next, consider the forgetful functor
$$\bfO(\olX_w)\ra\Sh(\olX_w),\quad\scrN\mapsto\scrN^\C.$$
Recall that for $\scrM\in\bfO(X)$ we have $\scrM=\ilim\scrM_w$ with
$\scrM_w\in\bfO(\olX_w)$. The tuple of sheaves of $\C$-vector spaces
$$\ilim_{x\geqs
w}i_{w,x}^!(\scrM_x^\C),\ w\in\frakS,$$ gives
an object in $\Sh(X)$. Let us denote it by $\scrM^\C$. The assignment
$\scrM\mapsto\scrM^\C$ yields a faithful exact functor
$$\bfO(X)\ra\Sh(X)$$
such that
\begin{equation}\label{eq:gamma}
\Gamma(X,\scrM)=\Gamma(X,\scrM^\C),
\end{equation}
see (\ref{eq:gammaindM}) for the definition of the left hand side. Now, let $\scrF=(\scrF_w)$ be an object in $\Sh(X)$. The vector spaces
$\End(\scrF_w)$ form a projective system via the maps
$$\End(\scrF_x)\ra\End(\scrF_w),\quad f\mapsto i_{w,x}^!(f).$$ We
set
\begin{equation}\label{eq:endM}
\End(\scrF)=\plim_w\End(\scrF_w).
\end{equation}
We say that an object $\scrF$ of $\Sh(X)$ is a \emph{$\g$-module} if it is equipped with
an algebra homomorphism $$\calU(\g)\ra\End(\scrF).$$ For instance,
for $\scrM\in\bfM^T(X^\dag)$ the object
$(\scrM^{\dag})^\C$ of $\Sh(X)$ is a $\g$-module via the
algebra homomorphism
\begin{equation}\label{eq:g'}
  \delta_l: \calU(\g)\ra\Gamma(X^\dag,\scrD_{X^\dag})
\end{equation}
See the beginning of Section \ref{ss:functorgamma}.

\vskip1mm

\emph{Step $2$.} Next, we define $G$-modules in $\hat\bfO(X)$. A \emph{standard parabolic subgroup} of $G$ is a group scheme of the form $P=Q\times_{G_0}P_0$ with $P_0$ a parabolic subgroup of $G_0$. Here the morphism $Q\ra G_0$ is the canonical one. We fix a subposet $\PS\subset\frakS$ such that for $w\in\PS$ the subscheme $\olX_w\subset X$ is stable under the $P$-action and
$$X=\ilim_{w\in\PS}\olX_w.$$ We say that an object $\scrF=(\scrF_w)$ of $\hat\bfO(X)$ has an \emph{algebraic $P$-action} if $\scrF_w$ has the structure of a $P$-equivariant quasicoherent $\scrO_{\olX_w}$-module for $w\in\PS$ and if the isomorphism $i_{w,x}^\ast\scrF_x\cong\scrF_w$ is $P$-equivariant for $w\leqs x$. Finally, we say that $\scrF$ is a \emph{$G$-module} if it is equipped with an action of the (abstract) group $G$ such that for any standard parabolic subgroup $P$, the $P$-action on $\scrF$ is algebraic.

We are interested in a family of $G$-modules $\scrV^i$ in $\hat\bfO(X)$ defined as follows. Fix a basis $(m_i)_{i\in\N}$ of $V(\nu)$ such that
each $m_i$ is a weight vector of weight $\nu_i$ and $\nu_j>\nu_i$
implies $j<i$. By assumption we have $\nu_0=\nu$. For each $i$ let
$V^i$ be the subspace of $V(\nu)$ spanned by the vectors $m_j$ for
$j\leqs i$. Then $$V^0\subset V^1\subset V^2\subset\dots$$ is a
sequence of $B$-submodule of $V(\nu)$. We write $V^\infty=V(\nu)$.
For $0\leqs i\leqs \infty$ we define a $\scrO_{\frakX}$-module
$\scrV^i_{\frakX}$ on $\frakX$ such that for any open set
$U\subset\frakX$ we have
$$\Gamma(U,\scrV^i_\frakX)=\{f: p^{-1}(U)\ra V^i\,\,|\,\,
f(gb^{-1})=bf(g),\quad g\in G_\infty,\ b\in B\},$$ where $p:G_\infty\ra\frakX$ is the quotient
map. Let $\scrV^i_w$ be the restriction of $\scrV^i_\frakX$ to $\olX_w$. Then $(\scrV^i_w)_{w\in\frakS}$ is a flat
$G$-module in $\hat\bfO(X)$. We will denote it by $\scrV^i$. Note that since $V(\nu)$ admits a $G$-action, the $G$-module
$\scrV^\infty\in\hat\bfO(X)$ is isomorphic to the $G$-module
$\scrO_X\otimes V(\nu)$ with $G$ acting diagonally. Therefore, for $\scrM\in\bfM^\mu_0(X)$ the projection formula yields a canonical isomorphism of vector spaces
\begin{eqnarray}
\Gamma(\scrM)\otimes V(\nu)&=&\Gamma(X,\scrM^\dag)\otimes V(\nu)\nonumber\\
&=&\Gamma(X,\scrM^\dag\otimes_{\scrO_X}\scrV^\infty).\label{eq:iso1}
\end{eqnarray}
On the other hand, we have
\begin{eqnarray}
\Gamma(\Theta^\nu(\scrM))&=&\Gamma\bigl(X,(\scrM\otimes_{\scrO_{X^\dag}}\pi^\ast(\scrL^\nu))^\dag\bigr)\nonumber\\
&=&\Gamma(X,\scrM^\dag\otimes_{\scrO_X}\scrL^\nu).\label{eq:iso2}
\end{eqnarray}
Our goal is to compare the $\g$-modules $\Gamma(\Theta^\nu(\scrM))$ and the direct factor $\theta^\nu(\Gamma(\scrM))$ of $\Gamma(\scrM)\otimes V(\nu)$. To this end, we first define in Step $3$ a $\g$-action on $(\scrM^\dag\otimes_{\scrO_X}\scrV^i)^\C$ for each $i$, then we prove in Steps $4$-$6$ that the inclusion
\begin{equation}\label{eq:split}
(\scrM^\dag\otimes_{\scrO_X}\scrL^\nu)^\C
\ra(\scrM^\dag\otimes_{\scrO_X}\scrV^\infty)^\C
\end{equation}
induced by the inclusion $\scrL^\nu=\scrV^0\subset\scrV^\infty$ splits as a $\g$-module homomorphism in $\Sh(X)$.

\emph{Step $3$.} Let $P$ be a standard parabolic subgroup of $G$, and let $\frakp$ be its Lie algebra.
Let $\PS\subset\frakS$ be as in Step $2$.
The $P$-action on $\scrV^i$ yields a Lie algebra homomorphism
$$\frakp\ra\End(\scrV^i_w),\quad\forall\ w\in\PS.$$
Consider the $\g$-action on $(\scrM^\dag)^\C$ given by the map $\delta_l$ in (\ref{eq:g'}). Note that for $w\leqs x$ in $\PS$, any element $\xi\in\frakp$ maps a local section of $\scrM^\dag_x$ supported on $\olX_w$ to a local section of $\scrM^\dag_x$ with the same property. In particular, for $w\in\PS$ we have a Lie algebra homomorphism
\begin{equation}\label{eq:paction}
 \frakp\ra\End\bigl((\scrM^\dag_w\otimes_{\scrO_{\olX_w}}\scrV^i_w)^\C\bigr),\quad\xi\mapsto\bigl(m\otimes v\mapsto\xi m\otimes v+m\otimes\xi v\bigr),
\end{equation}
where $m$ denotes a local section of $\scrM^\dag_w$, $v$ denotes a local section of $\scrV^i_w$. These maps are compatible with the restriction
$$\End\bigl((\scrM^\dag_x\otimes_{\scrO_{\olX_x}}\scrV^i_x)^\C\bigr)
\ra\End\bigl((\scrM^\dag_w\otimes_{\scrO_{\olX_w}}\scrV^i_w)^\C\bigr),\quad f\mapsto i_{w,x}^!(f).$$
They yield a Lie algebra homomorphism
$$\frakp\ra\End\bigl((\scrM^\dag\otimes_{\scrO_X}\scrV^i)^\C\bigr).$$
As $P$ varies, these maps glue together yielding a Lie algebra homomorphism
\begin{equation}
\g\ra\End\bigl((\scrM^\dag\otimes_{\scrO_X}\scrV^i)^\C\bigr).
\end{equation}
This defines a $\g$-action on $(\scrM^\dag\otimes_{\scrO_X}\scrV^i)^\C$ such that the obvious inclusions
$$(\scrM^\dag\otimes_{\scrO_X}\scrV^0)^\C\subset(\scrM^\dag\otimes_{\scrO_X}\scrV^1)^\C\subset\cdots$$
are $\g$-equivariant. So (\ref{eq:split}) is a $\g$-module homomorphism. Note that the flatness of $\scrV^i$ yields an
isomorphism in $\bfO(X)$
\begin{equation}\label{eq:5}
\scrM^\dag\otimes_{\scrO_X}\scrV^i/
\scrM^\dag\otimes_{\scrO_X}\scrV^{i-1}
\cong\scrM^\dag\otimes_{\scrO_X}\scrL^{\nu_i}.
\end{equation}

\emph{Step $4$.} In order to show that the $\g$-module homomorphism (\ref{eq:split}) splits, we consider the generalized Casimir operator of $\g$. Identify $\frakt$ and $\frakt^\ast$ via the pairing $\pair{\bullet:\bullet}$. Let $\rho\spcheck\in\frakt$ be the image of $\rho$. Let $h_i$ be a basis of $\frakt_0$, and let $h^i$ be its
dual basis in $\frakt_0$ with respect to the pairing $\pair{\bullet:\bullet}$. For $\xi\in\g_0$ and $n\in\Z$ we will abbreviate $\xi^{(n)}=\xi\otimes t^n$ and $\xi=\xi^{(0)}$.
The generalized Casimir operator is given by the formal sum
\begin{equation}\label{eq:casimir}
\frakC=2\rho\spcheck+\sum_ih^ih_i+2\partial\bfone+\sum_{i<j}e_{ji}e_{ij}+\sum_{n\geqs 1}\sum_{i\neq j}e_{ij}^{(-n)}e_{ji}^{(n)}
+\sum_{n\geqs 1}\sum_{i}h^{i,(-n)}h_i^{(n)},
\end{equation}
see e.g., \cite[Section 2.5]{Kac}. Let $\delta_l(\frakC)$ be the formal sum given by applying $\delta_l$ term by term to the right hand side of (\ref{eq:casimir}). We claim that $\delta_l(\frakC)$ is a well defined element in $\Gamma(X^\dag,\scrD_{X^\dag})$, i.e., the sum is finite at each point of $X^\dag$. More precisely, let
$$\Sigma=\{e_{ij}\,|\,i< j\}\cup\{e_{ji}^{(n)}, h_i^{(n)}\,|\,i\neq j,n\geqs 1\},$$ and let $e$ be the base point of $X^\dag$. We need to prove that the sets
$$\Sigma_g=\{\xi\in\Sigma\,|\,\delta_l(\xi)(ge)\neq 0\},\quad g\in G,$$
are finite. To show this, consider the adjoint action of $G$ on $\g$
$$\Ad: G\ra\End(\g),\quad g\mapsto\Ad_g,$$
and the $G$-action on $\scrD_{X^\dag}\in\hat\bfO(X^\dag)$ induced by the $G$-action on $X^\dag$. The map $\delta_l$ is $G$-equivariant with respect to these actions. So for $\xi\in\g$ and $g\in G$ we have
$$\delta_l(\xi)(ge)\neq 0\quad\iff\quad \delta_l(\Ad_{g^{-1}}(\xi))(e)\neq 0.$$
Further the right hand side holds if and only if $\Ad_{g^{-1}}(\xi)\notin\frakn$. Therefore
$$\Sigma_g=\{\xi\in\Sigma\,\,|\,\, \Ad_{g^{-1}}(\xi)\notin \frakn\}$$ is a finite set, the claim is proved. By consequence $\frakC$ acts on the $\g$-module $(\scrM^{\dag})^{\C}$ for any $\scrM\in\bfM^T(X^\dag)$. Next, we claim that the
action of $\frakC$ on the $\g$-module $(\scrM^\dag\otimes_{\scrO_{X}}\scrV^i)^\C$ is also well defined. It is enough to prove this for
$(\scrM\otimes_{\scrO_{X}}\scrV^\infty)^\C$. By (\ref{eq:paction}) the action of $\frakC$
on $(\scrM\otimes_{\scrO_{X}}\scrV^\infty)^\C$ is given by the
operator
$$\frakC\otimes 1+1\otimes\frakC-\sum_{n\in\Z, i\neq j}e_{ij}^{(-n)}\otimes e_{ji}^{(n)}-\sum_{n\in\Z,i}h^{i,(-n)}\otimes
h_i^{(n)}.$$ Since for both $\scrM^\dag$ and $\scrV^\infty$, at each point, there are only finitely many elements from $\Sigma$ which
act nontrivially on it, the action of $\frakC$ on the tensor product is
well defined.

\emph{Step $5$.} Now, let us calculate the action of $\frakC$ on
$(\scrM^\dag\otimes_{\scrO_X}\scrL^{\nu_i})^\C$. We have $$\Ad_{g^{-1}}(\frakC)=\frakC,\quad\forall\ g\in G.$$ Therefore the global section $\delta_l(\frakC)$ is $G$-invariant and its value at $e$ is
$$\delta_l(\frakC)(e)=
\delta_l(2\rho\spcheck+\sum_ih^ih_i+2\partial\bfone)(e).$$
On the other hand, the right $T$-action on $X^\dag$
yields a map
$$\delta_r:\frakt\ra\Gamma(X^\dag,\scrD_{X^\dag}).$$
Since the right $T$-action commutes with the left $G$-action, for any $h\in\frakt$ the global section $\delta_r(h)$ is $G$-invariant. We have $\delta_r(h)(e)=-\delta_l(h)(e)$ because the left and right $T$-actions on the
point $e$ are inverse to each other. Therefore the global sections
$\delta_l(\frakC)$ and $\delta_r(-2\rho\spcheck+\sum_ih^ih_i+2\partial\bfone)$ takes the same value at the point $e$. Since both of them are $G$-invariant, we deduce that
$$\delta_l(\frakC)=\delta_r(-2\rho\spcheck+\sum_ih^ih_i+2\partial\bfone).$$
Recall from Section \ref{ss:twisted} that for $\lam\in\frakt^\ast$ and $\scrM\in\bfM^\lam(X)$ the operator
$\delta_r(-2\rho\spcheck+\sum_ih^ih_i+2\partial\bfone)$ acts on
$\scrM^\dag$ by the scalar
$$-\lam(-2\rho\spcheck+\sum_ih^ih_i+2\partial\bfone)=||\lam+\rho||^2-||\rho||^2.$$
Therefore $\frakC$ acts on $\scrM^\dag$ by the same scalar. In
particular, for $\scrM\in\bfM^\mu(X)$ and $i\in\N$, the element $\frakC$ acts on
$\scrM^\dag\otimes_{\scrO_X}\scrL^{\nu_i}$ by
$||\mu+\nu_i+\rho||^2-||\rho||^2$. Note that the isomorphism (\ref{eq:5}) is compatible with the $\g$-actions. So $\frakC$ also acts by $||\mu+\nu_i+\rho||^2-||\rho||^2$ on $(\scrM^\dag\otimes_{\scrO_X}\scrV^i/
\scrM^\dag\otimes_{\scrO_X}\scrV^{i-1})^\C$.

\emph{Step $6$.} Now, we can complete the proof of the proposition. First, we claim that
\begin{equation}\label{eq:fact}
 ||\lam+\rho||^2-||\rho||^2=||\mu+\nu_i+\rho||^2-||\rho||^2\ \iff\ \nu_i=\nu.
\end{equation}
The ``if'' part is trivial. For the ``only
if'' part, we have by assumption
\begin{eqnarray*}
    ||\mu+\nu+ \rho||^2&=&||\mu+\nu_i+\rho||^2\\
    &=&||\mu+\nu+\rho||^2+||\nu-\nu_i||^2-2\pair{\mu+\nu+\rho:\nu-\nu_i}.
\end{eqnarray*}
Since $\nu-\nu_i\in\N\Pi^+$ and $\mu+\nu+\rho=\lam+\rho$ is
antidominant, the term $-2\pair{\lam+\rho:\nu-\nu_i}$ is positive.
Hence the equality implies that $||\nu-\nu_i||^2=0$. So $\nu-\nu_i$
belongs to $\N\delta$. But $\pair{\lam+\rho:\delta}=\kappa<0$. So we
have $\nu=\nu_i$. This proves the claim in (\ref{eq:fact}). A direct consequence of this claim and of Step $5$ is that the $\g$-module monomorphism (\ref{eq:split}) splits. It induces an isomorphism of $\g$-modules
\begin{equation}\label{eq:eq}
\Gamma(\scrM^\dag\otimes_{\scrO_X}\scrL^\nu)= \pr_\lam\Gamma(\scrM^\dag\otimes_{\scrO_X}\scrV^\infty),\quad\scrM\in\bfM^\mu_0(X).\end{equation}
Finally, note that the vector spaces isomorphisms (\ref{eq:iso1}) and (\ref{eq:iso2}) are indeed isomorphisms of $\g$-modules by the definition of the $\g$-actions on $(\scrM^\dag\otimes_{\scrO_X}\scrV^\infty)^\C$ and $(\scrM^\dag\otimes_{\scrO_X}\scrL^\nu)^\C$. Therefore (\ref{eq:eq}) yields an isomorphism of $\g$-modules
$$\Gamma(\Theta^\nu(\scrM))=\theta^\nu(\Gamma(\scrM)).$$
\end{proof}

\begin{rk}\label{rk:maptoO}
We have assumed $\mu+\rho$ regular in Proposition \ref{prop:trans} in order to have $\Gamma(\bfM^\mu_0(X))\subset\tilde{\calO}_{\kappa,\mu}$. It follows from Proposition \ref{prop:trans} that this inclusion still holds if $\mu+\rho$ is not regular. So Proposition \ref{prop:trans} makes sense without this regularity assumption, and the proof is the same in this case.
\end{rk}

\subsection{Proof of Proposition \ref{prop:nonregular}.}
  By Proposition \ref{prop:KT} it is enough to prove the proposition in the case when $\lam+\rho$ is not regular. Let $\omega_i$, $0\leqs i\leqs m-1$, be the fundamental weights in
  $\frakt^\ast$. Let
  $$\nu=\sum\omega_i,$$ where the sum runs over all $i=0,\ldots,m-1$ such that $\pair{\lam+\rho:\,\al_i}=0$. The weight $\nu$ is dominant. Let $\mu=\lam-\nu$. Then $\mu+\rho$ is an antidominant weight. It is moreover regular, because we have
  $$\pair{\mu+\rho:\al_i}=\pair{\lam+\rho:\al_i}-\pair{\nu:\al_i}<0,\quad 0\leqs i\leqs m-1.$$ Let $\kappa'=\pair{\mu+\rho:\delta}$. So  Propositions \ref{prop:KT}, \ref{prop:trans} and the equation (\ref{eq:transstandard}) implies that
  $$\Gamma(\scrA^\lam_{w!})=\theta^\nu(N_{\kappa'}(w\cdot\mu)),\quad
  \Gamma(\scrA^\lam_{w!\bullet})=\theta^\nu(L_{\kappa'}(w\cdot\mu)),\quad
  \Gamma(\scrA^\lam_{w\bullet})=\theta^\nu(\bfD N_{\kappa'}(w\cdot\mu)).$$
  So parts (a), (c) follow from the properties of the translation functor $\theta^\nu$ given in \cite[Proposition 1.7]{Ku1}. Part (b) follows from (a) and the equality (\ref{eq:transdual}).
\qed

\vskip5mm

\section{Proof of Proposition \ref{prop:parvermadeformed}}
\label{s:appendixb}

In this appendix, we prove Proposition \ref{prop:parvermadeformed}. We will first study localization of deformed Verma modules, see Lemmas \ref{lem:dualvermadeformed}, \ref{lem:vermadeformed}. Then we use them to deduce the parabolic version. In this appendix we will keep the notation of Section \ref{s:geojant}. In particular, recall that $j=j_w,\ f=f_w$, etc.
\vskip1mm
\subsection{Deformed Verma modules}\label{ss:dvm}

Fix $\lam\in\Lam$ and $w\in\QS$. Let $n>0$ be an integer. For $v=xw\in\frakS$ with $x\in\frakS_0$, let
$r_{x}:X^\dag_{v}\ra Y_w^\dag$ be the canonical inclusion. We have $i_{v}=j\circ r_x$. The tensor product
\begin{equation}\label{eq:notationA}
\scrA_{v}^{(n)}=\scrA^\lam_{v}\otimes_{\scrO_{X^\dag_{v}}}r_x^\ast f^\ast\scrI^{(n)}
\end{equation}
is equal to $r_x^\ast(\scrB^{(n)})$. By Lemma \ref{lem:basic} it is an object of $\bfM^{\tilde\lam}_h(X_{v})$. We consider the following objects in
$\bfM^{\tilde{\lam}}_h(X)$
$$\scrA_{v!}^{(n)}=i_{v!}(\scrA_{v}^{(n)}),
\quad\scrA_{v!\bullet}^{(n)}=i_{v!\bullet}(\scrA_{v}^{(n)}),\quad
\scrA_{v\bullet}^{(n)}=i_{v\bullet}(\scrA_{v}^{(n)}).$$

For
$\mu\in{}_\kappa\!\frakt^\ast$ we have defined the Verma module
$N_\kappa(\mu)$ in Section \ref{ss:verma}. The deformed Verma module is the $\calU_\bfk$-module given by
$$N_\bfk(\mu)=\calU(\g)\otimes_{\calU(\frakb)}R_{\mu+s\omega_0}.$$ Here the $\frakb$-module
$R_{\mu+s\omega_0}$ is a rank one $R$-module over which $\frakt$ acts
by $\mu+s\omega_0$, and $\frakn$ acts trivially. The deformed dual Verma module is
$$\bfD N_\bfk(\mu)=\bigoplus_{\lam\in{}_{\bfk}\frakt^\ast}\Hom_R(N_\bfk(\mu)_\lam, R),$$
see (\ref{eq:dualdeform}). We will abbreviate
$$N^{(n)}_\bfk(\mu)=N_\bfk(\mu)(\wp^n),
\quad \bfD N^{(n)}_\bfk(\mu)=\bfD N_\bfk(\mu)(\wp^n).$$ For any
$R^{(n)}$-module (resp. $R$-module) $M$ let $\mu(s^i):M\ra M$
be the multiplication by $s^i$ and write $s^iM$ for the image of
$\mu(s^i)$. We define a filtration
$$F^\bullet M=(F^0M\supset F^1M\supset F^2M\supset\ldots)$$
on $M$ by putting
$F^iM=s^iM.$ We say that it is of length $n$ if $F^nM=0$ and $F^{n-1}M\neq 0$. We set
$$\gr M=\bigoplus_{i\geqs 0}\gr^i M,\quad \gr^iM=F^iM/F^{i+1}M.$$
For any
$\g_{R^{(n)}}$-module $M$ let $\dch(M)$ be the
$\frakt_{R^{(n)}}$-module image of $M$ by the forgetful functor.

\begin{lemma}\label{lem:A1}
  If $\lam+\rho$ is antidominant then we have
  an isomorphism of $\frakt_{R^{(n)}}$-modules
  $$\dch(\Gamma(\scrA_{v\bullet}^{(n)}))= \dch(\bfD
N^{(n)}_\bfk(v\cdot\lam)).$$
\end{lemma}
\begin{proof}
The proof is very similar to the proof of Proposition \ref{prop:KT}. We will use the notation introduced there.

\emph{Step 1.} Consider the nowhere vanishing section $f^s$ of $(f^\ast\scrI^{(n)})^\dag$ over $Y_w$. Its restriction to $X_v$ yields an isomorphism
$$(\scrA_v^{(n)})^\dag\cong\Omega_{X_v}\otimes_{\scrO_{X_v}}\scrL^\lam_{X_v}\otimes R^{(n)}.$$
Let $\omega$ be a nowhere
vanishing section of $\Omega_{X_v}$, and let $t^\lam$ be the nowhere vanishing section of
$\scrL^\lam_{X_{v}}$ over $X_v$ such that $t^\lam(u\dot{v} b)=e^{-\lam}(b)$ for
$u\in N$, $b\in B$. Then the global section $\omega\otimes t^\lam\otimes f^s$ of $\scrA^{(n)}_v$ defines an element
$$\vartheta^s\in\hat\Gamma(\scrA_{v\bullet}^{(n)})$$
in the same way as $\vartheta$ is defined in the first step of the proof of Proposition \ref{prop:KT}.

\emph{Step 2.} Let us show that
$$\dch(\overline{\Gamma}(\scrA^{(n)}_{v\bullet}))=\dch(\bfD
N^{(n)}_\bfk(v\cdot\lam)).$$
The proof is the same as in the second step of the proof of Proposition \ref{prop:KT}. The right $\scrD_{X_v}^\dag$-module structure on $(\scrA_v^{(n)})^\dag$ yields an isomorphism of sheaves of $\C$-vector spaces over $V^y_l$
\begin{eqnarray*}
j^{y\ast}_l(i^{v}_{l\ast}\scrO_{X_{v}})\otimes
\calU(\frakn^-(\Pi^-\cap v(\Pi^-))/\frakn^-_l)\otimes
R^{(n)}&\simra& j^{y\ast}_l\bigl(i^{v}_{l\bullet}(\scrA_{v}^{(n)})^\dag\bigr),\\
f\otimes p\otimes r&\mapsto& \bigl((\vartheta^s\cdot f)\cdot \delta_l(p)\bigr)r.
\end{eqnarray*}
This yields an
isomorphism of $\frakt_{R^{(n)}}$-modules
\begin{equation*}
\dch\bigl(\Gamma(\frakX^y_l,
i^{v}_{l\bullet}(\scrA_v^{(n)}))\bigr)
=\dch\bigl(\calU(\frakn^-/\frakn^-_l)\otimes
R^{(n)}_{v\cdot\lam+s\omega_0}\bigr).
\end{equation*}
Therefore we have
\begin{equation}\label{eq:hatgamma}\dch(\hat\Gamma(\scrA^{(n)}_{v\bullet}))
=\dch\bigl((\plim_l\calU(\frakn^-/\frakn^-_l))\otimes
R^{(n)}_{v\cdot\lam+s\omega_0}\bigr),
\end{equation} and
\begin{eqnarray}
  \dch(\overline{\Gamma}(\scrA^{(n)}_{v\bullet}))&
=&\dch\bigl(\calU(\frakn^-)\otimes R^{(n)}_{v\cdot\lam+s\omega_0}\bigr)\nonumber\\
&=&\dch(\bfD
N^{(n)}_\bfk(v\cdot\lam)).\label{eq:4}
\end{eqnarray}

\emph{Step 3.} In this step, we prove that
$\Gamma(\scrA_{v\bullet}^{(n)})=\overline{\Gamma}(\scrA_{v\bullet}^{(n)})$
as $\g_{R^{(n)}}$-modules. Since both of them are $\g_{R^{(n)}}$-submodules of $\hat{\Gamma}(\scrA_{v\bullet}^{(n)})$. It is enough to prove that they are equal as vector spaces. Consider the filtration $F^\bullet(\scrA^{(n)}_v)$ on $\scrA^{(n)}_v$. It is a filtration in $\bfM^{\tilde\lam}(X_v)$ of length $n$ and $$\gr^i(\scrA^{(n)}_v)=\scrA^\lam_v,\quad 0\leqs i\leqs n-1.$$
Since $i_{v\bullet}$ is exact and $$R^i\Gamma(\scrA^\lam_{v\bullet})
=R^i\overline{\Gamma}(\scrA^\lam_{v\bullet})=0,\quad\forall\ i>0,$$
the functor $\Gamma\circ i_{v\bullet}$ commute with the filtration. Therefore both the filtrations $F^\bullet \Gamma(\scrA_{v\bullet}^{(n)})$ and $F^\bullet \overline{\Gamma}(\scrA_{v\bullet}^{(n)})$ have length $n$ and
$$\gr^i\Gamma(\scrA_{v\bullet}^{(n)})=\Gamma(\scrA_{v\bullet}^\lam),\quad
\gr^i\overline{\Gamma}(\scrA_{v\bullet}^{(n)})=\overline{\Gamma}(\scrA_{v\bullet}^\lam),\quad 0\leqs i\leqs n-1.$$ By Step 3 of the proof of Proposition \ref{prop:KT}, we have $\Gamma(\scrA^\lam_v)=\overline{\Gamma}(\scrA^\lam_v).$
We deduce that all the sections in $\Gamma(\scrA_{v\bullet}^{(n)})$ are $\frakt$-finite and all the sections in $\overline{\Gamma}(\scrA_{v\bullet}^{(n)})$ are supported on finite dimensional subschemes. This proves that
$$\Gamma(\scrA_{v\bullet}^{(n)})=\overline{\Gamma}(\scrA_{v\bullet}^{(n)}).$$
We are done by Step 2.
\end{proof}

\begin{lemma}\label{lem:dualvermadeformed}
If $\lam+\rho$ is antidominant there is an isomorphism of $\g_{R^{(n)}}$-modules
$$\Gamma(\scrA_{v\bullet}^{(n)})= \bfD
N^{(n)}_\bfk(v\cdot\lam).$$
\end{lemma}
\begin{proof}
Note that
\begin{eqnarray*}
\bfD N^{(n)}_\bfk(v\cdot\lam)&=&
\bigoplus_{\mu\in{}_\bfk\!\frakt^\ast}\Hom_{R}(N_\bfk(v\cdot\lam)_\mu, R)(\wp^n)\\
&=&\bigoplus_{\mu\in{}_\bfk\!\frakt^\ast}\Hom_{R^{(n)}}
(N^{(n)}_\bfk(v\cdot\lam)_\mu, R^{(n)}).
\end{eqnarray*}
For $\mu\in {}_\bfk\!\frakt^\ast$ let $\Gamma(\scrA_{v\bullet}^{(n)})_\mu$ be the weight space as defined in (\ref{eq:weight}). By Lemma \ref{lem:A1} we have
$$\Gamma(\scrA_{v\bullet}^{(n)})=\bigoplus_{\mu\in{}_\bfk\!\frakt^\ast}
\Gamma(\scrA_{v\bullet}^{(n)})_\mu,$$
because the same equality holds for $\bfD N^{(n)}_\bfk(v\cdot\lam)$. So we can consider the following $\g_{R^{(n)}}$-module
\begin{equation}\label{eq:10}
\bfD \Gamma(\scrA_{v\bullet}^{(n)})=\bigoplus_{\mu\in{}_\bfk\!\frakt^\ast}\Hom_{R^{(n)}}
(\Gamma(\scrA_{v\bullet}^{(n)})_\mu, R^{(n)}).
\end{equation}
It is enough to prove that we have an isomorphism of $\g_{R^{(n)}}$-modules
$$\bfD \Gamma(\scrA_{v\bullet}^{(n)})=
N^{(n)}_\bfk(v\cdot\lam).$$
By (\ref{eq:10}) we have
\begin{equation}\label{eq:11}
\dch\bigl(\bfD \Gamma(\scrA_{v\bullet}^{(n)})\bigr)=
\dch\bigl(\Gamma(\scrA_{v\bullet}^{(n)})\bigr).
\end{equation}
Together with Lemma \ref{lem:A1}, this yields an isomorphism of
$R^{(n)}$-modules
$$N^{(n)}_\bfk(v\cdot\lam)_{v\cdot\lam+s\omega_0}
=\bigl(\bfD\Gamma(\scrA_{v\bullet}^{(n)})\bigr)_{v\cdot\lam+s\omega_0}.$$ By the universal property of Verma modules, such an isomorphism induces a morphism of $\g_{R^{(n)}}$-module
$$\varphi:N^{(n)}_\bfk(v\cdot\lam)\ra\bfD \Gamma(\scrA_{v\bullet}^{(n)}).$$
We claim that for each $\mu\in{}_\bfk\!\frakt^\ast$ the
$R^{(n)}$-module morphism
$$\varphi_\mu:N^{(n)}_\bfk(v\cdot\lam)_\mu
\ra\bigl(\bfD\Gamma(\scrA_{v\bullet}^{(n)})\bigr)_\mu$$ given by the restriction of
$\varphi$ is invertible. Indeed, by Lemma \ref{lem:A1} and (\ref{eq:11}), we have
$$\ch(\bfD\Gamma(\scrA^{(n)}_{v\bullet}))=\ch(\bfD N^{(n)}_\bfk(v\cdot\lam))=\ch(N^{(n)}_\bfk(v\cdot\lam)).$$
So
$$N^{(n)}_\bfk(v\cdot\lam)_\mu
=\bigl(\bfD\Gamma(\scrA_{v\bullet}^{(n)})\bigr)_\mu$$
as $R^{(n)}$-modules. On the other hand, Proposition \ref{prop:nonregular} yields that the map
$$\varphi(\wp)=\varphi\otimes_{R^{(n)}}(R^{(n)}/\wp R^{(n)}):\
N_\kappa(v\cdot\lam)\ra\bfD\Gamma(\scrA^\lam_{v\bullet})$$ is an isomorphism
of $\g$-modules. So $\varphi_\mu(\wp)$ is also an isomorphism. Since the $R^{(n)}$-modules $N^{(n)}_\bfk(v\cdot\lam)_\mu$ and $\bigl(\bfD\Gamma(\scrA_{v\bullet}^{(n)})\bigr)_\mu$ are finitely generated, Nakayama's lemma implies that $\varphi_\mu$ is an isomorphism. So $\varphi$ is an isomorphism. The lemma is proved.
\end{proof}
\begin{lemma}\label{lem:vermadeformed}
If $\lam+\rho$ is antidominant and $v$ is a shortest element in $v\frakS(\lam)$, then there is an isomorphism of $\g_{R^{(n)}}$-modules
$$\Gamma(\scrA_{v!}^{(n)})=
N^{(n)}_\bfk(v\cdot\lam).$$
\end{lemma}
\begin{proof}
We abbreviate $\nu=v\cdot\lam$. The lemma will be proved in three steps.

\emph{Step 1.} Recall the character map from (\ref{eq:ch}).
Note that since $\Gamma$ and $i_!$ are exact, and
$$\gr \scrA_{v}^{(n)}=(\scrA_{v}^\lam)^{\oplus n},$$
we have an isomorphism of $\frakt$-modules
\begin{equation}\label{eq:grgamma}
\gr\Gamma(\scrA_{v!}^{(n)})=\Gamma(\scrA_{v!}^\lam)^{\oplus n}.
\end{equation}
Next, since the action of $s$ on $\Gamma(\scrA_{v!}^{(n)})$ is nilpotent, for any $\mu\in\frakt^\ast$ we have
\begin{equation*}
\dim_\C(\Gamma(\scrA_{v!}^{(n)})_{\tilde\mu})
=\dim_\C\bigl((\gr\Gamma(\scrA_{v!}^{(n)}))_{\tilde\mu}\bigr).
\end{equation*}
We deduce that as a $\frakt$-module $\Gamma(\scrA_{v!}^{(n)})$ is a generalized weight module and
\begin{equation}\label{eq:eqch}
  \ch(\Gamma(\scrA_{v!}^{(n)}))=\ch(\gr\Gamma(\scrA_{v!}^{(n)}))=n\ch \Gamma(\scrA_{v!}^\lam).
\end{equation}
On the other hand, we have the following isomorphism of $\frakt$-modules
\begin{equation}\label{eq:grverma}
\gr N^{(n)}_\bfk(\nu)=N_\kappa(\nu)^{\oplus n}.
\end{equation}
Therefore we have
\begin{equation*}
  \ch(N^{(n)}_\bfk(\nu))=n\ch(N_\kappa(\nu)).
\end{equation*}
Since $\Gamma(\scrA_{v!}^\lam)=N_\kappa(\nu)$ as $\g$-modules by Proposition \ref{prop:nonregular}, this yields
\begin{equation}
  \ch(\Gamma(\scrA_{v!}^{(n)}))=\ch(N^{(n)}_\bfk(\nu)).
\end{equation}
Further, we claim that there is an isomorphism of $R^{(n)}$-module \begin{equation}
\Gamma(\scrA_{v!}^{(n)})_{\tilde\mu}=N^{(n)}_\bfk(\nu)_{\tilde \mu},\quad\forall\ \mu\in\frakt^\ast.
\end{equation}
Note that $\Gamma(\scrA_{v!}^{(n)})_{\tilde\mu}$ is indeed an $R^{(n)}$-module because the action of $s$ on $\Gamma(\scrA_{v!}^{(n)})$ is nilpotent. To prove the claim, it suffices to notice that for any finitely generated $R^{(n)}$-modules $M$, $M'$ we have that $M$ is isomorphic to $M'$ as $R^{(n)}$-modules if and only if $\gr^iM=\gr^{i}M'$ for all $i$.
So the claim follows from the isomorphisms of $\frakt$-modules (\ref{eq:grgamma}), (\ref{eq:grverma}) and Proposition \ref{prop:nonregular}.

\emph{Step 2.} In this step, we prove that as a $\frakt_{R^{(n)}}$-module
$$\Gamma(\scrA_{v!}^{(n)})_{\tilde\nu}=R^{(n)}_{\nu+s\omega_0}$$
where $R^{(n)}_{\nu+s\omega_0}$ is the rank one $R^{(n)}$-module over which $\frakt$ acts by the weight $\nu+s\omega_0$. Let us consider the canonical morphisms in $\bfM^{\tilde\lam}_h(X)$
\[\xymatrix{\scrA^{(n)}_{v!}\ar@{>>}[r] &\scrA^{(n)}_{v!\bullet}\ar@{^{(}->}[r]
&\scrA^{(n)}_{v\bullet}.}\]
Since $\Gamma$ is exact on $\bfM^{\tilde\lam}(X)$, we deduce the following chain of $\g_{R^{(n)}}$-module morphisms
\[\xymatrix{\Gamma(\scrA^{(n)}_{v!})\ar@{>>}[r]^{\al} &\Gamma(\scrA^{(n)}_{v!\bullet})\ar@{^{(}->}[r]^{\beta}
&\Gamma(\scrA^{(n)}_{v\bullet}).}\]
Consider the following $\frakt_{R^{(n)}}$-morphisms given by the restrictions of $\al$, $\beta$
\[\xymatrix{\Gamma(\scrA^{(n)}_{v!})_{\tilde\nu}\ar@{>>}[r]^{\al_\nu} &\Gamma(\scrA^{(n)}_{v!\bullet})_{\tilde\nu}\ar@{^{(}->}[r]^{\beta_\nu}
&\Gamma(\scrA^{(n)}_{v\bullet})_{\tilde\nu}.}\]
We claim that $\al_\nu$ and $\beta_\nu$ are isomorphisms.
Note that by (\ref{eq:eqch}) we have
$$\dim_\C\Gamma(\scrA^{(n)}_{v!})_{\tilde\nu}=
n\dim_{\C}\Gamma(\scrA^{\lam}_{v!})_{\tilde\nu}=n.$$
By Lemma \ref{lem:dualvermadeformed} we also have $\dim_\C\Gamma(\scrA^{(n)}_{v\bullet})_{\tilde\nu}=n$.
Next, consider the exact sequence in $\bfM^{\tilde\lam}_h(X_v)$,
$$0\ra F^{i+1}\scrA^{(n)}_v\ra F^i\scrA^{(n)}_v\ra \gr^i\scrA^{(n)}_v\ra 0.$$
Applying the functor $i_{v!\bullet}$ to it yields a surjective morphism $$i_{v!\bullet}(F^i\scrA^{(n)}_v)/i_{v!\bullet}(F^{i+1}\scrA^{(n)}_v)\ra i_{v!\bullet}(\gr^i\scrA^{(n)}_v).$$
Since $i_{v!\bullet}(F^i\scrA^{(n)}_v)=F^i(i_{v!\bullet}(\scrA^{(n)}_v))$ and $\gr^i\scrA^{(n)}_v=\scrA^{\lam}_v$, we deduce a surjective morphism
$$\gr^i\scrA^{(n)}_{v!\bullet}\ra\scrA^{\lam}_{v!\bullet},\quad 0\leqs i\leqs n-1.$$
Applying the exact functor $\Gamma$ to this morphism and summing over $i$ gives a surjective morphism of $\g$-modules
$$\gamma:\gr \Gamma(\scrA^{(n)}_{v!\bullet})\ra\Gamma(\scrA^{\lam}_{v!\bullet})^{\oplus n}.$$
Since $v$ is minimal in $v\frakS(\lam)$, by Proposition \ref{prop:nonregular}(c) the right hand side is equal to $L_\kappa(\nu)$. We deduce from the surjectivity of $\gamma$ that
\begin{eqnarray*}
\dim_{\C}\Gamma(\scrA^{(n)}_{v!\bullet})_{\tilde\nu}
&=&\dim_{\C}\gr \Gamma(\scrA^{(n)}_{v!\bullet})_{\tilde\nu}\\
&\geqs&\dim_{\C}(L_\kappa(\nu)_{\tilde\nu})^{\oplus n}\\
&=&n
\end{eqnarray*}
It follows that the epimorphism $\al_\nu$ and the monomorphism $\beta_\nu$ are isomorphisms. The claim is proved. So we have an isomorphisms of $\frakt_{R^{(n)}}$-modules
$$\beta_\nu\circ\al_\nu:\Gamma(\scrA^{(n)}_{v!})_{\tilde\nu}\ra \Gamma(\scrA^{(n)}_{v\bullet})_{\tilde\nu}.$$
In particular, we deduce an isomorphisms of $\frakt_{R^{(n)}}$-modules
$$\Gamma(\scrA^{(n)}_{v!})_{\nu+s\omega_0}\ra \Gamma(\scrA^{(n)}_{v\bullet})_{\nu+s\omega_0},$$
because $$\Gamma(\scrA^{(n)}_{v\sharp})_{\nu+s\omega_0}\subset
\Gamma(\scrA^{(n)}_{v\sharp})_{\tilde\nu},\quad \text{for}\ \sharp=!,\bullet.$$
By Lemma \ref{lem:dualvermadeformed}, we have $$
\Gamma(\scrA^{(n)}_{v\bullet})_{\nu+s\omega_0}=R^{(n)}_{\nu+s\omega_0}.$$
We deduce an isomorphism of $\frakt_{R^{(n)}}$-modules
$$
\Gamma(\scrA^{(n)}_{v!})_{\nu+s\omega_0}=R^{(n)}_{\nu+s\omega_0}.$$
\emph{Step 3.} By the universal property of Verma modules and Step $2$, there exists a morphism of $\g_{R^{(n)}}$-modules
$$\varphi:N^{(n)}_\bfk(\nu)\ra \Gamma(\scrA^{(n)}_{v!}).$$
For any $\mu\in\frakt^\ast$ this map restricts to a morphism of $R^{(n)}$-modules
$$\varphi_{\mu}:N^{(n)}_\bfk(\nu)_{\tilde\mu}\ra \Gamma(\scrA^{(n)}_{v!})_{\tilde\mu}.$$
By Step 1, the $R^{(n)}$-modules on the two sides are finitely generated and they are isomorphic. Further, the induced morphism
$$\varphi(\wp):N_\kappa(\nu)\ra\Gamma(\scrA^\lam_{v!})$$
is an isomorphism by Proposition \ref{prop:nonregular}. So by Nakayama's lemma, the morphism $\varphi_\mu$ is an isomorphism for any $\mu$. Therefore $\varphi$ is an isomorphism. The lemma is proved.
\end{proof}
\begin{rk}
  The hypothesis that $v$ is a shortest element in $v\frakS(\lam)$ is probably not necessary but this is enough for our purpose.
\end{rk}

\subsection{Proof of Proposition \ref{prop:parvermadeformed}}

Consider the canonical embedding $r:X^\dag_w\ra Y^\dag_w$. We claim that the adjunction map yields a surjective morphism
\begin{equation}\label{eq:7}
r_!r^\ast(\scrB^{(n)})\ra\scrB^{(n)}.
\end{equation}
Indeed, an easy induction shows that it is enough to prove that $\gr^i(r_!r^\ast(\scrB^{(n)}))\ra\gr^i\scrB^{(n)}$ is surjective for each $i$. Since the functors $r_!$, $r_\ast$ are exact and $\gr^i\scrB^{(n)}=\scrB$, this follows from Lemma \ref{lem:vermaquotient}(a). Note that $r^\ast(\scrB^{(n)})=\scrA^{(n)}_w$. So the image of (\ref{eq:7}) by the exact functor $\Gamma\circ j_!$ is a surjective morphism
\begin{equation}\label{eq:8}
\Gamma(\scrA_{w!}^{(n)})\ra\Gamma(\scrB^{(n)}_!).
\end{equation}
By Lemma \ref{lem:vermadeformed} we have $\Gamma(\scrA_{w!}^{(n)})=N^{(n)}_\bfk$. Since the $\g_{R^{(n)}}$-module $\Gamma(\scrB^{(n)}_!)$ is $\frakq$-locally finite and $M^{(n)}_\bfk$ is the
largest quotient of $N^{(n)}_\bfk$ in $\calO_\bfk$, the morphism (\ref{eq:8}) induces a surjective morphism
\begin{equation*}
\varphi:M^{(n)}_\bfk\ra\Gamma(\scrB^{(n)}_!).
\end{equation*}
Further, by Proposition \ref{prop:localization}(b) the map
$$\varphi(\wp):M_\kappa\ra\Gamma(\scrB_!)$$
is an isomorphism. The same argument as in Step 1 of the proof of Lemma \ref{lem:vermadeformed} shows that for each $\mu\in\frakt^\ast$ the generalized weight spaces $(M^{(n)}_\bfk)_{\tilde{\mu}}$ and $\Gamma(\scrB^{(n)}_!)_{\tilde{\mu}}$ are isomorphic as $R^{(n)}$-modules. We deduce that $\varphi$ is an isomorphism by Nakayama's lemma. This proves the first equality. The proof for the second equality is similar. We consider the adjunction map
\begin{equation}
\scrB^{(n)}\ra r_\bullet r^\ast(\scrB^{(n)}).
\end{equation}
It is injective by Lemma \ref{lem:vermaquotient}(b) and the same arguments as above. So by applying the exact functor $\Gamma\circ j_\bullet$, we get an injective morphism
\begin{equation*}
\varphi':\ \Gamma(\scrB^{(n)}_\bullet)\ra \bfD M^{(n)}_\bfk.
\end{equation*}
Again, by using Proposition \ref{prop:localization}(b) and Nakayama's lemma, we prove that $\varphi'$ is an isomorphism.
\qed

\vskip1cm

\section*{Index of notation}

\begin{itemize}\setlength{\itemsep}{1mm}

\item[\textbf{1.1}:] $\calC\cap R\proj$, $R'\calC$.

\item[\textbf{1.2}:] $\calC^\Delta$, $\Delta$, $\nabla$, $D\spcheck$.

\item[\textbf{1.3}:] $R=\C[[s]]$, $\wp$, $K$, $M(\wp^i)$, $M_K$, $f(\wp^i)$, $f_K$, $\calC(\wp)$, $\calC_K$.

\item[\textbf{1.4}:] $J^iD(\wp)$.

\item[\textbf{1.5}:] $F(\wp)$.

\item[\textbf{2.1}:] $G_0,$ $B_0$, $T_0$, $\g_0$, $\frakb_0$,
$\frakt_0,$ $\g$, $\frakt$, $\bfone$, $\partial$, $c$, $\kappa=c+m$, $\fraka_R$,
$\calU_\kappa$, $M_\lam$, $\frakt_0^\ast$,
$\frakt^\ast$, $\delta$, $\omega_0$, $\ep_i$, $\pair{\bullet:\bullet},$
$||h||^2$, ${}_\kappa\!\frakt^\ast$, $a$, $z$, $\Pi$, $\Pi_0$,
$\Pi^+$, $\Pi^+_0$, $\al_i$, $\frakS$, $\frakS_0$, $w\cdot\lam$,
$\rho_0$, $\rho$, $\frakS(\lam)$, $l:\frakS\ra\N$.

\item[\textbf{2.2}:] $\frakq$, $\frakl$, $\Lam^+$, $M_\kappa(\lam)$,
$L_\kappa(\lam)$, $\bfc$, $\bfk$, $\calU_\bfk$,
$M_\bfk(\lam)$, $\lam_s$, ${}_\bfk\!\frakt^\ast$.

\item[\textbf{2.3}:] $\sigma$, $\bfD M_\bfk(\lam)$, $J^iM_\kappa(\lam)$.

\item[\textbf{2.4}:] $\calO_\bfk$, $\calO_\kappa$, ${}^r_\kappa\!\frakt^\ast$,
${}^r_\bfk\!\frakt^\ast$, ${}^r\!\calO_\kappa$, ${}^r\!\calO_\bfk$,
${}^r\!\Lam^+$, ${}^r\!P_\kappa(\lam)$,
$L_\bfk(\lam)$, ${}^r\!P_\bfk(\lam)$.

\item[\textbf{2.5}:] $\calP_n$, $\preceq$, $\unlhd$, $E$, $\calE_\kappa$, $\g'$, $\calE_\bfk$,
$P_\bfk(E)$, $P_\kappa(E)$.

\item[\textbf{2.6}:] $\calA_\bfk$, $\calA_\kappa$, $\Delta_\bfk$,
$\Delta_\kappa$.

\item[\textbf{2.7}:] $\bfD$.

\item[\textbf{2.9}:] $\scrH_\bfv$, $\bfS_\bfv$, $\calA_{\bfv}$, $W_\bfv(\lam)$, $\Delta_{\bfv}$.

\item[\textbf{2.10}:] $v=\exp(2\pi i/\kappa)$, $\bfv=\exp(2\pi i/\bfk)$, $J^iW_v(\lam)$.

\item[\textbf{2.11}:] $\bfH_{1/\kappa}$, $\calB_\kappa$, $B_\kappa(\lam)$,
$\frakE_\kappa$, $\bfH_{1/\bfk}$, $\calB_\bfk$, $\frakE_\bfk$, $B_\bfk(\lam)$.

\item[\textbf{3.1}:] $\scrO_Z$, $\bfO(Z)$, $f_\ast$, $f^\ast$, $f^!$, $\scrD_Z$, $\bfM(Z)$, $\Omega_Z$, $\scrD_{Y\ra Z}$, $\scrM^{\scrO}$, $\bfM(Z,Z')$, $i^\ast$, $i_\bullet$, $i^!$.

\item[\textbf{3.2}:] $\bfM_h(Z)$, $\bbD$, $i_!$, $i_{!\bullet}$.

\item[\textbf{3.3}:] $\bfM^T(Z)$, $\bfM^T(X,Z)$.

\item[\textbf{3.4}:] $\scrM^\dag$, $\scrD^\dag_Z$, $\delta_r$, $\frakm_\lam$, $\bfM^\lam(Z)$, $\bfM^{\tilde\lam}(Z)$, $\bfM^\lam(X,Z)$, $\bfM^{\tilde\lam}(X,Z)$, $\bfM(\scrD^\dag_Z)$.

\item[\textbf{3.5}:] $\bfM^T_h(Z^\dag)$, $\bfM^{\lam}_h(Z)$, $\bfM^{\tilde\lam}_h(Z)$, $\bbD'$, $\scrL^\lam_Z$, $\Theta^\lam$, $\bbD=\bbD^\lam$, $\scrD^\lam_Z$.

\item[\textbf{3.6}:] $\tilim\calC_\al$, $\tplim\calC_\al$.

\item[\textbf{3.7}:] $X=\ilim X_\al$, $\bfO(X)$, $\Gamma(X,\scrM)$ (for $\scrM\in\bfO(X)$), $\hat\bfO(X)$, $\scrO_X$, $-\otimes_{\scrO_X}\scrF$, $\Gamma(X,\scrF)$ (for $\scrF\in\hat\bfO(X)$).

\item[\textbf{3.8}:] $\bfM(X)$ (with $X$ an ind-scheme), $\scrM^{\scrO}$, $\Gamma(X,\scrM)$ (for $\scrM\in\bfM(X)$), $\bfM^T(X)$, $\bfM^\lam(X)$, $\bfM^{\tilde\lam}(X)$, $\Gamma(\scrM)$ (for $\scrM\in\bfM^{\tilde\lam}(X)$), $\bfM^\lam_h(X)$, $\bfM^{\tilde\lam}_h(X)$, $i_!$, $i_\bullet$.

\item[\textbf{3.9}:] $\scrD_X$ (with $X$ a formally smooth ind-scheme).

\item[\textbf{4.1}:] $G$, $B$, $Q$, $T$, $N$, $\g$, $\frakb$, $\frakn$.

\item[\textbf{4.2}:] $X=G/B$, $X^\dag=G/N$, $\pi:X^\dag\ra X$, $X_w$, $\dot{w}$, $\olX_w$.

\item[\textbf{4.3}:] $\delta_l:\calU(\g)\ra\Gamma(X^\dag,\scrD_{X^\dag})$, $\bfM(\g)$, $\Gamma$

\item[\textbf{4.4}:] $M_{\tilde\lam}$, $\ch(M)$, $\tilde{\calO}$, $\tilde{\calO}_\kappa$, $N_\kappa(\lam)$, $L_\kappa(\lam)$, $\Lam$, $\scrA^\lam_w$, $i_w$, $\scrA^\lam_{w!}$, $\scrA^\lam_{w!\bullet}$, $\scrA^\lam_{w\bullet}$.

\item[\textbf{4.5}:] $\QS$, $w_0$, $Y_w$, $\olY_w$, $j_w$, $\scrB^\lam_w$, $\scrB^\lam_{w!}$, $\scrB^\lam_{w!\bullet}$, $\scrB^\lam_{w\bullet}$, $r:X^\dag_w\ra Y^\dag_w$.

\item[\textbf{5.1}:] $\calC_R$, $\mu_M$, $for$, $F_R$.

\item[\textbf{5.2}:] $Q'$, $f_w$.

\item[\textbf{5.3}:] $j$, $f$, $\scrB$, $\scrB_!$, $\scrI^{(n)}$, $x^s$, $f^s$, $\scrB^{(n)}$, $\scrB^{(n)}_!$, $\scrB^{(n)}_{!\bullet}$, $\scrB^{(n)}_\bullet$.

\item[\textbf{5.4}:] $M_\kappa$, $M_\bfk$, $M^{(n)}_\bfk$, $\bfD M^{(n)}_\bfk$.

\item[\textbf{5.5}:] $\psi(a,n)$, $\pi^a(\scrB)$, $J^i(\scrB_!)$.

\item[\textbf{6.1}:] $\MHM(Z)$, $W^\bullet\scrM$, $(k)$, $\Perv(Z)$, $\varrho$.

\item[\textbf{6.2}:] $\bfM^\lam_0(\olX_v)$, $\DR$, $\MHM_0(\olX_v)$, $\eta$,
$\tilde\scrA^\lam_w$, $\tilde\scrA^\lam_{w!}$,
$\tilde\scrA^\lam_{w!\bullet}$, $\tilde\scrB^\lam_w$,
$\tilde\scrB^\lam_{w!}$, $\tilde\scrB^\lam_{w!\bullet}$, $q$,
$\scrH_q(\frakS)$, $T_w$, $\Psi$, $P_{x,w}$,

\item[\textbf{A.1}:] $\Pi^-$, $\frakn(\Upsilon)$, $\frakn^-(\Upsilon)$, $\Pi^-_l$, $\frakn^-_l$, $N^-$, $N^-_l$, $\frakX$, $\frakX^\dag$, $\frakX^w$, $\frakX^w_l$, $\frakX^{w\dag}_l$, $p_{l_1l_2}$, $p_l$.

\item[\textbf{A.2}:] $\bfH^{\tilde\lam}(\frakX^y,\olX_w)$, $\bfH^{\tilde\lam}(\olX_w)$, $\bfH^{\tilde\lam}(X)$.

\item[\textbf{A.3}:] $\hat\Gamma(\scrM)$, $\overline{\Gamma}(\scrM)$.

\item[\textbf{A.5}:] $\tilde{\calO}_{\kappa,\lam}$, $\pr_\lam$, $V(\nu)$, $\theta^\nu$, $\scrL^\lam$, $\Theta^{\lam-\mu}$.

\item[\textbf{B.1}:] $r_x$, $\scrA_v^{(n)}$, $\scrA^{(n)}_{v!}$, $\scrA^{(n)}_{v!\bullet}$, $\scrA^{(n)}_{v\bullet}$, $N_\bfk(\mu)$, $R_{\mu+s\omega_0}$, $\bfD N_\bfk(\mu)$, $N^{(n)}_\bfk(\mu)$, $\bfD N^{(n)}_\bfk(\mu)$, $s^iM$, $F^\bullet M$, $\gr M$, $\dch(M)$.
\end{itemize}

\end{document}